%% file: DescentGroupAct_revised2.tex
\documentclass{amsart}
\usepackage{verbatim}
\usepackage{hyperref}
\hypersetup{
 hidelinks,colorlinks=false,breaklinks=true,
 pdfkeywords={K-theory, Descent, Redshift},
 pdfauthor={Dustin Clausen, Akhil Mathew, Niko Naumann, Justin Noel},
}
\usepackage{xy}
\usepackage{cleveref}
 \input xy
 \xyoption{all}
 
\crefformat{enumi}{#2(#1)#3}
\Crefformat{enumi}{#2(#1)#3}
\numberwithin{equation}{section}
\newcommand{\TC}{\mathrm{TC}}

\newcommand{\rep}{\mathrm{Rep}}

\newcommand{\tr}{\mathrm{tr}}
\newcommand{\BG}{\mathrm{BG}}
\renewcommand{\hom}{\mathrm{Hom}}

\newcommand{\bor}{\mathrm{Bor}}
\newcommand{\sOFG}{\mathcal{O}_{\sF}(G)}
\newcommand{\pic}{\mathrm{Pic}}
\newcommand{\burng}{\mathrm{Burn}^{\mathrm{eff}}_G}
\newcommand{\catst}{\mathrm{Cat}^{\mathrm{perf}}_{\infty}}
\newcommand{\cats}[1]{\mathrm{Cat}^{\mathrm{perf}}_{#1, \infty}}
 \newcommand{\prl}{\mathcal{P}r_{\mathrm{st}}^L}
\newcommand{\prr}{\mathcal{P}r_{\mathrm{st}}^R}

\newcommand{\cobor}{\mathrm{coBor}}

\newcommand{\mot}{\mathrm{Mot}}
\renewcommand{\hom}{\mathrm{Hom}}

\newcommand{\mack}{\mathrm{Mack}}
\renewcommand{\sp}{\mathrm{Sp}}
\newcommand{\triv}{\mathscr{T}}
\newcommand{\burn}{\mathrm{Burn}^{\mathrm{eff}}}

\newcommand{\sF}{\mathscr{F}}
\newcommand{\sG}{\mathscr{G}}

\newcommand{\msk}{\mathrm{Mack}}
\newcommand{\sw}{\mathrm{Rep}}
\newcommand{\perf}{\mathrm{Perf}}
\usepackage{amssymb}
\usepackage{mathrsfs}
\usepackage{geometry} 
\usepackage[toc,page]{appendix} 
\usepackage[
disable=true, 
color=orange!80, 
bordercolor=black,
textwidth=.8in,
textsize=small]
{todonotes}

\newcommand{\sOG}{\mathscr{O}(G)}

\newcommand{\fun}{\mathrm{Fun}}
\newcommand{\Fun}{\mathrm{Fun}}

\newcommand{\Res}{\mathrm{Res}}

\usepackage{url} 

\newtheorem{theorem}{Theorem}[section]
\newtheorem{theorema}{Theorem}

\newtheorem{thm}[theorem]{Theorem}
\newtheorem{lem}[theorem]{Lemma}
\newtheorem{lemma}[theorem]{Lemma}
\newtheorem{corollary}[theorem]{Corollary}
\newtheorem*{corollarystar}{Corollary}

\newtheorem{proposition}[theorem]{Proposition}
\newtheorem{prop}[theorem]{Proposition}
\newtheorem{cor}[theorem]{Corollary}

\theoremstyle{definition}

\newtheorem{conj}[theorem]{Conjecture}
\newtheorem{remark}[theorem]{Remark}
\newtheorem{definition}[theorem]{Definition}
\newtheorem{example}[theorem]{Example}
\newtheorem{construction}[theorem]{Construction}

\newtheorem{cons}[theorem]{Construction}

\newcommand{\GSpec}{\mathrm{Sp}_G}

\newcommand{\e}[1]{\mathbb{E}_{#1}}

\renewcommand{\theenumi}{\arabic{enumi}}

\makeatletter
\renewcommand{\p@enumii}{\theenumi.}
\makeatother
\newcommand{\CAlg}{\mathrm{CAlg}}

\begin{document}
	\title{Descent and vanishing in chromatic algebraic $K$-theory via group actions}
	\author{Dustin Clausen}
	\author{Akhil Mathew}
	\author{Niko Naumann}

	\author{Justin Noel}

	\date{\today}

	\begin{abstract}
	We prove some $K$-theoretic descent results for finite group actions on 
	stable $\infty$-categories, including the $p$-group case of the  Galois descent
	conjecture of 
	Ausoni--Rognes.  We also prove vanishing results in accordance with
	Ausoni--Rognes's redshift philosophy: in particular, we show that if $R$ is an $\mathbb{E}_\infty$-ring spectrum with $L_{T(n)}R=0$, then $L_{T(n+1)}K(R)=0$.  Our key observation is that descent and vanishing are logically interrelated, permitting to establish them simultaneously by induction on the height.

\end{abstract}

	\maketitle

\setcounter{tocdepth}{1}
\tableofcontents

\section{Introduction}

In this paper, we prove some results
concerning the algebraic $K$-theory of ring spectra and stable
$\infty$-categories after  $T(n)$-localization. 
Throughout this paper, 
 our telescopes $T(n)$ are taken at a fixed implicit prime $p$ and height
$n\geq 0$; we adopt the convention $T(0)=\mathbb{S}[1/p]$.   Our starting
point is the following two 
results concerning classical commutative rings $R$:

\begin{theorem} [\cite{Mitchell90}] 
\label{Mithm}
For $n \geq 2$, we have $L_{T(n)} K(R) = 0$.\end{theorem} 

\begin{theorem}[\cite{Tho85}, \cite{TT90},
\cite{CMNN}] 
\label{galoisdescthm}
For $G$ a finite group and $R \to R'$ a $G$-Galois extension, the natural comparison map 
$L_{T(1)} K(R) \to (L_{T(1)}K(R'))^{hG}$ is an equivalence. 
\end{theorem} 

Thus, the $K$-theory of an ordinary commutative ring has no chromatic
information beyond height one, and the localization to height one is
well-behaved in its descent properties. In fact, $T(1)$-local $K$-theory is even better-behaved than suggested by Theorem \ref{galoisdescthm}: under mild finiteness hypotheses, the Galois descent can be upgraded to an \'etale hyperdescent result, which leads to a  descent
spectral sequence from \'etale cohomology to $T(1)$-local $K$-theory as produced by \cite{Tho85, TT90}. 
Furthermore, one knows that under such conditions, the map $K(R; \mathbb{Z}_p)
\to L_{T(1)} K(R)$ from $p$-adic $K$-theory to its $T(1)$-localization is an
equivalence in high enough degrees, i.e., one has the Lichtenbaum--Quillen
conjecture, thanks to the work of Voevodsky--Rost, cf.~\cite{RO06, CM19} for
accounts.  However, we will not touch on these more advanced aspects in this paper.

Moving from ordinary rings to more general ring spectra, Ausoni--Rognes suggested that the above two theorems should fit into a broader ``redshift'' philosophy in 
algebraic $K$-theory, \cite{AR02, AR}. 
For an $\mathbb{E}_1$-ring spectrum $R$, one expects
that taking algebraic $K$-theory increases the ``chromatic complexity'' of $R$ by one. 
In the setting of \Cref{Mithm}, the Eilenberg--MacLane spectrum $HR$ has no
chromatic information at heights $\geq 1$, while the result states that $K(R) =
K(HR)$ has no chromatic information at heights $\geq 2$; furthermore,
\Cref{galoisdescthm} and its refinement to hyperdescent control the height one information very precisely. 

For $\mathbb{E}_\infty$-rings $R$, there is a particularly well-behaved notion
of chromatic complexity, thanks to a theorem of Hahn \cite{Hahn16}: if
$L_{T(n)}R=0$, then $L_{T(m)}R=0$ also for all $m>n$.  If $R$ is an
$\mathbb{E}_\infty$-ring, then so is $K(R)$, and in this setting one possible
expression of the redshift philosophy would be that $L_{T(n)}R=0 \Leftrightarrow
L_{T(n+1)}K(R)=0$.
Here we prove half of this statement. \begin{theorema} 
\label{TateMitchell}
Let $R$ be an $\mathbb{E}_\infty$-ring and $n\geq 0$.  If $L_{T(n)}R=0$, then $L_{T(n+1)}K(R)=0$.
\end{theorema} 

Recent work of  Burklund--Schlank--Yuan \cite[Th.~9.11]{BSY22} and Yuan \cite{Yuan21}
proves the converse of \Cref{TateMitchell}: if $R$ is a $p$-local
$\mathbb{E}_\infty$-ring with $L_{T(n)} R \neq 0$, then
$L_{T(n+1)} K(R) \neq 0$. 
Many special cases of \Cref{TateMitchell}  were previously known. 
In particular, in important specific cases, much more precise (Lichtenbaum--Quillen) statements about $K(R)$ have been proved, as in \cite{HW20,
AKACHR22, HRW22,
Ausoni10, AR02}.

\Cref{TateMitchell} generalizes Mitchell's vanishing \Cref{Mithm}.  We note that there
is a more general statement which applies also to $\mathbb{E}_1$-rings $A$: if
both $L_{T(n)}A=0$ and $L_{T(n+1)}A=0$, then $L_{T(n+1)}K(A)=0$; see
\Cref{purityTnlocal}, which is also explored in \cite{LMMT20}.  

We also have an analog of Thomason's descent \Cref{galoisdescthm}. For the
statement, we
need to assume $T(n)$-local vanishing of the $C_p$-Tate construction $R^{tC_p}$
(taken with respect to the trivial action); this assumption is satisfied if $R$
is a discrete ring and $n = 1$, i.e., the setting of \Cref{galoisdescthm}.
In addition, we need to assume the finite group $G$ is a $p$-group, where $p$ is
the (throughout fixed) prime at which chromatic localizations are taken. 

\begin{theorema}\label{ASthmintro}
Let $R$ be an $\mathbb{E}_\infty$-ring and $n\geq 0$.  Suppose
$L_{T(n)}(R^{tC_p})=0$.  Then for $\mathcal{C}$ any $R$-linear
idempotent-complete stable
$\infty$-category equipped with an $R$-linear action of a finite $p$-group $G$, the homotopy fixed point
comparison map for $T(n+1)$-local $K$-theory is an equivalence:
\begin{equation} \label{compmapintrostatement}
L_{T(n+1)}K(\mathcal{C}^{hG})\overset{\sim}{\longrightarrow}(L_{T(n+1)}K(\mathcal{C}))^{hG}.\end{equation}
\end{theorema} 

If $R\rightarrow R'$ is a $G$-Galois extension of commutative rings, then by
Galois descent we have
$\operatorname{Perf}(R)\overset{\sim}{\rightarrow}\operatorname{Perf}(R')^{hG}$;
thus, when $n=0$,  \Cref{ASthmintro} recovers the $p$-group case of
\Cref{galoisdescthm}.  But in fact the case of general $G$ in
\Cref{galoisdescthm} reduces to the $p$-group case by a simple transfer
argument, as already pointed out and exploited by Thomason; in particular,
\Cref{ASthmintro} implies \Cref{galoisdescthm}. 

However,
\Cref{ASthmintro} does not hold for an arbitrary finite
group $G$, essentially because the $G$-action is allowed to be arbitrary.  In
fact, for the trivial action of $G$ on $\operatorname{Perf}(\mathbb{C})$ and
$n=0$, one can calculate both sides of \eqref{compmapintrostatement} using Suslin's equivalence \cite{Suslin84}
between topological and algebraic $K$-theory. 
One obtains that the source is the $p$-completed $G$-equivariant topological $K$-theory
of a point while the target is  $KU_{\hat{p}}^{BG}$. 
For $G$ of order prime-to-$p$
the result is evidently false because $KU_{\hat{p}}^{BG} = KU_{\hat{p}}$, 
while for $G$ a $p$-group, \Cref{ASthmintro} amounts to the
$p$-complete Atiyah--Segal completion theorem.  Nonetheless, there is a generalization of \Cref{ASthmintro} to arbitrary finite groups which shows that the descent question for arbitrary $G$ reduces to that for cyclic subgroups of order prime to $p$; see \Cref{ASthmwithcyclic}.

We remark that these theorems also hold with $K$-theory replaced by an arbitrary additive invariant, and one also has ``co-descent" or ``assembly map" equivalences dual to the descent statements of \Cref{ASthmintro}; see \Cref{generalreduction} for more details.

Let us now give the basic example of these results. 
Throughout this paper, we will use the notation 
$L_{n}^{p, f} = L_{T(0) \oplus \dots \oplus T(n)}$, following \cite{LMMT20}; in
particular, we have the $L_n^{p,f}$-local sphere $L_n^{p,f} \mathbb{S}$. 
An $L_n^{p, f}$-local stable $\infty$-category
is one where the mapping spectra are $L_{n}^{p, f}$-local, or equivalently one
which is $L_n^{p, f} \mathbb{S}$-linear. 
By Kuhn's ``blueshift" theorem \cite{Kuh04},  if a spectrum $X$ is $L_n^{p,f}$-local then
$X^{tC_p}$ is $L_{n-1}^{p,f}$-local.  Thus, 
from \Cref{TateMitchell} and \Cref{ASthmintro}
we deduce the following:

\begin{theorema}\label{chromatic}
Let $n\geq 0$, and let $\mathcal{C}$ be an $L_n^{p,f}$-local
idempotent-complete stable $\infty$-category.  Then $L_{T(m)}K(\mathcal{C})=0$ for all $m\geq n+2$, and for any finite $p$-group $G$ acting on $\mathcal{C}$ we have
$$L_{T(n+1)}K(\mathcal{C}^{hG})\overset{\sim}{\rightarrow}(L_{T(n+1)}K(\mathcal{C}))^{hG}.$$
\end{theorema}

In fact, for the proofs of \Cref{TateMitchell} and \Cref{ASthmintro} we proceed
by first proving this special case, \Cref{chromatic}.  Then we combine with a
recent result of Land--Mathew--Meier--Tamme \cite{LMMT20} to the effect that
$L_{T(n)}K(R)\overset{\sim}{\rightarrow} L_{T(n)}K(L_n^{p,f}R)$ (for $n\geq 1$)
which lets us deduce the general case.  (Actually, we also use the result of
\cite{LMMT20} in the proof of \Cref{chromatic}, but in a more indirect way.)

An interesting aspect of our arguments is that we show a logical connection
between the vanishing and the descent theorems.  This is expressed in the
following result, from which we deduce 
all of the above theorems. 

\begin{theorem}[Inductive vanishing,
\Cref{Cpgeovanishlemma}]\label{inductivevanishingthm}
Let $R$ be an $\mathbb{E}_\infty$-ring spectrum and $n\geq 1$.  Then for the following conditions, we have the implications (A) $\Rightarrow$ (B) $\Rightarrow$ (C):
\begin{enumerate}
\item[(A)] $L_{T(n)}R=0$ and $L_{T(n)}K(R^{tC_p})=0$.
\item[(B)] For any action of a finite $p$-group $G$ on an $R$-linear
idempotent-complete stable $\infty$-category $\mathcal{C}$, the comparison map
$$L_{T(n)}K(\mathcal{C}^{hG})\overset{\sim}{\rightarrow} (L_{T(n)}K(\mathcal{C}))^{hG}$$
is an equivalence.
\item[(C)] $L_{T(i)}K(R)=0$ for $i\geq n+1$.
\end{enumerate}
\end{theorem}

\Cref{inductivevanishingthm} allows an inductive approach to simultaneously
proving vanishing and descent statements.  In fact, \Cref{chromatic} follows
immediately from it by inductively taking $R=L_n^{p,f}\mathbb{S}$ (and
replacing $n$ with $n+1$), via Kuhn's blueshift theorem.


Concerning the general descent result \Cref{ASthmintro}, we have already
mentioned that it recovers Galois descent for $K(1)$-local $K$-theory and the $p$-completed $p$-group case of the Atiyah-Segal completion theorem.  We also use it to obtain the following $p$-group case of a conjecture of Ausoni--Rognes
\cite[Conj.~4.2]{AR}:

\begin{corollarystar}[\Cref{ARpgroup}] 
Let $A \to B$ be a $T(n)$-local $G$-Galois extension of
$\mathbb{E}_\infty$-rings, in the sense of Rognes \cite{Rognes08}, for $G$ a finite
$p$-group. Then the maps 
$L_{T(n+1)} K(A) \to L_{T(n+1)} ( K(B)^{hG}) \to (L_{T(n+1)} K(B))^{hG}$ are
equivalences. 
\end{corollarystar}

Besides the above thread of results, we also prove some other descent results of a slightly different nature with different techniques.  Like the results of our previous paper \cite{CMNN}, these work uniformly for all chromatic
heights, including height zero, and do not assume $G$ to be a $p$-group; but on the other hand they make more restrictive assumptions on the action of the group $G$.

\begin{theorema} 
Let $\mathcal{C}$ be a monoidal, idempotent-complete stable $\infty$-category with biexact
tensor product equipped with a (monoidal)
action of a finite group $G$.  
Let $\mathrm{tr} \colon \mathcal{C} \to \mathcal{C}^{hG}$ denote the
$G$-equivariant biadjoint to the
forgetful functor $\mathcal{C}^{hG} \to \mathcal{C}$. 
Suppose the $G$-equivariant object $\mathrm{tr}(\mathbf{1}) \in \fun(BG,
\mathcal{C}^{hG})$ has  class in $K_0( \fun(BG, \mathcal{C}^{hG}))$ equal
to that of the induced $G$-object $ \bigoplus_G \mathbf{1}_{\mathcal{C}^{hG}}
\in \fun(BG, \mathcal{C}^{hG})$. Then the 
comparison map
$$K(\mathcal{C}^{hG})\rightarrow K(\mathcal{C})^{hG}$$
 induces an equivalence after $T(n)$-localization
for any $n\geq 0$ and for any prime $p$.
\label{normalbasisthm}
\end{theorema} 

\Cref{normalbasisthm} states that a type of normal basis property (e.g., for a
Galois extension of fields, the condition on $\mathrm{tr}(\mathbf{1})$ follows from the normal basis theorem)
for the $G$-equivariant object $\mathrm{tr}( \mathbf{1}) \in \mathcal{C}^{hG}$ 
implies that the homotopy fixed point comparison map for $K$-theory is an equivalence after $T(n)$-localization. 
The argument  is based on the vanishing 
\cite{Kuh04}
of Tate spectra in telescopic homotopy
theory. 
In fact, \Cref{normalbasisthm} yields another proof of \Cref{galoisdescthm}
avoiding the use of $\mathbb{E}_\infty$-structures, see \Cref{rem:alternate_argument}.

Next, we use \cite{MNN_nilpotence} to prove a third descent result (\Cref{swandescintro} below),
which applies in more general situations, albeit
with a weaker conclusion. For this, we first formulate a generalization of 
the homotopy fixed point comparison maps with respect to a family of subgroups of $G$. 

\begin{construction}[Comparison maps for families of subgroups] 
Let $\mathcal{C}$ be an idempotent-complete stable $\infty$-category with an action of a finite
group $G$. 
Let $\sF$ be a family of subgroups of $G$, i.e., $\sF$ is nonempty and closed under 
subconjugation, and  
let $\sOFG$ be the category of $G$-sets of the form $G/H, H \in \sF$. 
We obtain a comparison map 
\begin{equation}  \label{compmap2} K(\mathcal{C}^{hG}) \to \varprojlim_{G/H \in \sOFG^{op}}
K(\mathcal{C}^{hH}); \end{equation}
this generalizes the homotopy fixed point comparison map, which is the case where $\sF =
\left\{(1)\right\}$. 
\end{construction}

The map \eqref{compmap2} is dual to the type of assembly maps which (for
infinite groups) 
are the subject of the Farrell--Jones conjecture and its variants. 
In the rest of this paper, we will introduce 
a basic condition on an $\mathbb{E}_\infty$-ring that guarantees the maps
\eqref{compmap2} are equivalences after telescopic localization, and 
implies a bound on the chromatic complexity of the algebraic $K$-theory. 

\begin{definition}[Swan $K$-theory, Malkiewich \cite{Mal17}] 
Let $R$ be an $\mathbb{E}_\infty$-ring spectrum, and let $G$ be a
(discrete) group. 
We define the ring $\rep(G, R)$, called the \emph{Swan $K$-theory of $R$}, via
\[ \rep(G, R)  = K_0 ( \fun(BG, \perf(R))).  \]
For $G$ finite, 
using induction and restriction functors, one makes $\rep(-, R)$ into a Green
functor, cf.~\Cref{def:swan_green}.  \end{definition}

The $\infty$-category $\fun(BG, \perf(R))$ is an analog of the category of
complex representations of the finite group $G$; 
studying this in analogy with complex or modular representation theory for $R =
KU$ has been proposed by Treumann \cite{Treumann15}. 
In this analogy, the ring $\rep(G, R)$ is an analog of the classical
representation ring of $G$ (to which it reduces when $R = H\mathbb{C}$). 
In general, the calculation of the rings $\rep(G, R)$ seems to be an
interesting problem (e.g., for $R = KU$), although we know very few examples. 

\begin{definition}[$R$-based Swan induction] 
Fix a finite group $G$ and an $\mathbb{E}_\infty$-ring $R$.
If the Green functor $\rep(-, R) \otimes \mathbb{Q}$ (for subgroups of $G$) is induced
from a family $\sF$ of subgroups of $G$, then we say that \emph{$R$-based Swan induction}
holds for the family $\sF$. 
\end{definition}

In \cite{Swa60}, Swan shows that $H\mathbb{Z}$-based Swan induction holds for the family
of cyclic groups for any finite group; see also \cite{Swa70} for a detailed
treatment of Swan $K$-theory for a discrete ring. For $H\mathbb{C}$, this is Artin's
classical induction theorem for the representation ring. 
Via
some explicit geometric arguments, we prove the following instances of Swan
induction for $\mathbb{E}_\infty$-ring spectra. 

\begin{theorema}\label{Ex:Swaninduction}
If $R$ is an $\mathbb{E}_\infty$-ring, then for every finite group, $R$-based Swan induction holds for: 
\begin{enumerate}
\item the family of abelian subgroups if $R$ 
admits an $\mathbb{E}_1$-map from $MU$. 
\item\label{Ex:Swaninduction_KU} the family of abelian subgroups of rank $\leq 2$ if $R=KU$. 
\item\label{Ex:Swaninduction_Morava} the family of  abelian subgroups of $p$-rank $\leq n+1$ and $\ell$-rank $\leq 1$ for primes $\ell \neq p$ if $R=E_n$ is Morava $E$-thory of height $n$ at the prime $p=2$. \end{enumerate}
\end{theorema} 

This statement subsumes \Cref{cxorientswan}, \Cref{artinindKU}, and parts of \Cref{thm:2-primary}.
Informally, 
\Cref{Ex:Swaninduction_KU} states that while a complex representation 
of a finite group $G$ is determined by its character (i.e., its restriction to
cyclic subgroups), the class in $\rep(G, KU)$ of a ``representation'' of $G$ with $KU$-coefficients should be
determined by a sort of ``2-character,'' defined on pairs of commuting elements
of $G$;   moreover, there should be generalizations to higher heights.
We conjecture that \Cref{Ex:Swaninduction_Morava} should be true for odd
primes too (\Cref{conj:Morava_Swan}).

We show that the Swan induction condition guarantees
that the maps \eqref{compmap2} become equivalences after telescopic
localization, for every $R$-linear $\infty$-category.  
This relies on similar techniques as in \cite{CMNN}.

\begin{theorema}[Descent via Swan induction, see \Cref{swaninddesc}] 
\label{swandescintro}
Let $R$ be an $\mathbb{E}_\infty$-ring spectrum, $G$ a finite group, and $\sF$
a family of subgroups of $G$. 
Suppose that 
$R$-based Swan induction holds for the family 
$\sF$. 
Then, for every $R$-linear idempotent-complete stable $\infty$-category $\mathcal{C} $ with an action of
$G$, the natural map \eqref{compmap2}, namely 
\begin{equation*} K(\mathcal{C}^{hG}) \to \varprojlim_{G/H \in \sOFG^{op}}
K(\mathcal{C}^{hH}); \end{equation*}
becomes an equivalence after $T(n)$-localization, for any $n$ and any implicit prime $p$.
Furthermore, the limit in \eqref{compmap2} commutes with $L_{T(n)}$.
\end{theorema} 

In fact, 
\Cref{swandescintro} can be combined with \Cref{ASthmintro}, yielding the
following result (\Cref{ASthmwithcyclic}): if $L_{T(n)} (R^{tC_p}) = 0$, then the
comparison map
\eqref{compmap2} becomes an equivalence after $T(n+1)$-localization, for $\sF$
the family of cyclic subgroups of order prime to $p$.

Our final main result, which is inspired by the character theory of
Hopkins--Kuhn--Ravenel \cite{HKR00},  is that a certain case of Swan induction implies the
vanishing of the $T(i)$-localizations of algebraic $K$-theory for large $i$.

\begin{theorema} \label{swanvanishintro}
Let $R$ be an $\mathbb{E}_{\infty}$-ring, $p$ a prime, and $n>0$.
Suppose that $R$-based Swan induction holds for the family of
proper subgroups of $C_p^{\times n}$. 
Then $L_{T(i)} K(R) = 0$ for $i \geq n$ at the implicit prime $p$. 
\end{theorema} 

As a consequence, we obtain a new proof of Mitchell's theorem (\Cref{Mithm}),
and we recover several chromatic bounds, e.g., that if $p=2 $ or $n = 1$, then 
we have $L_{T(i)} K(E_n) = 0$ for $i \geq n + 2$. These bounds are special
cases
of \Cref{TateMitchell} above, although the method is different and
could be useful in other settings as well; for instance, in
\Cref{thm:2-primary} we prove 2-primary Swan induction results for $MO\left
\langle n\right\rangle$ which therefore implies bounds on the chromatic
complexity of $K( MO\left \langle n\right\rangle)$.

\subsection*{Conventions}

We let $\mathcal{S}$ denote the $\infty$-category of anima, $\sp$ denote the
$\infty$-category of spectra, $\mathcal{S}_G$ the $\infty$-category of
$G$-anima (i.e., genuine $G$-spaces), and $\GSpec$ the
$\infty$-category of (genuine) $G$-spectra. 
We denote by $\mathbb{S}$ the unit of either of these (i.e., the sphere
spectrum). We write $\mathbb{D}X$ for the Spanier--Whitehead dual of $X$. 

We let $L_n^{p, f}$ denote the finite localization \cite{Miller92} on $\sp$ away from a 
finite type $n+1$ spectrum (at the implicit prime $p$). In particular, for $n =
0$, we have $L_0^{p, f}(X) = X[1/p]$. This convention follows \cite{LMMT20};
for $p$-local spectra, it agrees with what is usually denoted $L_n^{f}$. 
Equivalently, if $T(i)$ denotes the telescope of a $v_i$-self map of a finite
type $i$ complex (so by convention $T(0) $ can be taken to be
$\mathbb{S}[1/p]$), then $L_{n}^{p, f} = L_{T(0) \oplus \dots \oplus T(n)}$. 

We write $K$ for connective algebraic $K$-theory. 
Most of our results hold only after telescopic localization, after which
there is no difference between connective and nonconnective $K$-theory, and we
will anyway state them in the generality of additive invariants. 

We write $\catst$ for the $\infty$-category of small, idempotent-complete stable
$\infty$-categories and exact functors between them, cf.~\cite{BGT13}. 
More generally, given an $\mathbb{E}_\infty$-ring $R$, we write $\cats{R}$ for the
$\infty$-category of 
small, idempotent-complete $R$-linear stable
$\infty$-categories and $R$-linear functors between them, so $\cats{R}$ is
$\perf(R)$-modules in $\catst$. Compare \cite{HSS17} for a treatment. 

An $\infty$-category $\mathcal{C}$ is called \emph{preadditive} (also called
\emph{semiadditive} in the literature) if it is
pointed, admits finite coproducts, and finite coproducts are (canonically)
identified with finite products, see \cite[Sec.~2]{GGN15}. 
Given a preadditive $\infty$-category $\mathcal{C}$, 
we say that $\mathcal{C}$ is \emph{additive} if the $\mathbb{E}_\infty$-anima
$\hom_{\mathcal{C}}(X, Y)$ for $X, Y \in \mathcal{C}$ are grouplike. 
\subsection*{Acknowledgments}

We thank Clark Barwick, Jesper Grodal, Jeremy Hahn, Rune Haugseng, Michael Hopkins, Marc
Hoyois, Markus Land, Jacob Lurie, Lennart Meier,  John Rognes, Georg Tamme, and Craig
Westerland for helpful discussions. 
We are very grateful to the referee for many comments on an earlier version of
this paper. 
The second author thanks the University of Copenhagen and the University of
Regensburg for hospitality. 
The second author was supported by the NSF Graduate Fellowship under grant
DGE-114415,  
a Clay Research Fellowship, and the National Science Foundation under Grant
No.~ 2152311. 
The third and fourth authors were partially supported by the SFB 1085, Regensburg.

\section{Mackey functors and equivariant algebraic $K$-theory}

In this section, we review the setup of equivariant algebraic 
$K$-theory which plays an integral role in our approach to the present descent
theorems.
The use of equivariant algebraic $K$-theory refines the use of 
the transfer, which is central to all such descent
results, going back to \cite{Tho85}.\footnote{A toy example of this argument is
the Galois descent for rationalized algebraic $K$-theory,
cf.~\cite[Th.~2.15]{Tho85}.} 
In studying the comparison map $K(\mathcal{C}^{hG})\rightarrow K(\mathcal{C})^{hG}$, one
observes that 
there is also a map
\begin{equation}\label{eq:coind}
K(\mathcal{C})_{hG} \to K(\mathcal{C}^{hG}),
\end{equation} 
arising from the $G$-equivariant functor $ \mathcal{C} \to
\mathcal{C}^{hG}$ biadjoint to the forgetful functor $\mathcal{C}^{hG} \to
\mathcal{C}$, such that the composition with the comparison map is the norm
map $K(\mathcal{C})_{hG} \to K(\mathcal{C})^{hG}. $ 
In the case of $\mathcal{C} = \mathrm{Perf}(R')$ for a $G$-Galois extension $R \to R'$ of commutative rings, 
then Galois descent
gives $\mathrm{Perf}(R) \simeq \mathrm{Perf}(R')^{hG}$ and 
\eqref{eq:coind} is the map 
\[ K(R')_{hG} \to K(R)  \]
which arises from restriction of scalars from $R'$-modules to $R$-modules. 
These types of transfer maps and their functorialities can be encoded using the language
of (genuine) $G$-spectra, and some of the techniques for proving descent results can be
expressed using the language of $\sF$-nilpotence \cite{MNN17, MNN19}.

Several authors have considered 
the setup of equivariant algebraic $K$-theory, including Merling \cite{Me17},
 Barwick \cite{Bar17}, 
 Malkiewich--Merling \cite{MM19}, 
 and Barwick--Glasman--Shah
\cite{BGS20}. We will follow the setup of \cite{Bar17, BGS20}, but will try to keep 
the exposition mostly self-contained. 
In particular, we will use the theory of spectral Mackey functors, 
which is equivalent to the theory of $G$-spectra by work of 
Guillou--May \cite{GM11}
and Nardin \cite{Na16}; we will also give another proof of this equivalence in
the appendix. 

\begin{definition}[{The effective Burnside $\infty$-category, 
\cite[Sec.~3]{Bar17}}]\label{def:burnside}
For a finite group $G$, let $\burng$ denote the \emph{effective Burnside $\infty$-category}
of the category of finite (left) $G$-sets
and $G$-maps; informally, 
$\burng$ is the nerve of the (weak) $(2, 1)$-category
defined as follows: 
\begin{itemize}
\item  
The
objects of $\burng$ are finite $G$-sets $S$,
\item
 Given finite $G$-sets $S$ and
$T$, $\hom_{\burng}(S, T)$ is the nerve of the groupoid of spans 
of finite $G$-sets
\[ \xymatrix{
& U \ar[rd] \ar[ld] & \\
S &  & T
}\]
and isomorphisms 
of spans.
\item Composition is given by pullback of spans. 
\end{itemize}
\end{definition}

That is, $\burng$ is the span category of the category of finite $G$-sets, as in 
\cite[App.~C]{BH17}. 
The $\infty$-category
$\burng$ is preadditive, and the direct sum comes from the disjoint union of
finite $G$-sets.

One then obtains the following definition 
\cite[Sec.~6]{Bar17}
of a Mackey functor; this reduces to the  classical notion  when
$\mathcal{C}$ is the category of abelian groups.

\begin{definition}[Mackey functors] 
Given any presentable, preadditive $\infty$-category $\mathcal{C}$, we define a
\emph{$\mathcal{C}$-valued Mackey functor} (for the finite group $G$) to be a $\mathcal{C}$-valued
presheaf on $\burng$ which takes finite coproducts of $G$-sets to products in
$\mathcal{C}$. 
We let $\msk_G(\mathcal{C})$ denote the $\infty$-category of
$\mathcal{C}$-valued Mackey functors. 

 Let $M \in \msk_G(\mathcal{C})$. 
Given a subgroup $H \subseteq G$, we write
\[ M^H  \stackrel{\mathrm{def}}{=} M(G/H), \]
and call this the \emph{$H$-fixed points} of $M$. 
\end{definition} 

\begin{remark}[Comparison with  the $\mathcal{P}_{\Sigma}$-construction]\label{rem:p_sigma}
Consider the nonabelian derived $\infty$-category $\mathcal{P}_{\Sigma}(\burng)$
of $\burng$, in the sense of \cite[Sec.~5.5.8]{Lur09}, i.e.,
$\mathcal{P}_{\Sigma}(\burng)$ is the $\infty$-category of presheaves on
$\burng$ which preserve finite products, or equivalently 
$\mathcal{P}_{\Sigma}(\burng)$ is obtained by freely adding sifted colimits to
$\burng$. Then $\msk_G(\mathcal{C}) = \mathcal{P}_{\Sigma}(\burng) \otimes
\mathcal{C}$ via the Lurie tensor product,
cf.~\cite[Sec.~4.8.1]{Lur17}. 
\end{remark}

\begin{construction}[The symmetric monoidal structure on Mackey functors] 
Suppose now $\mathcal{C}$ is a presentably symmetric monoidal $\infty$-category
which is preadditive. Then there is a canonical structure of a
presentably symmetric  monoidal structure 
on $\msk_G(\mathcal{C})$, obtained (implicitly by Day convolution) as follows.  
We consider the symmetric monoidal structure on $\burng$ obtained from the
cartesian product on finite $G$-sets (and products of spans). 
This symmetric monoidal structure commutes with finite coproducts in each
variable. Applying $\mathcal{P}_\Sigma$, we obtain a canonical
presentably symmetric monoidal 
structure on $\mathcal{P}_\Sigma( \burng)$ such that the Yoneda functor
is symmetric monoidal, \cite[Prop.~4.8.1.10]{Lur17}. 
Now via the Lurie tensor product, $\msk_G(\mathcal{C}) = \mathcal{P}_{\Sigma}(
\burng) \otimes \mathcal{C}$
then acquires the structure of a presentably symmetric monoidal
$\infty$-category. 
\end{construction} 
\begin{remark}[Functoriality of $\msk_G(-)$]\label{rem:functoriality}
Let $\mathcal{C}, \mathcal{D}$ be presentable, preadditive $\infty$-categories.
Let $F \colon \mathcal{C} \to \mathcal{D}$ be an accessible functor 
which commutes with finite coproducts (but not necessarily all colimits). 
Then we can still define a natural functor $\msk_G(\mathcal{C}) \to
\msk_G(\mathcal{D})$ induced by $F$ by sending a $\mathcal{C}$-valued Mackey functor to the
corresponding $\mathcal{D}$-valued one (i.e., composing with $F$). 
However, this is slightly awkward
to formulate in our setup where $\msk_G(\mathcal{C}) = \mathcal{C} \otimes
\mathcal{P}_\Sigma(\burng)$, since this tensor product takes place in the world
of presentable $\infty$-categories. 
We can modify this by fixing a suitable cardinal $\kappa$, considering the
$\kappa$-compact objects $\mathcal{C}^{\kappa} \subseteq \mathcal{C}$, then defining the
cocontinuous functor $\mathrm{Ind}( \mathcal{C}^\kappa) \to \mathcal{D}$ and
applying $\msk_G(-) = (-) \otimes \mathcal{P}_\Sigma(\burng)$ to it. Varying
$\kappa$, we obtain a functor out of $\mathcal{C}$. In particular, this also shows that if
$\mathcal{C}, \mathcal{D}$ are symmetric monoidal and if $F$ has a lax symmetric
monoidal structure, then $\msk_G(\mathcal{C}) \to \msk_G(\mathcal{D})$ has a lax
symmetric monoidal structure; alternatively one could see this using Day
convolution, cf.~\cite{Gla16} or \cite[Sec.~2.2.6]{Lur17}. 
\label{technicalpt}
\end{remark}

\begin{remark}[Spectral Mackey functors and $G$-spectra] 
Suppose $\mathcal{C} = \sp$ is the $\infty$-category of spectra. 
Then by \cite{GM11, Na16}, we have an equivalence between $\mack_G( \sp)$ and 
the $\infty$-category $\GSpec$ of (genuine) $G$-spectra, cf.~also the appendix for an independent
account of this equivalence. The
target of equivariant algebraic $K$-theory will naturally be $\mack_G(\sp)$, and
so we can equally regard equivariant algebraic $K$-theory as a $G$-spectrum. 
\end{remark} 

\begin{example}[The case of the trivial group] 
Suppose $G =(1)$ is the trivial group. 
Then $\burn_{(1)}$ is the category of finite sets and correspondences between
them. 
This is 
the free preadditive category on a single generator,
a result due to Cranch \cite{Cra10}, cf.~\cite[Prop.~C.1]{BH17} for another account. 
It follows that $\msk_{(1)}(\mathcal{C}) = \mathcal{C}$. 
In particular, it follows that $\mathcal{P}_{\Sigma}( \burn_{(1)})$ is the
$\infty$-category of $\mathbb{E}_\infty$-anima, since this is the free presentable
preadditive $\infty$-category on one object, cf.~\cite{GGN15}. 
\end{example}

\begin{cons}[Relation to the orbit category]
\label{relationtoorbit}
Let $\sOG$ be the \emph{orbit category} of $G$, i.e., the category of
nonempty transitive $G$-sets.
We have a natural functor $\sOG \to \burng$ which sends the $G$-set
$S$ to $S \in \burng$ and the $G$-map $f \colon S \to T$  to the span
\[ \xymatrix{
& \ar[ld]^{\mathrm{id}} S  \ar[rd]^f  \\
S &   & T. 
}\]
We also obtain a natural functor $\sOG^{op} \to \burng$ in a similar
(dual) manner. 
Suppose $f \colon G/H \to G/H'$ is a morphism in $\sOG$. Given a
$\mathcal{C}$-valued Mackey functor $M$, we then obtain 
morphisms in $\mathcal{C}$
\[ f^*  \colon M^{H'} \to M^H, \quad f_* \colon M^H \to M^{H'} . \]
Thus, given the Mackey functor $M$, we obtain two functors
\[ \sOG^{op} \to \mathcal{C}, \quad \sOG \to \mathcal{C},  \]
which both send $G/H \mapsto M(G/H)=M^H$, and such that the functoriality is via
$f^*$ in the first case and via $f_*$ in the second case. \end{cons}

Next, we discuss the most basic source of Mackey functors: the Borel-equivariant
ones, or those $M$ for which $M(G/H) \simeq M( G)^{hH}$ for all subgroups $H
\subseteq G$. We begin with the case where $\mathcal{C}$ is given by $\mathbb{E}_\infty$-monoids
in anima. 

\begin{proposition} 
\label{borelequivspaces}
There is a   symmetric monoidal Bousfield localization functor
$\mathcal{P}_{\Sigma}(\burng) \to \fun(BG, \mathcal{P}_{\Sigma}(\burn_{(1)}))$
such that the essential image of its right adjoint inclusion consists of those
product-preserving presheaves $\mathcal{F} \colon (\burng)^{op} \to \mathcal{S}$ such
that 
for each finite $G$-set $S$, the natural map 
$\mathcal{F}(S) \to \mathcal{F}(G \times S)^{hG}$ is an equivalence. 
Here $G$ acts on $G$ (in the category of finite $G$-sets, and hence in $\burng$)
by right multiplication. 
\end{proposition} 
\begin{proof} 
Let $y \colon \burng \to 
\mathcal{P}_{\Sigma}(\burng)$ be the Yoneda embedding, and consider the
Bousfield localization functor $L_{\mathcal{I}}$ on $\mathcal{P}_{\Sigma}(\burng)$ with respect 
to the maps 
$\mathcal{I} = \{ y(G \times S)_{hG} \to y(S) \}$,
 for each finite $G$-set $S$. Here we use the map of $G$-sets $G \times S \to
 S$ given by projection onto the second factor, and the $G$-action on the source
 (in the category of $G$-sets) by right multiplication on the first
 factor. 

Since $y$ is symmetric monoidal and the tensor product on 
$\mathcal{P}_{\Sigma}(\burng)$ commutes with colimits in each variable, the class $\mathcal{I}$  is preserved by tensoring
with objects in the image of $y$, and we see that this Bousfield localization
$L_{\mathcal{I}}$
respects the symmetric monoidal structure. 
Unwinding the definitions, we see that the image of $L_{\mathcal{I}}$ is precisely those product-preserving
presheaves $\mathcal{F}$ as in the statement
because $\mathrm{Hom}_{\mathcal{P}_{\Sigma}(\burng)}(y(G\times S)_{hG},\mathcal{F})=
 \mathcal{F}(G\times S)^{hG}$.
 In particular, for any finite
$G$-set $S$ which is $G$-free, $y(S)$ is $\mathcal{I}$-local, as one sees by
unwinding the definition of mapping anima in $\burng$. 

We claim that the $\left\{y(S)\right\}$ for $S$   finite and $G$-free
form a set of compact projective generators for 
$L_{\mathcal{I}} \mathcal{P}_{\Sigma}( \burng)$. Compactness and projectivity follow because
for a finite free $G$-set $S$, 
the functor
$\mathcal{F} \mapsto \mathcal{F}(S)$ (with values in $\mathcal{S}$) commutes with sifted colimits 
on 
$\mathcal{P}_{\Sigma}( \burng)$
and 
carries the maps in 
$\mathcal{I}$ to equivalences. Moreover,  the $y(T)$ for $T \in
\burng$ can be expressed up to $\mathcal{I}$-equivalence as colimits of the
$y(S)$ for $S$ finite $G$-free by 
construction; therefore, the $\left\{y(S)\right\}$ generate. This verifies the claim about 
$L_{\mathcal{I}} \mathcal{P}_{\Sigma}( \burng)$. 

The symmetric monoidal functor 
$\burng \to \fun(BG, \burn_{(1)})$  which remembers an underlying set, or
correspondence, with $G$-action
extends to a cocontinuous symmetric monoidal functor
$\mathcal{P}_{\Sigma}(\burng) \to \fun(BG,
\mathcal{P}_{\Sigma}( \burn_{(1)}))$. 
Evidently, this functor carries the class of maps $\mathcal{I}$ to equivalences,
and factors symmetric monoidally through the 
Bousfield localization $L_{\mathcal{I}}$. It remains to show that the induced
functor 
$L_{\mathcal{I}} \mathcal{P}_{\Sigma}( \burng) \to 
\fun(BG,
\mathcal{P}_{\Sigma}( \burn_{(1)}))$ is an equivalence. The 
compact projective generators  on both sides are given by  $y(S)$ for $S$ a finite free
$G$-set, so it suffices to compare maps between them. Equivalently, it suffices to show that the map 
$\hom_{\burng}( S, T) \to 
\hom_{\fun(BG, \burn_{(1)})}( S,T)$ is an equivalence for $S, T$
finite free $G$-sets (in fact, it suffices for the $G$-action on one of them to
be free). 
By decomposing $S$ and $T$ and using duality, it suffices to prove that this map
is an equivalence when $S = \ast$ and $T = G$; then one checks directly that
both sides are the free $\mathbb{E}_\infty$-space on a generator and the map is
an equivalence. 
\end{proof}

\begin{construction}[Borel-equivariant objects]\label{construct:borel}
Let $\mathcal{C}$ be a presentably symmetric monoidal, preadditive $\infty$-category.
Tensoring the Bousfield localization of \Cref{borelequivspaces} with
$\mathcal{C}$, 
we obtain a symmetric monoidal Bousfield localization functor
$$\msk_{G}(\mathcal{C}) \to \fun(BG, \mathcal{C}),$$ 
with a fully faithful lax symmetric monoidal right adjoint functor 
called ``Borelification'', 
\[ (-)^{\bor}  \colon \fun(BG, \mathcal{C}) \to \msk_G(\mathcal{C}). \]
The essential image of $(-)^{\bor}$ (called \emph{Borel-equivariant} objects) is given by those product-preserving
presheaves $\mathcal{F} \colon (\burng)^{op} \to \mathcal{C}$ such that for any finite $G$-set
$S$, we have $\mathcal{F}(S) \xrightarrow{\sim} \mathcal{F}(G \times S)^{hG}$. 
In other words, $\mathcal{F}$ is Borel-equivariant if and only if the restriction of
$\mathcal{F}$ to $\sOG^{op}$ is right Kan
extended from the full subcategory spanned by the $G$-set $G$. 
\label{borelequivobjects}
\end{construction}

We will be interested in the above construction when $\mathcal{C} = \catst$.  
For this, recall 
that 
$\catst$ is preadditive
(cf.~\cite[Prop.~4.7]{Bar16} for this result in the
closely related context of
Waldhausen $\infty$-categories). Moreover, $\catst$ is presentable \cite[Cor.~4.25]{BGT13}, and symmetric
monoidal under the Lurie tensor product \cite[Th.~3.1]{BGT13}.
For an idempotent-complete stable $\infty$-category $\mathcal{A}$
with $G$-action, we
obtain
a $\catst$-valued Mackey functor $M_{\mathcal{A}}$ such that 
$M_{\mathcal{A}}(G/H)  = \mathcal{A}^{hH}$. 
For a map of finite $G$-sets $f \colon G/H \to G/K$, then $f^*$  is the
natural pullback map $\mathcal{A}^{hK} \to \mathcal{A}^{hH}$. 
We will need to know that in this case, $f_*$ can also be described explicitly. 

\begin{proposition}\label{prop:push_forward_borel}
Let $M \in \msk_G( \catst)$ be Borel-equivariant. 
Then for any map $f \colon S \to T$  of finite $G$-sets, the 
functor $f_*: M(T) \to M(S)$ (of \Cref{relationtoorbit}) is both left and right adjoint to $f^*: M(S) \to
M(T)$. 
\end{proposition} 
\begin{proof} 
We will verify this by invoking from \cite{Bar17} a construction of a
$\catst$-valued Mackey functor
which does have the desired adjointness property, which is Borel, and 
whose underlying object of $\fun(BG, \catst)$ agrees with that of $M$. 

Let $\mathcal{A} = M(G/\left\{e\right\}) \in \fun(BG, \catst)$. 
Note first that for any map of finite sets $f \colon S_0 \to T_0$, the 
pullback functor
$f^*  \colon \fun(T_0, \mathcal{A}) \to \fun(S_0, \mathcal{A})$ admits a right (and
left) adjoint $f_* \colon \fun(S_0, \mathcal{A}) \to \fun(T_0, \mathcal{A})$ given by
summing over the fibers. 
Moreover, one has the base-change property: 
given a pullback square of finite sets 
\[ \xymatrix{
U_0 \ar[d]  \ar[r] & V_0 \ar[d]  \\
S_0 \ar[r] &  T_0,
}\]
the induced square in $\catst$ obtained by applying 
pullback everywhere is left and right adjointable \cite[Def.~7.3.1.2]{Lur09}. 

Now for every $G$-set $S$, we consider $\fun_G( S, \mathcal{A}) \in \catst$;
this is also $\fun( S , \mathcal{A})^{hG}$ for the diagonal $G$-action (with
$G$ acting on both $S$ and $\mathcal{A}$). 
Given a map of $G$-sets $f \colon S \to T$, we have a pullback functor $f^* \colon \fun_G(T,
\mathcal{A}) \to \fun_G(S, \mathcal{A})$; we obtain a $\catst$-valued presheaf
on the category of finite $G$-sets. 
We claim that for any map $f \colon S \to T$, the functor $f^* \colon  \fun_G(T, \mathcal{A}) \to \fun_G(S,
\mathcal{A})$ 
admits  an adjoint (in either direction), and furthermore that 
for any pullback square of finite $G$-sets
\[ \xymatrix{
U \ar[d]  \ar[r] & V \ar[d]  \\
S \ar[r] &  T,
}\]
the induced square in $\catst$ obtained by pullback functoriality 
is adjointable (in either direction). This follows from the 
above verification in the case of 
finite sets, and then taking homotopy fixed points in view of 
\cite[Cor.~4.7.4.18]{Lur17}. 
Indeed, the result of \emph{loc.~cit.} shows that given a square in $\fun(BG,
\mathrm{Cat}_\infty)$ which is left (or right) adjointable, the square in
$\mathrm{Cat}_\infty$ obtained by taking $G$-homotopy fixed points remains left
(or right) adjointable. 

We thus have a $\catst$-valued presheaf on the category of finite $G$-sets, $S
\mapsto \fun_G(S, \mathcal{A})$, and we have verified the
adjointability conditions needed to
apply 
the unfurling construction of \cite[Sec.~11]{Bar17}, which produces a
$\catst$-valued Mackey functor $M'$
extending the above presheaf whose restriction to $\mathscr{O}(G)$ is given by the
adjoints $f_*$. In particular, $M'$ satisfies the condition of the
proposition. 
The Mackey functor $M'$ is Borel-complete (since this
condition only depends on the restriction to the category of finite $G$-sets) and must therefore agree with
$M$, since the restrictions of $M, M'$ in $\fun(BG, \catst)$ agree; it follows
now that $M$ has the desired property in the proposition. 
\end{proof}

Now we describe the fundamental construction for our purposes, the equivariant
algebraic $K$-theory of group actions, in the form constructed by
Barwick--Glasman--Shah, cf.~\cite[Sec.~8]{BGS20}.

\begin{construction}[{Equivariant $K$-theory of group actions
\cite[Sec.~8]{BGS20}}]
It follows from \Cref{construct:borel} that we have a lax symmetric monoidal functor
of ``Borelification''
\[ 
  (-)^{\mathrm{Bor} } \colon 
\fun(BG, \catst) \to \msk_G(\catst),  \]
and composing it with the lax symmetric monoidal algebraic $K$-theory functor as in \Cref{rem:functoriality},
we obtain a lax
symmetric monoidal functor 
\begin{equation} K_G \colon \fun(BG, \catst) \to \mack_G( \sp), \end{equation}
such that $K_G(\mathcal{A})^H=K(\mathcal{A}^{hH})$ whenever
$\mathcal{A}\in \fun(BG, \catst)$ and $H\subseteq G$.
\end{construction} 

\begin{example}[Equivariant algebraic $K$-theory of $\mathbb{E}_\infty$-rings]
Let $R$ be an $\mathbb{E}_\infty$-ring with $G$-action. 
Then we write $K_G(R)$ for $K_G( \perf(R))$. 
\end{example}

We will actually need a slight generalization of the above, in order to handle
invariants other than algebraic $K$-theory. 
Given a base $\mathbb{E}_\infty$-ring $R$, we consider the presentably symmetric
monoidal $\infty$-category $\mot_R$ of
$R$-linear noncommutative motives, cf.~\cite{BGT13, HSS17}. 
We have a symmetric monoidal  functor 
$\mathcal{U} \colon \cats{R} \to \mot_R$ which is an additive
invariant, i.e., it preserves filtered colimits and 
carries semiorthogonal decompositions 
to direct sums in $\mot_R$; 
moreover, 
$\mathcal{U}$ is initial for these data and conditions. 

\begin{construction}[$\mot_R$-valued Mackey functors] 
\label{motRmackey}
Composing the functor $(-)^{\mathrm{Bor}}$ with 
$\mathcal{U}$, we obtain a lax symmetric monoidal functor
\[ \mathcal{U}_G \colon \fun(BG, \cats{R}) \to \mack_G(  \mot_R) \simeq \mack_G(\sp)
\otimes \mot_R, \]
i.e., $\mathcal{U}_G$ takes values in Mackey functors in
$R$-linear noncommutative motives. (Here we use
\Cref{technicalpt}, since $\mathcal{U}$ does not preserve all colimits.) 
Since for any $\mathcal{A} \in \cats{R}$, 
the algebraic $K$-theory $K(\mathcal{A})$ can be recovered as $\hom_{\mot_R}(
\mathbf{1}, \mathcal{U}(A) )$ by \cite{BGT13} and \cite{HSS17}, it follows that 
the equivariant algebraic $K$-theory functor $K_G$
is the composition of $\mathcal{U}_G$ and the functor
$\mathrm{id} \otimes \hom_{\mot_R}( \mathbf{1}, -) \colon \mack_G( \sp) \otimes \mot_R
\to \mack_G(\sp)$. 
\end{construction}

Finally, in order to treat assembly-type maps for group rings, we will need to discuss the coBorel variant of the above. 

\begin{construction}[coBorel Mackey functors] 
Let $M$ be a $\mathcal{C}$-valued Mackey functor, for $\mathcal{C}$ a
presentable preadditive $\infty$-category. 
Note that $\msk_G(\mathcal{C}) = \mathcal{P}_{\Sigma}(\burng) \otimes
\mathcal{C}$ is naturally tensored over 
$\mathcal{P}_{\Sigma}(\burng)$. Let $y: \burng \to \mathcal{P}_\Sigma(\burng)$
denote the Yoneda embedding. 

We will say that $M$ is \emph{coBorel} if the natural map 
$(M \otimes y(G))_{hG} \to M$ is an equivalence in $\msk_G(\mathcal{C})$. 
Any $M \in \msk_G(\mathcal{C})$ admits its {\em coBorelification} $M_{\cobor} = (M
\otimes y(G))_{hG}$, which is the universal coBorel Mackey functor mapping to
$M$; this follows because the object $y(G)_{hG}$ in
$\mathcal{P}_{\Sigma}(\burng)$ is an idempotent object for the tensor
structure.\footnote{In fact, there is a natural symmetric monoidal functor from the
$\infty$-category of $G$-anima to $\mathcal{P}_\Sigma( \burng)$ which is the
identity on $G$-sets; then $y(G)_{hG}$ is the image of the $G$-space $EG$.}
The coBorelification only depends on the underlying object of $\fun(BG,
\mathcal{C})$ (since $M \otimes y(G)$ does), so we can also consider this as a functor
$$(-)_{\cobor}  \colon \fun(BG, \mathcal{C}) \to \msk_{G}(\mathcal{C}).$$ 
Dually as in \Cref{borelequivobjects}, 
$(-)_{\cobor}$ is fully faithful; the essential image consists of those
$M \in \msk_G(\mathcal{C})$ such that 
the dual comparison maps $\left( M^{\left\{1\right\}}\right)_{hH} \to M^H$ are equivalences
for all $H \subseteq G$. 
By the universal property, we obtain a natural map $(-)_{\cobor} \to (-)^{\bor}$
which is an equivalence after forgetting to
$\fun(BG, \mathcal{C})$. 
The functor $(-)_{\cobor}
\colon \fun(BG, \mathcal{C}) \to \msk_{G}(\mathcal{C})$ is the left adjoint to
the localization functor $\msk_G(\mathcal{C}) \to \fun(BG, \mathcal{C})$ of \Cref{construct:borel} (whose right adjoint was the
Borelification). 
\end{construction}

We now describe the coBorelification of 
the $\catst$-valued Mackey functors 
constructed above. 
This involves controlling homotopy orbits in $\catst$. 
To begin with, we need some facts about limits and colimits of presentable, stable $\infty$-categories,
cf.~\cite[Sec.~5.5.3]{Lur09}. 
Let $\prl$ denote the $\infty$-category of presentable, stable
$\infty$-categories and left adjoint functors between them. 
Let $\prr$ 
 denote the $\infty$-category of presentable, stable
$\infty$-categories and right adjoint functors between them, so we have an
equivalence in $\prl \simeq (\prr)^{op}$. 
It follows that the underlying $\infty$-category of   a colimit in $\prl$ (of
some diagram $i \mapsto \mathcal{C}_i, i \in I$) can be calculated by 
taking the inverse limit along $I^{op}$ of the right adjoints
\cite[Th.~5.5.3.18]{Lur09}. 
Explictly, via the Grothendieck construction, we can 
express the diagram $I \to \prl$ in terms of a presentable fibration 
$\widetilde{\mathcal{C}} \to I$, which is both a cartesian and a cocartesian
fibration (cf.~\cite[Def.~5.5.3.2]{Lur09}); the 
limit in $\prl$ is given by the 
$\infty$-category of cocartesian sections, whereas the colimit is given by the
$\infty$-category of cartesian sections. 

We can use this to describe colimits in $\catst$.  

\begin{cons}[Colimits in $\catst$] 
\label{relationcatstprl}
Consider the functor $\mathrm{Ind}  \colon \catst \to \prl$ (\cite[Sec.~5.3.5]{Lur09}).
This functor admits a right adjoint 
sending a presentable, stable $\infty$-category to its subcategory of compact
objects; therefore, $\mathrm{Ind}$ 
commutes with all colimits. 
To compute a colimit in $\catst $ of an $I$-indexed diagram, $i \mapsto
\mathcal{A}_i$, we therefore form the $I^{op}$-indexed diagram of
$\mathrm{Ind}(\mathcal{A}_i)$ and the right adjoint functors, and then take the
compact objects in the limit. 
\end{cons}

\begin{example}[Homotopy orbits in $\catst$] 
Let $\mathcal{A} \in \fun(BG, \catst)$. 
We claim that  $\mathcal{A}_{hG}$ is naturally described as the full subcategory of 
compact objects in 
$\mathrm{Ind}(\mathcal{A})^{hG}$. 

To see this, we first describe the homotopy orbits
$\mathrm{Ind}(\mathcal{A})_{hG}$ in $\prl$. 
Form the presentable fibration over $BG$ with fiber
$\mathrm{Ind}(\mathcal{A})$; as above, the cocartesian sections give
$\mathrm{Ind}(\mathcal{A})^{hG}$ (the homotopy limit in $\prl$) while the
cartesian sections give $\mathrm{Ind}(\mathcal{A})_{hG}$. Since $BG$ is an
$\infty$-groupoid, the cartesian and cocartesian sections are the same and we
have a canonical identification $\mathrm{Ind}(\mathcal{A})_{hG} =
\mathrm{Ind}(\mathcal{A})^{hG}$ in $\prl$. The claim about $\mathcal{A}_{hG}$ now follows by passage to
compact objects. 

Equivalently, we find that $\mathcal{A}_{hG} \in \catst$ is the full subcategory
of $\mathcal{A}^{hG}$ generated as a thick subcategory by
the image of the functor $\mathcal{A} \to \mathcal{A}^{hG}$ adjoint to the
forgetful functor, since this image forms a set of compact generators of
$\mathrm{Ind}(\mathcal{A})^{hG}.$
In particular, we have a natural fully faithful embedding $\mathcal{A}_{hG}
\subseteq
\mathcal{A}^{hG}$ (which we verify below to be the norm map); it follows that 
for any diagram $\mathcal{A}_j, j \in J$ in $\fun(BG, \catst)$, the
natural map 
$(\varprojlim_J \mathcal{A}_j)_{hG} \to \varprojlim_J (\mathcal{A}_j)_{hG}$ is
fully faithful, since $(-)^{hG}$ commutes with limits. 
\end{example}

\begin{proposition} 
\label{fullyfaithfulnormmap}
Let $\mathcal{A} \in \fun(BG, \catst)$. 
Then the natural map 
of $\catst$-valued Mackey functors
$$ \mathcal{A}_{\cobor} \to \mathcal{A}^{\bor},$$
is fully faithful on $H$-fixed points for $H \subseteq G$. 
Moreover, $(\mathcal{A}_{\cobor})^H \subseteq (\mathcal{A}^{\bor})^{H} =
\mathcal{A}^{hH}$ is the
thick subcategory generated by the image of the biadjoint $\mathcal{A} \to
\mathcal{A}^{hH}$. 
\end{proposition} 
\begin{proof} 

Let $H \subseteq G$. Then we claim that the natural map (i.e., the norm map)
in $\catst$,
\begin{equation} \label{compmapcats}  \mathcal{A}_{hH} =((\mathcal{A} \otimes G)^{hH})_{hG}  \to \mathcal{A}^{hH} 
= ((\mathcal{A} \otimes G)_{hG})^{hH}
\end{equation}
is fully faithful; this map is the $H$-fixed points of $\mathcal{A}_{\cobor} \to
\mathcal{A}^{\bor}$. 
But this follows from the observation in the previous example: 
we saw that if $\mathbf{T} \colon \fun(BG, \catst) \to \catst$ is the functor 
$\mathcal{B}
\mapsto \mathcal{B}_{hG}$, then if  
$\mathcal{B}$ has a $G \times H$-action, then 
$\mathbf{T}( \mathcal{B}^{hH}) \to \mathbf{T}(\mathcal{B})^{hH}$ is fully
faithful. 
It follows that $\mathcal{A}_{\cobor} \to \mathcal{A}^{\bor}$ is fully faithful
on each fixed points. To see that its essential image is  the subcategory as
claimed, we observe that $\mathcal{A} \to \mathcal{A}_{hH}$ has image
generating the target as a thick subcategory. 
\end{proof} 

\begin{example}[Assembly maps] \label{assembly maps}
Let $G$ act trivially on $\perf(R)$. 
Then we find that for each subgroup $H \subseteq G$, one has $\perf(R)_{hH} \simeq \perf(
R[H]) \subseteq \fun(BH, \perf(R))$ is the collection of compact objects in 
$\fun(BH, \mathrm{Mod}(R))$. In particular, the $\catst$-valued Mackey functor 
$(\perf(R))_{\cobor}$ is precisely the one that leads to the theory of assembly
maps, cf.~\cite{RV18}. 
\end{example}

\begin{construction}[Equivariant algebraic $K$-theory, coBorel version] 
\label{equivalgKcobor}
Combining the above, we obtain a functor
$$ \mathcal{U}_{G, \cobor}  \colon \fun(BG, \cats{R})  \xrightarrow{(-)_\cobor} \msk_G(\cats{R}) \to 
\mack_G( \mot_R),$$
which is the coBorel version 
of \Cref{motRmackey}. 
If $\mathcal{A}$ is an algebra object of $\fun(BG, \cats{R})$, then 
$\mathcal{A}_{\cobor} = \mathcal{A}^{\bor} \otimes y(G)_{hG}$ is a module over $\mathcal{A}^{\bor}$ in
$\msk_G(\cats{R})$. 
Therefore, $\mathcal{U}_{G, \cobor}(\mathcal{A})$ is a module over
$\mathcal{U}_G(\mathcal{A})$ (which is an algebra object of $\mack_G(\mot_R)$). 
\end{construction}

\section{Review of nilpotence}

To prove our descent theorems, it will be convenient to use the language of
nilpotence, as in \cite{MNN17, MNN19}. 
For the material in section 5 and further, we also need the variant of $\epsilon$-nilpotence, as used in \cite{CMNN}. 

\begin{definition}[Nilpotence]
Given a finite group $G$ and a family $\sF$ of subgroups, 
a $G$-spectrum $X$ is said to be \emph{$\sF$-nilpotent}
\cite[Def.~6.36]{MNN17} if it 
belongs to the thick subcategory (or equivalently the thick
$\otimes$-ideal) generated by $G$-spectra which are induced from
subgroups in $\sF$. 
We say that a $G$-spectrum $X$ is \emph{$(\sF, \epsilon)$-nilpotent} if 
there exists a finite set of prime numbers $\Sigma$ such that for every finite
spectrum $F$ whose localizations at primes in $\Sigma$ are nontrivial, 
then $X$ belongs to the thick $\otimes$-ideal of $G$-spectra 
generated by $F$ and the $\sF$-nilpotent $G$-spectra. 
(This somewhat involved definition 
 in particular implies that every passage to $T(n)$-local coefficients makes $X$
 $\mathcal{F}$-nilpotent, and this is an if and only if 
for the endomorphism $G$-ring spectrum of $X$.  
 Compare~\cite[Sec.~2.3]{CMNN} with $A = \prod_{H \in \sF} F(G/H_+, \mathbb{S})$.) 
\end{definition}

\begin{definition}[{$\sF$-completeness, cf.~\cite[Sec.~6.1]{MNN17}}] 
Given a finite group $G$ and a family $\sF$ of subgroups, 
let $E \sF$ be the classifying space of the family $\sF$ as reviewed in
\cite[Cons.~6.3]{MNN17}. 
We say that $X \in \GSpec$ is \emph{$\sF$-complete}
if the map $X \to F( E\sF_+, X)$ is an equivalence, or equivalently if $X$ is
complete with respect to the algebra object $A = \prod_{H \in \sF} F(G/H_+,
\mathbb{S})$. This in particular implies that $X^G \xrightarrow{\sim}
\varprojlim_{G/H \in \sOFG^{op}} X^H$. 
\end{definition} 

\begin{proposition} 
Given an $(\sF, \epsilon)$-nilpotent $G$-spectrum $X$, 
the natural comparison maps \begin{equation} 
\label{FcompmapX}
\varinjlim_{G/H \in \sOFG} X^H \to X^G
\to \varprojlim_{G/H \in \sOFG^{op}} X^H,
\end{equation} 
become equivalences after applying $L_{T(n)}$ for any height $n$ and implicit
prime $p$; moreover, the functor $L_{T(n)}$ can be applied either inside or outside the 
homotopy limit on the right of \eqref{FcompmapX},
i.e., the map  
\begin{equation} 
L_{T(n)}\left( \varprojlim_{G/H \in \sOFG^{op}} X^H\right) \to \varprojlim_{G/H \in
\sOFG^{op}} L_{T(n)}X^H
\label{compmap2X}
\end{equation} 
is an equivalence. 
\label{Fepsilonimpliescompmap}
\end{proposition} 
\begin{proof} 
Fix $n$ and the implicit prime $p$. 
Given an $\sF$-nilpotent $G$-spectrum $X$, the maps of \eqref{FcompmapX} are
equivalences, cf.~\cite[Prop.~2.8]{MNN19} (in particular, $X$ is
$\sF$-complete). 
If $X$ is 
$\sF$-nilpotent then the $T(n)$-localization of $X$ remains $\sF$-nilpotent by
a thick subcategory argument, whence \eqref{compmap2X} is also an equivalence. 
Now the collection 
of $G$-spectra for which \eqref{FcompmapX} and \eqref{compmap2X} 
are equivalences is a thick subcategory which contains the $\sF$-nilpotent
$G$-spectra and the $G$-spectra of the form $F \otimes Y$ for $Y \in \GSpec$ and
$F$ a finite torsion spectrum of type (at $p$) $\geq n+1$; this collection therefore
contains the $(\sF, \epsilon)$-nilpotent $G$-spectra. 
\end{proof} 

We now discuss some criteria for nilpotence, starting with the case of the
family $\triv$ consisting only of the trivial subgroup. 
Let $EG$ denote the universal free $G$-space (or equivalently the classifying
space of the family consisting of the trivial subgroup). 
Let $\widetilde{EG}$ denote the
cofiber of $EG_+ \to \mathbb{S}$ in $\GSpec$; it is naturally an algebra object
in $\GSpec$, as the smashing localization of $\mathbb{S}$ in $\GSpec$ away from the
localizing $\otimes$-ideal generated by the free $G$-spectra,
cf.~\cite[Prop.~6.5]{MNN17}. 
In the following, let $R$ be an associative algebra in $\GSpec$. 
Then we consider the associative algebra $ (R \otimes \widetilde{EG})^G$ in $\sp$. 
Since $EG = (G)_{hG}$ (in the $\infty$-category of $G$-anima), we have a
cofiber sequence
\[ R_{hG} \to R^G \to (R \otimes \widetilde{EG})^G,  \]
where 
$R_{hG} = R^{\left\{1\right\}}_{hG}$ and 
$R_{hG} \to R^G$ is the transfer for the $G$-spectrum
$R$. 

\begin{proposition}[Criteria for $\triv$-nilpotence] 
\label{crittrivnilp}
An associative algebra $R$ in $\sp_G$ is $\triv$-nilpotent (for $\triv = \left\{(1)\right\}$) if and only if
$(R  \otimes \widetilde{EG})^G$ is
contractible. 
\end{proposition} 

\begin{proof} 
This follows from \cite[Th.~4.19]{MNN17}, since $R \otimes \widetilde{EG}$ is the localization of
$R$ (in $\GSpec$) away from the localizing $\otimes$-ideal generated by the
free $G$-spectra. 
\end{proof} 

\begin{proposition}[Criterion for $(\triv, \epsilon)$-nilpotence] 
Let $R$ be an associative algebra in $\sp_G$. Suppose that 
$(R \otimes\widetilde{EG})^G $ has trivial $T(n)$-localization for
$n \geq 1$ and all primes $p$  and trivial rationalization. Then $R$ is $(\triv,
\epsilon)$-nilpotent. 
\label{crittrivnilpepsilon}
\end{proposition}

\begin{proof} 
Our assumptions imply that 
there exists a finite set of prime numbers $\Sigma$ such that for every finite
complex $F$ with nontrivial localizations at primes in $\Sigma$, 
the associative algebra spectrum 
$(R \otimes \widetilde{EG})^G$ belongs to the thick $\otimes$-ideal generated by
$F$, cf.~\cite[Prop.~2.7]{CMNN}. Thus, 
the $G$-spectrum 
$R \otimes \widetilde{EG}$
belongs to the thick $\otimes$-ideal generated by $F$; here we use the natural
adjunction $(i_*, (-)^G): \sp \rightleftarrows \GSpec$. 
Therefore, $R$ belongs to the localizing $\otimes$-ideal generated by $F$ and by
the $G$-spectrum $G_+$ (using the fiber sequence $R \otimes EG_+ \to R \to R
\otimes \widetilde{EG}$), and hence it belongs to the similarly
generated thick $\otimes$-ideal by 
\cite[Th.~4.19]{MNN17} (which we apply to $\mathcal{C} = \GSpec$ and the dualizable associative algebra object $F \otimes \mathbb{D} F \times \mathbb{D}
G_+ \in \GSpec$, for $\mathbb{D}$ the categorical dual) 
again.
The result follows. 
\end{proof}

We next include three general results about $\sF$-nilpotence for an arbitrary
family. 
The first result states that when $R$ is rational (i.e, $T(0)$-local), $\sF$-nilpotence is a purely
algebraic condition on $\pi_0$; the second (which will only be used with $\sF
= \triv$) gives a generalization of this to
$T(n)$-local objects.  
The third result allows us to transfer rational $\sF$-nilpotence to $(\sF,
\epsilon)$-nilpotence in the presence of an $\mathbb{E}_\infty$-structure, using
the May nilpotence
conjecture \cite{MNN_nilpotence}.
For this, we let $E\sF$ denote the universal $G$-space
for the family $\sF$ and $\widetilde{E \sF}$ the cofiber of the map
$E\sF_+ \to \mathbb{S}$ in $\GSpec$, so $\widetilde{E \sF}$ is the
localization of $\mathbb{S}$ away from the localizing $\otimes$-ideal generated by the
$\left\{G/H_+, H \in \sF\right\}$.

\begin{proposition}[{\cite[Prop.~4.11]{MNN19}}]\label{prop:rational}
Suppose that the associative algebra $R$ in $\sp_G$ is rational. Then $R$ is $\sF$-nilpotent  if and only if the
induction map 
$\bigoplus_{H \in \sF}\pi_0( R^H) \to \pi_0(R^G)$ is surjective, or equivalently
has image containing the unit. \qed
\end{proposition} 

Let $L_{T(i)} \GSpec$ denote the full subcategory of $\GSpec$ spanned by the
$T(i)$-local objects, i.e., those for which the $H$-fixed points for each
subgroup $H \subseteq G$ are $T(i)$-local spectra (at the implicit prime $p$); this
equivalence follows because the orbits form a set of compact generators for
$\GSpec$. 
We next give a criterion for $\sF$-completeness in $L_{T(i)}\GSpec$. 
This will use the vanishing of the Tate constructions in $L_{T(i)} \sp$,
due to \cite{Kuh04}, in the following equivalent form: 

\begin{lemma} 
If $\mathcal{C}$ is any presentable stable $\infty$-category and $X \in \fun(BH,
\mathcal{C})$ is an $H$-object in $\mathcal{C}$ for some finite group $H$, then $X_{hH}
\otimes T(i) \in \mathcal{C}$ belongs to the thick subcategory generated by $X
\otimes T(i)$. 
\label{Tatevanishconseq}
\end{lemma} 
\begin{proof} 
An equivalent form 
of the telescopic Tate vanishing is that, as an object of $\fun(BH, \sp)$ (with
trivial $H$-action), $T(i)$ belongs to the thick subcategory of $\fun(BH,
\sp)$ generated by $T(i) \otimes H_+$, cf.~\cite[Prop.~5.31]{MNN19}. 
From this, the result easily follows, since $(X \otimes T(i) \otimes H_+)_{hH}
= X \otimes T(i)$. 
\end{proof}

In the next result, we use the notation $\Phi^H$ for the $H$-geometric fixed
points functor on $\GSpec$. 

\begin{proposition}[Properties of $T(i)$-local $G$-spectra]\label{rem:k1_local_g_spectra}
Let $G$ be a finite group, $\mathcal{F}$ a family of subgroups and $i\geq 0$. 
Let $M \in L_{T(i)} \GSpec$. Then the following are equivalent: 
\begin{enumerate}
\item  $M$ is $\sF$-complete. 
\item For every finite type $i$ complex $F$, the $G$-spectrum $M \otimes F$ is $\sF$-nilpotent. 
\item We have $L_{T(i)} (\Phi^H M) =0 $ for $H \notin \sF$. 
\end{enumerate}
\end{proposition}
\begin{proof} 
We first claim that for each family $\sG $ of subgroups of $G$, the $G$-spectrum $E \sG_+ \otimes
T(i)$ is
$\sG$-nilpotent; we prove this by induction on $\sG$. 
To start with, when $\sG = \triv$, then $E\sG_+ = (G_+)_{hG}$; this
uses the $G$-action on the $G$-space $G$ (by right multiplication, so in the
category of $G$-anima). 
It follows from \Cref{Tatevanishconseq} that 
$E\triv_+ \otimes T(i)$ is $\triv$-nilpotent. 
Now we treat the inductive step. 
Fix a proper family $\sG$ such that 
$E \sG_+ \otimes T(i)$ is $\sG$-nilpotent. 
Choose a subgroup $H  \subseteq G$ which is minimal for the
property of not belonging to $\sG$; one forms a new family $\sG'$ obtained by
adding the conjugates of $H$ to $\sG$. 
Then there is a cofiber sequence
of pointed $G$-anima
\[ E \sG_+ \to E \sG'_+ \to  E \sG'_+ \wedge \widetilde{E} \sG = (G/H_+)_{h
W_H} \wedge \widetilde{E} \sG,    \]
where $W_H$ is the Weyl group of $H \subseteq G$. 
Using this, the inductive assumption, and  
\Cref{Tatevanishconseq}, the inductive step follows and the claim is proved.

Now we prove the result. Suppose $M \in L_{T(i)} \GSpec$ is $\sF$-complete. 
By the thick subcategory theorem, condition $(2)$ is independent of the choice of $F$
 and we
 choose a finite type $i$ complex $F$ such that $F$ admits the structure of a ring
spectrum; given a $v_i$-self map $v$ of $F$ which we may assume central, we can take $T(i) = F[v^{-1}]$.
Then $M \otimes F$ admits the structure of a $T(i)$-module, since $M$ is
$T(i)$-local. It follows that the 
\emph{$\sF$-cellularization}\footnote{Compare \cite[Cons.~3.2, Prop.~6.5]{MNN17} for an
account, where cellularization is called acyclization.} $E \sF_+ \otimes M \otimes F$ of $M \otimes F$ 
belongs to the thick $\otimes$-ideal  of $\GSpec$ generated by $E \sF_+ \otimes
T(i) $ and is therefore $\sF$-nilpotent, by our initial claim. Consequently, the $\sF$-completion
$F(E \sF_+, E \sF_+ \otimes M \otimes F)$ (which is $M \otimes F$ again since
this is $\sF$-complete) is also $\sF$-nilpotent (here we implicitly use that the
$\sF$-completion of a $G$-spectrum depends only on its $\sF$-cellularization). Thus, (1) implies (2). 
Clearly (2) implies (3), again by smashing with $F$. 
If (3) holds, then $M \otimes F = M \otimes T(i)$ has trivial geometric fixed
points $\Phi^H$ for $H \notin \sF$, whence 
$M  \otimes F = E \sF_+ \otimes M \otimes F = E \sF_+ \otimes M \otimes T(i)$,
which we have seen is $\sF$-nilpotent. Thus, (3) implies (2). Finally, (2)
implies (1) by writing $M$ (which is assumed $T(i)$-local) as an inverse limit of $M \otimes F$ for suitable
finite type $i$ complexes $F$ (e.g., generalized Moore spectra). 
\end{proof} 

\begin{corollary} 
\label{moduleovercomplete}
Let $R \in L_{T(i)} \sp_G$ be an algebra object which is $\sF$-complete. 
Then any $R$-module $M \in L_{T(i)} \sp_G$ is $\sF$-complete. 
\end{corollary} 
\begin{proof} 
This follows from item (3) of 
\Cref{rem:k1_local_g_spectra}, since $\Phi^H$ is a symmetric monoidal functor. 
\end{proof} 

\begin{proposition} 
\label{splittingcor}
Let $R \in L_{T(i)} \sp_G$ be an $\mathbb{E}_1$-algebra, and let $M \in
L_{T(i)} \sp_G$ be an $R$-module. 
Then the map $M^G \to M^{hG}$ admits a section as $R^G$-modules. 
Similarly, the map $L_{T(i)} (M_{hG}) \to M^G$ admits a section as
$R^G$-modules. If $G = C_p$, then $M$ is Borel-complete if and only if either of
these maps is an equivlaence. 
\end{proposition} 
\begin{proof} 
All of this follows because the composite map 
$L_{T(i)} (M_{hG}) \to M^G \to M^{hG}$ is the norm, which is an equivalence
since Tate constructions
vanish in $T(i)$-local homotopy  \cite{Kuh04}. 
\end{proof} 

\begin{proposition} 
Let $R$ be an $\mathbb{E}_\infty$-algebra
in the symmetric monoidal $\infty$-category $\GSpec$.
Suppose that the rationalization $R_{\mathbb{Q}}$
is $\sF$-nilpotent. Then $R$ is $(\sF, \epsilon)$-nilpotent. 
\label{rationaltoFnilp}
\end{proposition} 
\begin{proof} 
By assumption, the $\mathbb{E}_\infty$-ring $(R \otimes \widetilde{E
\sF})^G_{\mathbb{Q}}$ is 
contractible. 
Therefore, by the main result of \cite{MNN_nilpotence}, the
$\mathbb{E}_\infty$-ring 
$(R \otimes \widetilde{E
\sF})^G$ is 
annihilated by $L_{T(n)}$ for all $n$ and implicit primes $p$. 
In particular, by 
\cite[Prop.~2.7]{CMNN}, there exists a finite set $\Sigma$ of primes such that 
$(R \otimes \widetilde{E
\sF})^G$
belongs to the thick $\otimes$-ideal of spectra generated by any finite spectrum
$F$ such that $F_{(p)} \neq 0$ for $p \in \Sigma$. 
This implies that $R \otimes E\widetilde{\sF}$ belongs to the thick
$\otimes$-ideal of $\GSpec$ generated by $F$, whence $R$ belongs to the
localizing $\otimes$-ideal generated by $F$ and $\{G/H_+, H \in \sF\}$ in view of the
cofiber sequence $R \otimes E\sF_+ \to R \to R \otimes \widetilde{E\sF}$. 
Finally, \cite[Th.~4.19]{MNN17} again implies that $R$ belongs to the thick $\otimes$-ideal
generated by $F$ and 
$\{G/H_+, H \in \sF\}$ in $\GSpec$, as desired. 
\end{proof}

\begin{definition} 
Let $\mathcal{C}$ be a presentably symmetric monoidal stable $\infty$-category. 
We say that an object of $\mack_G(\mathcal{C}) = \mack_G(\sp) \otimes
\mathcal{C}  \simeq \GSpec \otimes \mathcal{C}$ is 
\emph{$\sF$-nilpotent} (resp. \emph{$(\sF, \epsilon)$-nilpotent}) if it belongs to the
thick $\otimes$-ideal of $\mack_G(\mathcal{C})$ generated by the
$\sF$-nilpotent (resp. $(\sF, \epsilon)$-nilpotent) objects in $\GSpec$. 
\end{definition}

It follows that for any cocontinuous functor $\mathcal{C} \to \sp$, 
the induced 
functor $\mack_G(\mathcal{C}) \to \mack_G(\sp) \simeq \GSpec$ carries 
$\sF$-nilpotent (resp. $(\sF, \epsilon)$-nilpotent) objects in the source to 
$\sF$-nilpotent (resp. $(\sF, \epsilon)$-nilpotent) objects in the target. 
Using the adjunction 
$\sp \rightleftarrows \mathcal{C}$
where the symmetric monoidal left adjoint carries $\mathbb{S}$ to the unit, 
we obtain the next result.

\begin{proposition} 
\label{nilpinmot}
Let $\mathcal{C}$ be a presentably symmetric monoidal stable $\infty$-category,
and 
suppose $\mathbf{1} \in \mathcal{C}$ is compact. 
Let $A$ be an object of $\mack_G(\mathcal{C})$ which admits a unital
multiplication in the homotopy category. 
Then $A$ is $\sF$-nilpotent (resp. $(\sF, \epsilon)$-nilpotent) in
$\mack_G(\mathcal{C})$ if and only if 
it is carried to 
an
$\sF$-nilpotent (resp. $(\sF, \epsilon)$-nilpotent)
object of $\mack_G(\sp)$ under the functor $\hom_{\mathcal{C}}(\mathbf{1},-)\colon 
\mack_G(\mathcal{C}) \to \mack_G(\sp)$. 
\end{proposition} 
\begin{proof} 
Let $i^*: \mack_G( \sp) \to \mack_G(\mathcal{C})$ denote the canonical symmetric
monoidal functor (obtained from $\sp \to \mathcal{C}$), and let $i_*:
\mack_G(\mathcal{C}) \to \mack_G(\sp)$ denote its right adjoint (which is
equally obtained by the cocontinuous functor $\hom_{\mathcal{C}}( \mathbf{1}, -):
\mathcal{C} \to \sp$). 
By assumption, $i_*A$ is $\sF$-nilpotent (resp. $(\sF, \epsilon)$-nilpotent);
thus, so is $i^* i_* A$ and hence so is $A$ since our assumption of a unital
multiplication implies that $A$
belongs to the thick $\otimes$-ideal generated by $i^* i_* A$. 
\end{proof}

\begin{example} 
Let $\mathcal{A} \in \fun(BG, \cats{R})$. 
Consider $\mathcal{U}_G( \mathcal{A}) \in \mack_G(\mot_R)$, a Mackey functor valued in $\mot_R$. 
Suppose $\mathcal{A}$ is an algebra object of $\cats{R}$ (i.e., is an $R$-linear
monoidal stable $\infty$-category). 
Then $\mathcal{U}_G(\mathcal{A})$ is $\sF$-nilpotent (resp. $(\sF, \epsilon)$-nilpotent) if and only if 
the $G$-spectrum 
$K_G(\mathcal{A})$
is $\sF$-nilpotent 
(resp. $(\sF, \epsilon)$-nilpotent), using the representability of $K$-theory. 
\label{Gmotnilp}
\end{example}

\section{Descent for $p$-groups; proof of \Cref{TateMitchell} and \Cref{ASthmintro}}

In this section, we give the proof of Theorems~\ref{TateMitchell} and \ref{ASthmintro}
 via \Cref{inductivevanishingthm}.  We start with the following general reduction.

\begin{proposition} 
\label{generalreduction}
Let $R$ be an $\mathbb{E}_2$-ring, and let $j \geq 0$. Then the following
are equivalent: 
\begin{enumerate}
\item $L_{T(j)} (\Phi^{C_p} K_{C_p}(R) )= 0$.  
\item The $C_p$-spectrum $L_{T(j)} K_{C_p}(R)$ is Borel-complete.
\item For every $R$-linear idempotent-complete stable $\infty$-category $\mathcal{C}$ equipped with
an ($R$-linear) action of a finite $p$-group $G$, and every additive invariant $E$ with values in $T(j)$-local spectra, we have 
$E(\mathcal{C}^{hG}) \xrightarrow{\sim} E(\mathcal{C})^{hG}$. 
\item For every $R$-linear idempotent-complete stable $\infty$-category $\mathcal{C}$ equipped with
an ($R$-linear) action of a finite $p$-group $G$, and every additive invariant $E$, we have 
\begin{equation} L_{T(j)}(E(\mathcal{C})_{hG}) \xrightarrow{\sim} L_{T(j)} E(\mathcal{C}_{hG})\xrightarrow{\sim}L_{T(j)}E(\mathcal{C}^{hG}) \xrightarrow{\sim} L_{T(j)}( E(\mathcal{C})^{hG}) 
\xrightarrow{\sim} ( L_{T(j)} E(\mathcal{C}))^{hG}.
\label{bigcompmap} \end{equation}
\end{enumerate}
\end{proposition} 
\begin{proof} 
(1) and (2) are equivalent by \Cref{rem:k1_local_g_spectra}; (2) is the special
case of (3) where $E=L_{T(j)}K(-)$ and $G=C_p$ acts trivially on
$\operatorname{Perf}(R)$; and (3) is a special case of (4).  Thus let us show (1) implies (3)
and (3) implies (4).

First, we show (1) implies (3). 
Since every $p$-group has a composition series with successive quotients cyclic of order $p$, we can use
d\'evissage to reduce to the case when $G = C_p$. 
Let $E_{C_p}(\mathcal{C}) = E(\mathcal{C}^{\bor})$ denote the $C_p$-spectrum
obtained by applying $E$ to the $C_p$-Mackey functor $\mathcal{C}^{\bor}$ in
$\catst$. 
By construction, $E_{C_p}(\mathcal{C})$ is a module in $C_p$-spectra over
$K_{C_p}(R)$. In fact, this follows because $\mathcal{C}^{\bor}$ is a module
over $\perf(R)^{\bor}$, and $\mathcal{U}_{C_p}(\mathcal{C}) \in
\mack_{C_p}(\mot)$ is therefore a module over $K_{C_p}(R)$. 
Since $E_{C_p}(\mathcal{C})$ is $T(j)$-local, we find that
$E_{C_p}(\mathcal{C})$ is a module over $L_{T(j)} K_{C_p}(R)$ and is therefore
Borel-complete by \Cref{moduleovercomplete}. 

Finally, we show (3) implies (4). 
To this end, we will produce a sequence of $G$-spectra which we will show to be
Borel-complete, and which on fixed points 
realizes the maps in \eqref{bigcompmap}. 
In fact, 
consider the $G$-Mackey functors $\mathcal{C}^{\bor}, \mathcal{C}_{\cobor}$
with values in $\catst$; we have a natural map $\mathcal{C}_{\cobor} \to
\mathcal{C}^{\bor}$. Both are modules over $\perf(R)^{\bor}$ in
$\msk_G(\catst)$. 
Applying $E$ and then $L_{T(j)}$, we obtain a sequence of $G$-spectra 
\begin{equation} \label{seqmaps} E(\mathcal{C}_{\cobor})_{\cobor} \to E(\mathcal{C}_{\cobor}) \to
E(\mathcal{C}^{\bor}) \to E(\mathcal{C}^{\bor})^{\bor} \to (L_{T(j)}
E(\mathcal{C}^{\bor}))^{\bor}. \end{equation}
Note that all of these $G$-spectra are modules over $K_G(R)$. 
Therefore, the $T(j)$-localization of the 
$G$-spectra in \eqref{seqmaps} are modules over the $G$-ring spectrum $L_{T(j)}
K_G(R)$, which is Borel-complete by (3) (applied to the trivial $G$-action on
$\mathrm{Perf}(R)$). Consequently, in view of \Cref{moduleovercomplete}, the $T(j)$-localizations of the $G$-spectra in
\eqref{seqmaps} are all Borel-complete. Finally, all the maps of $G$-spectra in \eqref{seqmaps}
induce $T(j)$-equivalences on underlying spectra; consequently, the
$T(j)$-localizations induce equivalences on $G$-fixed points, whence the
equivalences in (4). 
\end{proof} 

For future reference, we recall also the following lemma. 

\begin{lemma}\label{triv on morava}
Let $E_i$ denote Morava $E$-theory of height $i$.  For any $T(i)$-local $\mathbb{E}_\infty$-algebra $R$ over $E_i$, we have that $E_i^{hC_p}\otimes_{E_i}R\overset{\sim}{\rightarrow} R^{hC_p}$, and this is a free $R$-module of rank $p^i$.  Here we always have $C_p$ acting trivially, and the relative tensor product is algebraic, not (a priori) $T(i)$-localized.
\end{lemma}
\begin{proof}
As $E_i$ is complex oriented and even periodic, and the $p$-series $[p](t)\in
(\pi_0 E_i)[[t]]$ of its associated formal group law is a nonzerodivisor, the
Gysin sequence for $S^1\rightarrow BC_p\rightarrow BS^1$ shows that $E_i^{hC_p}
= E_i^{BC_p}$
is also even periodic, and $\pi_0 E_i^{hC_p} = (\pi_0 E_i)[[t]]/[p](t)$.  Since
the formal group has height $i$, this is a free module of rank $p^i$ over $\pi_0 E_i$.  Since $E_i$ is
$T(i)$-local, Kuhn's Tate vanishing result from \cite{Kuh04} (or the earlier
\cite{GS96}) shows that this implies that $L_{T(i)} ((E_i)_{hC_p})$ is free of rank $p^i$.  Mapping out to an arbitrary $T(i)$-local $E_i$-module $M$, we deduce that 
$$M^{hC_p} = \hom_{E_i} ( L_{T(i)} ((E_i)_{hC_p}), M) = 
\hom_{E_i}( L_{T(i)} ((E_i)_{hC_p}), E_i) \otimes_{E_i} M = 
E_i^{hC_p} \otimes_{E_i} M
,$$ implying all the desired claims.
\end{proof}

\subsection{The case of ordinary rings}

For the proof of \Cref{inductivevanishingthm}, we will need to give an independent treatment of a special case: namely, the case where $R$ is an ordinary ring.  Note that for $n = 1$, $T(1)$ and $K(1)$-local homotopy coincide 
\cite{Ma81, Mi81}, and for all $n$ we have $L_{T(n)}A=0$ if and only if
$L_{K(n)}A=0$ whenever $A$ is a ring spectrum, thanks to the nilpotence theorem;
see \cite[Lem.~2.3]{LMMT20}.  This will let us consider $K(n)$ instead of $T(n)$.

What we will really need for the main proof is the following.

\begin{lemma}\label{K(1)localthm}
Let $R$ be a commutative ring.  Then the assembly map 
$$K(R)_{hC_p}\rightarrow K(R[C_p])$$
is a $T(n)$-equivalence for all $n\geq 1$.\footnote{The hypothesis that $R$
should be commutative is a posteriori not necessary, by \Cref{combinationAB}, but it will be
used in the proof here.}
\end{lemma} 

\begin{proof}
In fact, we will show that the assembly map is a $T(n)$-local equivalence for
$n=1$. By Mitchell's theorem (\Cref{Mithm}), $L_{T(n)}K(\mathbb{Z})=0$ for $n\geq 2$, which
implies the statement also holds when $n\geq 2$ since both sides vanish.

Thus, assume $n = 1$. 
 Since $K(A)\rightarrow K(A[1/p])$ is a $K(1)$-equivalence for all
rings $A$ (see \cite{BCM20, LMMT20, MTR20} for three different proofs), we can reduce to the case where $R$ is a $\mathbb{Z}[1/p]$-algebra.  By
transfer along the degree $p-1$ extension
$\mathbb{Z}[1/p]\rightarrow\mathbb{Z}[1/p,\zeta_p]$, we can even assume $R$ is a
$\mathbb{Z}[1/p,\zeta_p]$-algebra. 
The claim 
is equivalent to the assertion that the $C_p$-spectrum $K( \perf(R)_{\cobor})$
(obtained by applying $K(-)$ to the $C_p$-Mackey functor
$\perf(R)_{\cobor}$) has the property that $(KU \otimes K(
\perf(R)_{\cobor}))/p$
is Borel-complete. 
Equivalently, we need to show that the map 
\begin{equation} 
\left( KU
\otimes K(R) \right)_{hC_p} \to 
KU \otimes K(R[C_p])  \label{thismap} \end{equation}
induces an equivalence upon $p$-completion. 

The $p$-completion of the map \eqref{thismap} admits a retraction as $(KU \otimes K(R))_{\hat{p}}$-modules by \Cref{splittingcor}.\footnote{For this argument,
cf.~\cite{Mal17}.} 
We will show that the $p$-completions of both sides are free $(KU \otimes K(R))_{\hat{p}}$-modules of
rank $p$, which will therefore imply the claim. 
Indeed, the fact that the $p$-completion of $ \left( KU\otimes
K(R)\right)_{hC_p}$ is free of rank $p$ follows from
\Cref{triv on morava}.  Moreover, the fact that $(KU\otimes
K(R[C_p]))_{\widehat{p}}$ is free of rank $p$ follows because the
standard idempotents in the group ring give $R[C_p]\simeq R^{\times p}$ as
$R$-algebras.
This proves the claim and hence the lemma. 
\end{proof}

\begin{remark}[A proof of Mitchell's theorem] 
The above methods also reprove the vanishing $L_{T(n)}K(\mathbb{Z})=0$ for $n\geq 2$
using similar methods; it suffices to prove
$L_{K(n)} K(\mathbb{Z}) =0$ for such $n$.   By Quillen's
localization sequence $K(\mathbb{F}_p)\rightarrow K(\mathbb{Z})\rightarrow
K(\mathbb{Z}[1/p])$ and Quillen's calculation
$K(\mathbb{F}_p)_{(p)}=\mathbb{Z}_{(p)}$, it suffices to show
$L_{K(n)}K(\mathbb{Z}[1/p])=0$, or again by a transfer argument
$L_{K(n)}K(\mathbb{Z}[1/p,\zeta_p])=0$.
We now run a similar argument as above. 
Let $R = \mathbb{Z}[1/p, \zeta_p]$. The map
\begin{equation} \label{Knfreemodulesplit} L_{K(n)} (E_n \otimes K(R)_{hC_p}) \to L_{K(n)}(E_n \otimes
K(R[C_p])) \end{equation}
admits a retraction of $L_{K(n)}(E_n \otimes K(R))$-modules by
\Cref{splittingcor}. 
We showed in the proof of \Cref{K(1)localthm} that $K(R[C_p])$ is
a free $K(R)$-module of rank $p$, whence 
the right-hand-side of 
\eqref{Knfreemodulesplit}
is a free $L_{K(n)}(E_n \otimes K(R))$-module of rank $p$. 
By \Cref{triv on morava}, the left-hand-side of \eqref{Knfreemodulesplit} is a
free module of rank $p^n$. 
In particular, we obtain a split injection from a free module of rank $p^n$ over
$L_{K(n)} (E_n \otimes K(R))$ to a free module of rank $p$. As $p^n > p$, this
forces 
$L_{K(n)} (E_n \otimes K(R)) = 0$, whence the claim. 
\end{remark}

\begin{remark} 
Suppose $R$ is an associative $\mathbb{Z}/p^n$-algebra for some $n \geq 1$. 
Then the assembly map $K(R)_{hC_p} \to K(R[C_p])$ is 
a $T(0)$-equivalence as well. 
In fact, after $T(0)$-localization this map is simply the map $K(R)[1/p] \to
K(R[C_p])[1/p]$. This map admits a section given by the augmentation $R[C_p] \to
R$ which is surjective with nilpotent kernel, and induces an equivalence on
$K(-)[1/p]$ by \cite[Th.~2.25]{LT19}. We thank the referee for this remark. 
\end{remark}

\subsection{Extending to higher heights}

In this subsection, we prove 
\Cref{ASthmintro} and \Cref{TateMitchell} together and in full generality via
an inductive argument on the height.

To obtain the desired bounds on the chromatic complexity on $K(R)$,
we will use the following converse to chromatic blueshift.  
\begin{theorem}[{Cf.~\cite[Prop.~4.7]{Hahn16} and 
\cite[Th.~9.8]{BSY22}}]
\label{reverse:blueshift}
Let $A$ be an $\mathbb{E}_\infty$-ring and let $i \geq 0$. 
Suppose that $L_{T(i)} (A^{tC_p} )=0$. Then $L_{T(j)} A =0 $ for $j \geq i+1$. 
\end{theorem} 

For the convenience of the reader, we include a deduction of 
\Cref{reverse:blueshift} from the main theorem of \cite{Hahn16}.

\begin{lemma} 
\label{fflatlem}
The $\pi_0(E_{i+1})$-algebra $\pi_0(E_{i+1}^{tC_p})$ has the property that 
$\pi_0(E_{i+1}^{tC_p})/(p, v_1, \dots, v_{i-1})$ is faithfully flat 
over the field
$\pi_0(E_{i+1})/(p, v_1, \dots, v_{i-1})[v_{i}^{-1}] = k((v_i))$. 
\end{lemma}
\begin{proof} 
Note that 
$\pi_0(E_{i+1}^{tC_p})/(p, v_1, \dots, v_{i-1})$ is nonzero and has $v_{i}$ invertible, since
$E_{i+1}^{tC_p}$ has trivial $K(i+1)$-localization but nontrivial
$K(i)$-localization by \cite{GS96, HS96}. 
Therefore, the result follows. 
\end{proof}

\begin{proof} 
The vanishing results for the telescopic localizations are equivalent to those for the
analogous $K(j)$-localizations, i.e., it suffices to show that 
$L_{K(j)} A = 0$ for $j \geq i+1$ (cf.~\cite[Lem.~2.3]{LMMT20}). 
By \cite[Th.~1.1]{Hahn16}, it suffices to show that $L_{K(i+1)} A =0
$. 
Therefore, without loss of generality, we may replace $A$ with $L_{K(i+1)}
(E_{i+1} \otimes A)$ and assume that $A$ is a $K(i+1)$-local
$\mathbb{E}_\infty$-$E_{i+1}$-algebra such that $L_{K(i)} (A^{tC_p}) =0$; we then need to show that $A = 0$.

Now we have
\begin{equation}  A^{tC_p} = A \otimes_{E_{i+1}} E_{i+1}^{tC_p},  \end{equation}
by \Cref{triv on morava}.
Furthermore, 
$\pi_*(E_{i+1}^{tC_p})$ is a localization of $\pi_* (E_{i+1}^{hC_p})$ and is
therefore flat over $\pi_*(E_{i+1})$, 
whence it follows from \Cref{fflatlem} (and the K\"unneth spectral sequence) that 
the map 
$$
\pi_0(A/(p, v_1, \dots, v_{i-1})[v_{i}^{-1}]) \to \pi_0(A^{tC_p}/(p, v_1,
\dots, v_{i-1})[v_i^{-1}])$$  
is faithfully flat. 
Our assumption is that the target vanishes since $A^{tC_p}$ is
$L_{i-1}$-local. Therefore, the source vanishes and we find that $L_{K(i)} A =
0$, whence $L_{K(i+1)} A =0$ by \cite[Th.~1.1]{Hahn16}; thus $A =0$ since it is
$K(i+1)$-local, as desired. 
\end{proof}

We will also need to use some of the results of \cite{LMMT20} on
the chromatic behavior of algebraic $K$-theory. 
\begin{theorem}[{ \cite[Th.~3.8]{LMMT20}}] 
\label{LMTvanish}
Let $A$ be an $\mathbb{E}_1$-ring and let $n \geq 1$. Then the map 
$K(A) \to K(L_n^{p, f} A)$ is a $T(i)$-local equivalence for $1 \leq i \leq n$. 
\end{theorem} 

Now we get into the proofs of 
\Cref{ASthmintro} and \Cref{TateMitchell}; the following lemma, equivalent to
\Cref{inductivevanishingthm} from the introduction (thanks to \Cref{generalreduction}),  will be the key
inductive step. 

\begin{lemma} 
\label{Cpgeovanishlemma}
Let $R$ be an $\mathbb{E}_\infty$-ring and let $i \geq 1$.  For the following conditions, we have the implications (1) $\Rightarrow$ (2) $\Rightarrow$ (3):

\begin{enumerate}
\item $L_{T(i)} R =0 $ and $L_{T(i)} K(R^{tC_p}) =0 $. 
\item $L_{T(i)} \Phi^{C_p}( K_{C_p}(R)) =
0.$
\item $L_{T(j)} K(R) = 0 \quad \text{ for all } \quad j \geq i+1.$
\end{enumerate}
\end{lemma} 
\begin{proof}
In the following proof, we use the following notation: given a  $C_p$-Mackey
functor $\mathcal{M}$ with values in $\catst$, we simply write $K(\mathcal{M}),
\TC(\mathcal{M})$, etc.~for the associated $C_p$-spectrum obtained by applying
$K, \TC$, etc.
With this notation, we have that $K_{C_p}(R) = K( \perf(R)^{\bor})$; note that
our hypotheses imply that this is an $\mathbb{E}_\infty$-algebra in $C_p$-spectra.

We start by showing (1) implies (2). 
First, we reduce to the case where $R$ is connective. Indeed, 
given a coconnective spectrum $X$, the $C_p$-Tate construction $X^{tC_p}$ is
annihilated by $L_n^{p, f}$ for any
$n \geq 0$; this follows by d\'evissage and filtered colimits (note that
$(-)^{tC_p}$ commutes with filtered colimits on coconnective spectra) from the case
where $X$ is an Eilenberg--MacLane spectrum in a single degree. 
Therefore, the map $(\tau_{\geq 0} R)^{tC_p} \to R^{tC_p}$ induces an
equivalence on $L_n^{p, f}$-localizations for any $n  \geq 0$. 
It follows that $L_{T(i)} K( (\tau_{\geq 0} R)^{tC_p}) \xrightarrow{\sim}
L_{T(i)} K(R^{tC_p})$ by \Cref{LMTvanish}. Therefore, the hypotheses of (1)
hold for $R$ if and only if they hold for $\tau_{\geq 0}R $, so we may assume $R$ is
connective; note also that the conclusion of (2) for $\tau_{\geq 0}R$ implies
it for $R$. 

Now we have the categorical Mackey subfunctor 
$\perf(R)_{\cobor} \subseteq \perf(R)^{\bor}$. 
By definition, this map of categorical $C_p$-Mackey functors 
is an equivalence on underlying objects (both have underlying object of $\catst$
given by $\perf(R)$), and on $C_p$-fixed points it is given by the inclusion
$\perf(R[C_p]) \subseteq \fun(BC_p, \perf(R))$, see
\Cref{fullyfaithfulnormmap} and \Cref{assembly maps}. 
The Verdier quotient of categorical Mackey functors
$\perf(R)^{\bor}/\perf(R)_{\cobor}$ is therefore a $\catst$-valued
$C_p$-Mackey functor
with trivial underlying object and $C_p$-fixed points given by the Verdier
quotient 
$\fun(BC_p, \perf(R))/\perf(R[C_p])$, which is linear over the
$\mathbb{E}_\infty$-ring $R^{tC_p}$ (cf.~\cite[Sec.~I.3]{NS18}). 
Applying $K$-theory, 
we obtain a cofiber sequence 
of $C_p$-spectra
\[ 
K( \perf(R)_{\cobor}) \to K_{C_p}(R) \to K( \perf(R)^{\bor}/\perf(R)_{\cobor}).
  \]
If we suppose that $L_{T(i)} K(R^{tC_p})=0$, then by the above, it follows that $L_{T(i)} \Phi^{C_p}(-)$ annihilates the last
term. 
Therefore, in order to prove (2), it suffices to show that 
$L_{T(i)} \Phi^{C_p} K( \perf(R)_{\cobor}) = 0$. 

Now $K( \perf(R)_{\cobor})$ of any $\mathbb{E}_1$-algebra is 
group-ring $K$-theory, cf.~\Cref{assembly maps}. Thus, since $R$ is now
assumed connective, we may apply the Dundas--Goodwillie--McCarthy
theorem \cite{DGM13} to obtain a pullback square 
of $C_p$-spectra,
\[ \xymatrix{
K( \perf(R)_{\cobor}) \ar[d]  \ar[r] & \TC( \perf(R)_{\cobor}) \ar[d]  \\
K( \perf(\pi_0 R)_{\cobor}) \ar[r] & \TC( \perf( \pi_0 R)_{\cobor}) .  \\
}\]
Here $K( \perf( \pi_0 R)_{\cobor})$ has trivial $T(i)$-localized geometric
fixed points by \Cref{K(1)localthm}. 
Thus, to prove (2), it suffices to prove that 
$L_{T(i)} \Phi^{C_p} \TC( \perf(A)_{\cobor} ) =0 $ whenever $A$ is a
connective $\mathbb{E}_1$-ring with $L_{T(i)}A=0$ (as this holds for both $A=R$ by hypothesis and $A=\pi_0R$ trivially).

Now we use an expression for the $p$-completion of 
$\Phi^{C_p} \TC(\perf(A)_{\cobor})$ given in the work of Hesselholt--Nikolaus, 
\cite[Th.~1.4.1]{HN19}. 
Indeed, 
$\Phi^{C_p} \TC(\perf(A)_{\cobor})$
is the cofiber of the assembly map 
$\TC(A) \otimes BC_{p+} \to \TC(A[C_p])$ and \emph{loc.~cit.} shows that after
$p$-completion, this cofiber becomes $\mathrm{THH}(A;
\mathbb{Z}_p)_{h\mathbb{T}_p}[1] \otimes C_p$ for $\mathbb{T}_p$ the $p$-fold cover
of the circle $\mathbb{T}$. In particular, our assumption that $L_{T(i)} A =0$
thus implies that 
$L_{T(i)}\left(\Phi^{C_p} \TC( \perf(A)_{\cobor})\right) = 0$
as desired. 
This shows (1) implies (2).

Finally, we show (2) implies (3).
The Borel-completion of the $C_p$-spectrum $K_{C_p}(R)$ is the Borel-complete
$C_p$-spectrum associated to the trivial $C_p$-action on
$K(R)$; in particular, we have a map of $\mathbb{E}_\infty$-rings 
$\Phi^{C_p} K_{C_p}(R) \to K(R)^{tC_p}$. 
It follows from (2) that $L_{T(i)} (K(R)^{tC_p}) = 0$,
whence (3) by \Cref{reverse:blueshift}.  
\end{proof} 

We now prove 
\Cref{TateMitchell} and \Cref{ASthmintro} from the introduction, by starting with their special case \Cref{chromatic}, which we restate here:

\begin{theorem}
Let $n\geq 0$, and let $\mathcal{C}$ be an $L_n^{p,f}$-local stable $\infty$-category.  Then $L_{T(m)}K(\mathcal{C})=0$ for all $m\geq n+2$, and for any finite $p$-group $G$ acting on $\mathcal{C}$ we have
$$L_{T(n+1)}K(\mathcal{C}^{hG})\overset{\sim}{\rightarrow}(L_{T(n+1)}K(\mathcal{C}))^{hG}.$$
\end{theorem} 
\begin{proof}
Taking $R=L_n^{p,f}\mathbb{S}$, by applying \Cref{Cpgeovanishlemma} and \Cref{generalreduction} it suffices to show that $L_{T(n+1)}R=0$ and $L_{T(n+1)}K(R^{tC_p})=0$.  The first vanishing follows from the definition.  As for the second vanishing, we use induction on $n$.  When $n=0$ we have $R=\mathbb{S}[1/p]$ so $R^{tC_p}=0$.  When $n>0$, Kuhn's blueshift theorem \cite{Kuh04} shows that $R^{tC_p}$ is $L_{n-1}^{p,f}$-local, whence $K(R^{tC_p})$ is a module over $K(L_{n-1}^{p,f}\mathbb{S})$ and we conclude by induction.
\end{proof}

As a corollary of combining this theorem with
the results of \cite{LMMT20} (in particular, \Cref{LMTvanish}), one obtains the following
purity result in $T(n)$-local $K$-theory; this also appears in 
 \cite{LMMT20} and
is explored further there. 
 
\begin{corollary} 
\label{purityTnlocal}
Let $A$ be an $\mathbb{E}_1$-ring, and let $n \geq 1$. Then the map 
$A \to L_{T(n-1) \oplus T(n)} A$ induces an equivalence on $L_{T(n)}
K(-)$. 
\end{corollary} 
\begin{proof} 
By \Cref{LMTvanish}, we may assume that 
$A$ is already $L_n^{p, f}$-local. 
We have a pullback square
\[ \xymatrix{
\mathbb{K}(A) \ar[d]  \ar[r] &  \mathbb{K}( L_{T(n-1) \oplus T(n)} A ) \ar[d]  \\
\mathbb{K}( L_{n-2}^{p, f} A) \ar[r] &  \mathbb{K}( L_{n-2}^{p, f} ( L_{T(n-1) \oplus T(n)} A)),
}\]
since both vertical homotopy fibers are given by the (non-connective) $K$-theory of the thick subcategory of
$\perf(A)$ 
generated by $A \otimes F$, for $F$ a finite type $n-1$ complex. 
The result now follows from \Cref{chromatic}, which shows that the spectra
on the bottom row are $T(n)$-acyclic. 
\end{proof} 

Now we can input this back in to our arguments and obtain \Cref{TateMitchell} and \Cref{ASthmintro}, which we combine and restate here.

\begin{theorem}
\label{combinationAB}
Let $R$ be an $\mathbb{E}_\infty$-ring.  
\begin{enumerate}
\item  
Suppose $L_{T(n)} (R^{tC_p}) = 0$ for some $n \geq 0$.  Let $\mathcal{C}$ be an $R$-linear 
idempotent-complete stable $\infty$-category 
equipped with an $R$-linear action of a finite $p$-group $G$. Let $E$ be an additive
invariant of $R$-linear idempotent-complete stable $\infty$-categories. Then the natural maps induce
equivalences
$$L_{T(n+1)}(E(\mathcal{C})_{hG}) \xrightarrow{\sim} L_{T(n+1)}
E(\mathcal{C}_{hG})\xrightarrow{\sim}L_{T(n+1)}E(\mathcal{C}^{hG})
\xrightarrow{\sim} L_{T(n+1)}( E(\mathcal{C})^{hG}) 
\xrightarrow{\sim} ( L_{T(n+1)} E(\mathcal{C}))^{hG}.$$
\item
Suppose $L_{T(n+1)} R = 0$ for some $n \geq -1$. 
Then $L_{T(j)} K(R) =0 $ for $j \geq n+2$. 
\end{enumerate}
\end{theorem} 
\begin{proof} 
For (1), by \Cref{Cpgeovanishlemma} and \Cref{generalreduction} it suffices to show that $L_{T(n+1)}R=0$ and $L_{T(n+1)}K(R^{tC_p})=0$.  The first follows from \Cref{reverse:blueshift}; for the second, we also get the weaker vanishing $L_{T(n+1)}(R^{tC_p})=0$ (Hahn's theorem, \cite{Hahn16}) so this follows from the purity result \Cref{purityTnlocal}.  For (2), Hahn's theorem shows $L_{T(n+2)}R=0$ as well, so this follows from \Cref{purityTnlocal}.
\end{proof}

\begin{remark} 
The converse of part (2) of \Cref{combinationAB} is proved (for $p$-local
$\mathbb{E}_\infty$-rings) in
\cite[Th.~9.11]{BSY22} using the nonvanishing of $L_{T(n+1)} K(E_n)$ proved in
\cite{Yuan21}. 
\end{remark}

\subsection{Comparison with the redshift conjectures}
Finally, we discuss the relationship 
of our results 
to redshift. 
Conjecture 4.2 of \cite{AR} predicts that if $A \to B$ is a $K(n)$-local
$G$-Galois extension of $\mathbb{E}_\infty$-rings in the sense of
\cite{Rognes08}, then 
$L_{T(n+1)} K(A) \simeq L_{T(n+1)} ( K(B)^{hG})$. 
Here we will prove this conjecture in the case where $G$ is a $p$-group. In
fact, we will allow the 
(a priori more general) case of a $T(n)$-local $G$-Galois extension. 

We recall that the condition of being a $T(n)$-local $G$-Galois extension, in
which the map $B\otimes_AB\rightarrow \prod_G B$ need only be a
$T(n)$-equivalence, is much weaker than being a $G$-Galois extension of
underlying $\mathbb{E}_\infty$-ring spectra (also known as a ``global Galois extension"), and fundamental examples such as
the Galois extensions of the $K(n)$-local sphere produced by Devinatz--Hopkins
are only $T(n)$-locally (or $K(n)$-locally) Galois.  Thus the descent results in our previous paper \cite{CMNN} do not apply to them.  Moreover, even in the case of underlying $G$-Galois extensions our previous results required being able to verify an extra condition: the rational surjectivity of the transfer map.

In the global case, we directly obtain from \Cref{combinationAB} and Galois descent
the following. 

\begin{corollary} 
Let $A \to B$ be a faithful $G$-Galois extension of
$\mathbb{E}_\infty$-rings\footnote{The faithfulness assumption is imposed to
ensure Galois descent, in the form of \cite[Th.~9.4]{MGal} or \cite{Banerjee17}.}
with $G$ a finite $p$-group, and suppose that $L_{T(n)}(A^{tC_p}) = 0$. Then the maps 
$L_{T(n+1)} K(A) \to 
L_{T(n+1)} (K(B)^{hG}) \to 
( L_{T(n+1)} K(B))^{hG}$ are equivalences. \qed
\end{corollary} 

Now we consider the $T(n)$-local case, where we can also obtain results, 
but with an additional argument. 
Given a $T(n)$-local $\mathbb{E}_\infty$-ring $A$, we write $K'(A)$ for the
$K$-theory of the small, symmetric monoidal, stable $\infty$-category 
$\mathscr{D}(A)$
of 
dualizable objects in $T(n)$-local $A$-modules. 
We have a natural inclusion $\mathrm{Perf}(A) \subseteq \mathscr{D}(A)$, whence a
map of $\mathbb{E}_\infty$-rings $K(A) \to K'(A)$. 
The next result (together with \Cref{TateMitchell}) shows that this map is a
$T(n+1)$-equivalence and implies that $K$ or $K'$ can be used
equivalently in the
Ausoni--Rognes conjecture. 
\begin{proposition} 
\label{fibermoduleLnminusone}
Let $A$ be a $T(n)$-local $\mathbb{E}_\infty$-ring. 
Then the homotopy fiber of $K(A) \to K'(A)$ is naturally a module over $K(
L_{n-1}^{p, f}
\mathbb{S})$. 
\end{proposition} 
\begin{proof} 
Indeed, consider the Verdier quotient $\mathscr{D}(A)/\mathrm{Perf}(A)$. 
We claim that this stable $\infty$-category is naturally $L_{n-1}^{p, f}
\mathbb{S}$-linear. To this end, we need to show that if $F$ is a finite type
$n$ complex, then for any $M \in \mathscr{D}(A)$, we have $M \otimes F \in
\mathrm{Perf}(A)$ (so that this vanishes in the Verdier quotient). 
To this end, we observe that dualizability implies that the functor
\[ \hom_{L_{T(n)} \mathrm{Mod}(A)}(M, -): L_{T(n)} \mathrm{Mod}(A) \to L_{T(n)}
\mathrm{Mod}(A) \]
commutes with all colimits. Tensoring with $F$, we find that $M \otimes
F$ is a compact object of $L_{T(n)} \mathrm{Mod}(A)$; this uses that
tensoring with $F$ yields a colimit-preserving functor $L_{T(n)}
\mathrm{Sp} \to \mathrm{Sp}$. Since
$L_{T(n)} \mathrm{Mod}(A)$ is compactly generated by $A \otimes F$, it follows
that $M \otimes F$ belongs to the thick subcategory generated by $A \otimes F$
and is therefore a perfect $A$-module, whence the result. 
\end{proof} 

Given a $T(n)$-local $G$-Galois extension $A \to B$ with $G$ a finite group, 
we have $\mathscr{D}(A) \simeq \mathscr{D}(B)^{hG}$; this follows because 
$L_{T(n)} \mathrm{Mod}(A) \simeq \left(L_{T(n)} \mathrm{Mod}(B)\right)^{hG}$ by
Galois descent,\footnote{Note that $A \to B$ is automatically $T(n)$-locally
faithful. In fact, $A \simeq B^{hG} \simeq L_{T(n)}B_{hG}$, so tensoring with $B$ is
conservative on $L_{T(n)} \mathrm{Mod}(A)$.} and using \cite[Prop.~4.6.1.11]{Lur17} to
commute the formation of dualizable objects over homotopy fixed points. 
Therefore, the next result follows in a similar manner; this proves
\cite[Conj.~4.2]{AR} in the case of a $p$-group. 

\begin{corollary} 
Let $ A \to B$ be a $T(n)$-local $G$-Galois extension, with 
$G$ a finite $p$-group. Then $L_{T(n+1)} K(A) \xrightarrow{\simeq}
L_{T(n+1)} ( K(B)^{hG}) \xrightarrow{\simeq}
(L_{T(n+1)}
K(B))^{hG}$. 
\label{ARpgroup}
\end{corollary} 

Let $G$ be a finite $p$-group 
acting on a $T(n)$-local $\mathbb{E}_\infty$-ring $B'$; then the map $A'
\stackrel{\mathrm{def}}{=} B'^{hG} \to B'$ is $T(n)$-locally $G$-Galois,
cf.~\cite[Cor.~7.31]{BCSY}. 

\begin{proof} 
In fact, \Cref{ASthmintro} yields the analog of this result with $K'(-)$
replacing $K$, since $\mathscr{D}(-)$ satisfies $T(n)$-local Galois descent. 
Using \Cref{fibermoduleLnminusone}, we find that the difference between the
statements for $K'(-)$ and $K(-)$ is controlled by modules over $K( L_{n-1}^{p,
f}
\mathbb{S})$, which have trivial $T(n+1)$-localizations by \Cref{TateMitchell}. 
\end{proof} 

\begin{example} 
An important example of a $T(n)$-local (pro-)Galois extension is the 
map $L_{K(n)} \mathbb{S} \to E_n$, where the profinite group in question is the
extended Morava stabilizer group $\mathbb{G}_n$. 
We have a short exact sequence
$$ 1 \to \mathbb{S}_n \to \mathbb{G}_n \to \mathrm{Gal}(
\overline{\mathbb{F}}_p/\mathbb{F}_p) \to 1,$$
where $\mathbb{S}_n$ has an open subgroup which is pro-$p$. 
For each open subgroup $H \subseteq \mathbb{G}_n$, we write $E_n^{hH}$ for the
(Devinatz--Hopkins) continuous homotopy fixed points, \cite{DH04}. 
Now, we choose an open subgroup $U$ of $\mathbb{G}_n$ such that $U \cap
\mathbb{S}_n$ is a pro-$p$-group. 
Then for any normal inclusion $V' \trianglelefteq V \subseteq U$ of open
subgroups, we have
$L_{T(n+1)} K( E_n^{hV}) \xrightarrow{\sim} 
(L_{T(n+1)} K( E_n^{hV'}))^{h(V/V')}$, i.e., we obtain Galois descent 
for the $K(n)$-local finite Galois extensions of $E_n^{hU}$, and we therefore
obtain a sheaf of
$T(n+1)$-local spectra on finite continuous $U$-sets. 
This follows from the descent for $p$-groups proved above as well as the
descent for finite \'etale extensions, proved as in \cite{CMNN}. 
Our methods do not (to the best of our knowledge) yield hyperdescent for this
sheaf, which would closely relate $L_{T(n+1)} K(E_n)$ and $L_{T(n+1)}
K(L_{K(n)} \mathbb{S})$. 
\end{example}

\begin{remark} 
Finally, \cite[Conjecture~4.3]{AR} predicts that for appropriate $K(n)$-local
$\mathbb{E}_\infty$-ring spectra $B$ (e.g., $L_{K(n)} \mathbb{S}$), and for a
finite type $n+1$-complex $V$, the
map 
$V \otimes K(B) \to L_{T(n+1)} (V \otimes K(B))$ is an equivalence in high
enough degrees; this is a higher chromatic analog of the Lichtenbaum--Quillen
conjecture, cf.~\cite{AR02, Ausoni10, HW20, AKACHR22, HRW22} for instances where such
statements are proved. 
Our methods are certainly not strong enough to prove such statements; however,
this conjecture would imply 
the weaker assertion $L_{T(n+i)} K(B) = 0$ for $i \geq 2$, which we have proved above
as \Cref{TateMitchell}. 
\end{remark} 
\section{Descent by normal bases; proof of \Cref{normalbasisthm}}
In this section, we will give another condition that guarantees $T(n)$-local
descent, which will work uniformly for all $n$ (including $n = 0$). 

\begin{construction}[The transfer] 
We use the transfer map of the finite group $G$, which is a map of spectra $$\tr_{\BG} \colon BG_+ \to \mathbb{S}.$$
The adjoint map of anima 
$BG \to \Omega^\infty \mathbb{S}$
arises from interpreting the target as the $K$-theory of the category
$\mathrm{Fin}$ of finite sets (the
Barratt--Priddy--Quillen theorem), and considering the $G$-action on the
$G$-set $G$ (by right multiplication), so we take the composite map $BG \to \mathrm{Fin}^{\simeq} \to
\Omega^\infty K(
\mathrm{Fin}) = \Omega^\infty \mathbb{S}$. 
Our basic tool will be 
the following observation:
\end{construction} 

\begin{theorem} 
For any $n \geq 0$ and implicit prime $p$, 
the map $L_{T(n)} (\tr_{\BG}) \colon L_{T(n)} BG_+ \to L_{T(n)} \mathbb{S}$ admits a
section. 
\end{theorem} 
\begin{proof} 
This follows from (and is equivalent to, as explained in \cite{CM17}) the
vanishing of Tate spectra in the $T(n)$-local category, due to Kuhn
\cite{Kuh04}. In fact, this vanishing yields that $L_{T(n)} BG_+
\xrightarrow{\sim} C^*(BG, L_{T(n)} \mathbb{S})$ via the norm map, and the
transfer is the composite of the norm with the projection $C^*(BG, L_{T(n)}
\mathbb{S} ) \to L_{T(n)} \mathbb{S}$, which clearly admits a
section.\footnote{Note also that as explained in \cite{CM17}, the existence of the section in the
essential case $G = C_p$ follows via the Bousfield--Kuhn functor 
\cite{Bou01, Kuh89}
from 
the Kahn--Priddy theorem \cite{KP78}, which states that
$\Omega^{\infty+1}( \tr_{\BG}) \colon \Omega^{\infty +1} BG_+ \to
{\Omega^{\infty+1}}\mathbb{S}$
has a
section.} 
\end{proof}

\begin{proposition} 
\label{BGcrit}
Let $R$ be an associative algebra in $\GSpec$. Suppose that 
there is a factorization of $BG_+ \xrightarrow{\tr_{\BG}} \mathbb{S} \to R^G$ through 
the $R$-transfer $R_{hG} \to R^G$. Then 
$R$ is $(\triv, \epsilon)$-nilpotent. 
\end{proposition} 
\begin{proof} 
By \Cref{crittrivnilpepsilon}, it suffices to show that $L_{T(n)} (R^G/R_{hG})
=L_{T(n)} 
(R \otimes \widetilde{EG})^G = 
0$ for any $p$ and $n$
(including $p = 0$). 
For this, it suffices to show that the map 
\[ L_{T(n)}(R_{hG}) \to L_{T(n)}(R^G)  \]
has image on $\pi_0$ containing the unit. 
But this follows because we have seen above that 
\( L_{T(n)} ( \tr_{\BG})  \colon L_{T(n)}( BG_+) \to L_{T(n)} \mathbb{S}  \)
has image containing the unit. 
\end{proof} 

We will apply this below to associative $G$-ring spectra of a particularly special kind, 
where one has a homotopy commutative diagram
\begin{equation}  \xymatrix{
R_{hG} \ar[r] &  R^G \\
BG_+ \ar[u]\ar[r]^{\tr_{\BG}} & \mathbb{S}, \ar[u]
} \end{equation}
in which the factorization $BG_+ \to R_{hG}$ required in \Cref{BGcrit} is obtained as the $G$-homotopy orbits of the 
map $\mathbb{S} \to R$; in particular, these satisfy the
conditions of \Cref{BGcrit}. 
Note that this now is merely a condition on the algebra $R$ in $\sp_G$, namely that the diagram
\begin{equation}  \xymatrix{
R_{hG} \ar[r] &  R^G \\
BG_+ \ar[u]^{(\eta)_{hG}}\ar[r]^{\tr_{\BG}} & \mathbb{S}, \ar[u]^\eta
} \end{equation}
should commute up to homotopy, where $\eta$ denotes the unit map.
Such $R$ arise via the following categorical construction, namely by taking
 $R=K_G(\mathcal{C})$ below.

Let $(\mathcal{C}, \otimes, \mathbf{1})$ be a monoidal, stable $\infty$-category equipped with a
$G$-action. 
Let $f \colon G \to \ast$ be the map of $G$-sets. 
We use the induction functor 
$f_* \colon \mathcal{C} \to \mathcal{C}^{hG}$ (biadjoint to the forgetful functor
$\mathcal{C}^{hG} \to \mathcal{C}$); we note that this is $G$-equivariant
with respect to the trivial action on the target. 
Since $\mathbf{1} \in \mathcal{C}$ is $G$-invariant, we obtain a $G$-action
on $f_*(\mathbf{1}) \in \mathcal{C}^{hG}$.

\begin{definition}[The normal basis condition] 
We say that the $G$-action on $\mathcal{C}$ as above satisfies the \emph{normal basis
property} if the object
$f_*(\mathbf{1}) \in \fun(BG, \mathcal{C}^{hG})$ defines the same $K_0$-class as
the object $\mathbf{1}_{\mathcal{C}^{hG}} \otimes G_+ \in 
\fun(BG, \mathcal{C}^{hG})$. 
\end{definition} 

In other words, the normal basis condition implies that 
the following diagram, which is \emph{not} commutative,
\begin{equation} \label{updiagram} \xymatrix{
\mathcal{C}^{\simeq}_{hG} \ar[r]^{f_*} & 
(\mathcal{C}^{hG})^{\simeq}  \\
\ar[u]^{[\mathbf{1}]_{hG}}
BG   \ar[r]^{\tr_{\BG}} &  \Omega^\infty \mathbb{S},
\ar[u]^{\mathbf{1}_{\mathcal{C}^{hG}}}
}\end{equation}
gives rise to two objects in $\fun(BG, \mathcal{C}^{hG})$ with the same
$K_0$-class; the class obtained by going right and up is the normal basis class
$\mathbf{1}_{\mathcal{C}^{hG}} \otimes G_+$, while the class obtained by going
up and right is $f_*(\mathbf{1})$. 
It follows that we do have a homotopy commutative diagram
if we replace the top right 
in \eqref{updiagram}
with $\Omega^\infty K(\mathcal{C}^{hG})$.

We now prove the following result, which 
is a slight refinement of \Cref{normalbasisthm}. 

\begin{theorem} 
\label{normalbasisdescthm}
Suppose $R$ is an $\mathbb{E}_\infty$-ring, $\mathcal{C}$ is an algebra object of $\cats{R}$ equipped with an
action of a finite group $G$, and the $G$-action on $\mathcal{C}$ satisfies the normal basis
property. 
Then $\mathcal{U}_G(\mathcal{C}) \in \mack_G(\mot_R)$ is $(\triv, \epsilon)$-nilpotent. 
In particular, for any additive invariant $E$ on $\cats{R}$, the map 
$E(\mathcal{C}^{hG}) \to E(\mathcal{C})^{hG}$ induces an equivalence after
$T(i)$-localization for any $i$ and any implicit prime $p$, including $p=0$.

\end{theorem} 
\begin{proof} 
We show that $\mathcal{U}_G(\mathcal{C}) \in \mack_G(\mot_R)$ is $(\triv, \epsilon)$-nilpotent,
which also implies the other claims.
By \Cref{Gmotnilp}, it suffices to show that $K_G(\mathcal{C})$ is $(\triv,
\epsilon)$-nilpotent. 
By \Cref{BGcrit}, it suffices to show that we have a factorization 
of the $BG_+ \to \mathbb{S}$ through $K(\mathcal{C})_{hG} \to
K(\mathcal{C}^{hG})$. However, this follows from the diagram \eqref{updiagram}
(which is not commutative, but which becomes homotopy commutative 
when we replace the upper right by $\Omega^\infty K(\mathcal{C}^{hG})$ by our
hypotheses). 
\end{proof} 

\begin{remark}[Alternative proof of \Cref{galoisdescthm}] \label{rem:alternate_argument}
Let $R \to R'$ be a $G$-Galois extension of commutative rings. 
Then Zariski locally on $R$, one has the normal basis property (even before
passage to $K_0$):  the
$R[G]$-module $R'$ is locally isomorphic to $R[G]$; indeed, this follows
because of the usual normal basis theorem when $R$ is a field, and hence more
generally a local ring. 
Using Zariski descent for $K$-theory \cite{TT90}, one reduces to this case,
whence the result via \Cref{normalbasisdescthm}. 
\end{remark} 

\begin{remark} 
In fact, the above argument for \Cref{galoisdescthm} is valid more generally
(with $L_{T(n)}$-localization for any $n$) if
$R \to R'$ is a $G$-Galois extension of $\mathbb{E}_\infty$-ring spectra in the sense of
\cite{Rognes08} in the case where $\pi_0(R) \to \pi_0(R')$ is additionally $G$-Galois, i.e., $R
\to R'$ is \'etale in the sense of \cite[Sec.~7.5]{Lur17}.
This is a special case of the results 
of \cite{CMNN}, which assume a weaker condition on $R \to R'$, but
which essentially
use the $\mathbb{E}_\infty$-structures on the algebras in question.
\end{remark} 

\section{Swan induction and applications; proofs of Theorem~\ref{swandescintro}
and \ref{swanvanishintro}}

In this section, we recall 
the notion of Swan $K$-theory, 
and prove Theorems~\ref{swandescintro} and \ref{swanvanishintro} from the
introduction. In the final section, we will give a number of examples of Swan
induction. 

The {Swan $K$-theory} of a ring spectrum with respect to a finite group was introduced by Malkiewich in \cite{Mal17}, following ideas of Swan
\cite{Swa60, Swa70} who defined it for discrete rings. 
Throughout the subsection, let $R$ be an $\mathbb{E}_\infty$-ring spectrum. 

\begin{definition}[{\cite[Def.~4.11]{Mal17}}]\label{def:swan_green} 
Given a finite group $G$, we let $\sw(G, R)$ denote 
the Grothendieck ring of the stable $\infty$-category $\fun(BG,
\perf(R))$. We will call this the \emph{Swan $K$-theory of $R$ with respect to $G$}.
The groups $\{\sw(H, R)\}_{H
\subseteq G}$ 
form a {Green functor}, as the $\pi_0$ of the $\mathbb{E}_\infty$-algebra $K_G( R)$ in $\GSpec$. 

For a family $\sF$ of subgroups of $G$, we will say that \emph{$\sF$-based
Swan induction holds for $R$} if there exist classes $x_{H}
\in \sw(H, R) \otimes \mathbb{Q}$ for $H \in \sF$ such that
\begin{equation} 
	1  = \sum_{H \in \sF} \mathrm{Ind}_H^G(x_H) \in \sw(G, R) \otimes \mathbb{Q},
\end{equation} 
for $\mathrm{Ind}_H^G \colon \sw(H, R) \to \sw(G, R)$ the map obtained by induction
of representations on $R$-modules.
\end{definition}

\begin{remark} 
\label{SwaninductionandFnilp}
The condition that $\sF$-based Swan induction holds for $R$ is precisely the
condition that $K_G(R) \otimes \mathbb{Q} \in \GSpec$ is $\sF$-nilpotent, in
light of \Cref{crittrivnilp}. 
\end{remark}

\begin{example}[Classes in $\rep(G, R)$] 
Let $M$ be a finite $G$-CW complex. 
Then  $R \otimes M_+$ defines an object of 
$\fun(BG, \perf(R))$ and consequently an element $[R \otimes M_+] \in \rep(G, R)$. 
If $M$ has the homotopy type of a $G$-CW complex,
then 
a cell decomposition shows that 
the class $[R \otimes M_+]$ actually belongs to the image of the map $A(G) \to
\rep(G, R)$, for $A(G)$ the Burnside ring. 
\end{example} 

We now prove the descent statement in the $K$-theory of $R$-linear $\infty$-categories
that Swan induction implies (this is \Cref{swandescintro}); the use of a
rational statement to deduce telescopic ones follows \cite{CMNN}. 

\begin{theorem}[Descent via Swan induction]
\label{swaninddesc}
Let $R$ be an $\mathbb{E}_\infty$-ring and let $G$ be a finite group. 
Suppose that $R$-based Swan induction holds for the family $\sF$. 
Then for any $R$-linear $\infty$-category $\mathcal{C}$ equipped with a
$G$-action, and for any additive invariant $E$ on $\cats{R}$, the
maps 
\begin{equation} \label{Fassembly1} E(\mathcal{C}^{hG}) \to \varprojlim_{G/H \in \sOFG^{op}} E(\mathcal{C}^{hH})
\end{equation} and 
\begin{equation} \varinjlim_{G/H \in \sOFG} E(\mathcal{C}_{hH}) \to
E(\mathcal{C}_{hG})  \label{Fassembly2} \end{equation}
become an equivalence after $T(n)$-localization, for any $n$ and any implicit prime $p$.
\end{theorem} 
\begin{proof} 
For the first claim, it suffices to show that 
$\mathcal{U}_G(\mathcal{C}) \in \mack_G( \mot_R)$ is $(\sF,
\epsilon)$-nilpotent. By multiplicativity, this reduces to showing that 
$\mathcal{U}_G( \perf(R))$ is $(\sF, \epsilon)$-nilpotent, where the $G$-action
on $\perf(R)$ is trivial; for this in turn, it suffices to show that $K_G(R) \in
\GSpec$ is $(\sF, \epsilon)$-nilpotent  as in \Cref{Gmotnilp}. 
By 
\Cref{rationaltoFnilp}, it suffices to show that $K_G(R)_{\mathbb{Q}}$ is $\sF$-nilpotent, which is
precisely the condition of $R$-based Swan induction for $\sF$
(\Cref{SwaninductionandFnilp}).

For the second claim, we use the coBorel construction of 
\Cref{equivalgKcobor}. 
We claim that $\mathcal{U}_{G, \cobor}(\mathcal{C}) \in \mack_G(\mot_R)$ is $(\sF,
\epsilon)$-nilpotent. But this follows because it is a module over
$\mathcal{U}_G(\perf(R))$, which we have just seen is $(\sF,
\epsilon)$-nilpotent; thus $\mathcal{U}_{G, \cobor}(\mathcal{C})$ is $(\sF,
\epsilon)$-nilpotent. This implies that 
\eqref{Fassembly2} becomes an equivalence after $T(n)$-localization,
cf.~\Cref{Fepsilonimpliescompmap}. 
\end{proof} 

Now we record
a variant of \Cref{swaninddesc} specifically in the context where $R = L_n^{p, f}
\mathbb{S}$, and 
where the localization is precisely at height $n+1$; this relies on similar
techniques as in section~4, and follows from combining them with the above. In
fact, this yields a slight refinement of the results of section 4 to
non-$p$-groups.  
\begin{theorem}  
\label{ASthmwithcyclic}
Fix $n \geq 0$. Let $R$ be an $\mathbb{E}_\infty$-ring such that $L_{T(n)}
(R^{tC_p}) =0 $. 
Let $\mathcal{C}$ be an $ R$-linear idempotent-complete stable $\infty$-category equipped
with an $R$-linear action of a finite group $G$. Then for any additive invariant $E$, the
maps
\eqref{Fassembly1} and \eqref{Fassembly2} become equivalences after
$T(n+1)$-localization, for $\sF$ the family of cyclic subgroups of $G$ of
prime-to-$p$-power order. 
\end{theorem} 
\begin{proof} 
Let us first observe that we may replace $R$ by its connective cover $\tau_{\geq
0} R$. Indeed, via the Postnikov tower, we find that $(\tau_{\leq -1} R)^{tC_p}$ is
annihilated by $T(n)$. Thus, the vanishing assumption for $R$ is equivalent to
the same assumption for $\tau_{\geq 0} R$, whence we may assume for the rest of
the argument that $R$ is connective.

Now observe that for any finite group $H$, we have $L_{T(n)} (R^{tH}) =
0$. Indeed, we reduce to the case where $H$ is a $p$-group by restricting to a
$p$-Sylow. 
Inductively, we have a normal subgroup $C_p\trianglelefteq H$. The norm map $R_{hH} \to
R^{hH}$ factors as 
\[ R_{hH} = (R_{hC_p})_{h(H/C_p)} \to (R^{hC_p})_{h(H/C_p)} \to
(R^{hC_p})^{h(H/C_p)} = R^{hH}.  \]
By induction on $H$ and the assumption $L_{T(n)}(R^{tC_p}) =0$, we see that each map above has $T(n)$-acyclic cofiber, whence $R^{tH}$
is $T(n)$-acyclic.

The rest of the argument will closely follow that of \Cref{Cpgeovanishlemma}. 
We will  show that 
$L_{T(n+1)}K_G(  R)$ is $\sF$-complete, for $\sF$ as in the
statement. This will imply the result. 
Indeed, for any additive invariant $E$, the $G$-spectra 
$E(\mathcal{C}^{\bor}), E(\mathcal{C}_{\cobor}), F(E\sF_+,
E(\mathcal{C}^{\bor}))$ are modules over $K_G( R)$. 
If $L_{T(n+1)} K_G(R)$ is shown to be $\sF$-complete, then 
the $G$-spectra 
$$L_{T(n+1)}E(\mathcal{C}^{\bor}), \ L_{T(n+1)}E(\mathcal{C}_{\cobor}), \ 
L_{T(n+1)}F(E\sF_+,
E(\mathcal{C}^{\bor}))
$$ 
will be $\sF$-complete by \Cref{moduleovercomplete} and thus become
$\sF$-nilpotent after smashing with a finite type $(n+1)$-spectrum 
by \Cref{rem:k1_local_g_spectra}. The result then follows in light of
\cite[Prop.~2.8]{MNN19}. 

To see that $L_{T(n+1)}K_G( R)$ is $\sF$-complete, 
we let $\mathcal{D} = \perf( R)$ with trivial $G$-action. 
We have the fully faithful inclusion of $G$-Mackey functors 
$\mathcal{D}_{\cobor} \to \mathcal{D}^{\bor}$ (\Cref{fullyfaithfulnormmap}); the cofiber
takes values in $R^{tH}$-linear $\infty$-categories for various subgroups $H
\subseteq G$; indeed, this follows because the Verdier quotient
$\fun(BH, \perf(  R))/ \perf(  R[H])$ is linear over
$ R^{tH}$ as in \cite[Sec.~I.3]{NS18}. 
Therefore, using \Cref{TateMitchell}, we find an equivalence of $G$-Mackey functors 
$L_{T(n+1)} K( \mathcal{D}_{\cobor}) \simeq L_{T(n+1)} K( \mathcal{D}^{\bor})
= L_{T(n+1)} K_G(  R)$. 
Thus, it suffices to show that 
$L_{T(n+1)} K( \mathcal{D}_{\cobor})$ is $\sF$-complete, or equivalently that
its $T(n+1)$-local geometric fixed points vanish at all subgroups 
except possibly those 
which are cyclic of order prime to $p$ (\Cref{rem:k1_local_g_spectra}).

Now 
$K( \mathcal{D}_{\cobor})$ is group-ring $K$-theory (\Cref{assembly maps}). 
Since $R$ is connective, we obtain from the Dundas--Goodwillie--McCarthy theorem \cite{DGM13} 
a pullback square
of $G$-spectra,
$$ \xymatrix{
L_{T(n+1)} K( \mathcal{D}_{\cobor}) \ar[d] \ar[r] &   L_{T(n+1)} \mathrm{TC}( 
\perf(  R)_{\cobor}) \ar[d] 
\\
L_{T(n+1)} K( \perf( \pi_0  R)_{\cobor}) \ar[r] & L_{T(n+1)}
\mathrm{TC}( \perf( \pi_0  R)_{\cobor}).
}$$
Now we obtain from \cite[Th.~1.2]{LRRV19} that the $G$-spectra 
$\mathrm{TC}( 
\perf(  R)_{\cobor}), 
\mathrm{TC}( \perf( \pi_0 R)_{\cobor})$ are 
modulo $p$ induced from the family of cyclic subgroups of $G$; in particular,
their geometric fixed points at non-cyclic subgroups vanish modulo $p$. 
Finally, the term 
$L_{T(n+1)} K( \perf( \pi_0  R)_{\cobor})$ vanishes for $n \geq 1$
by Mitchell's theorem; if $n = 0$, it follows from \Cref{swaninddesc} and
Swan's induction theorem from \cite{Swa60} (reproved below as \Cref{swanthm}) that 
$L_{T(n+1)} K( \perf( \pi_0  R)_{\cobor})$
is 
induced from the family of cyclic subgroups. 

Thus, we find that $L_{T(n+1)}K(\mathcal{D}_{\cobor}) = 
L_{T(n+1)} K(\mathcal{D}^{\bor}) = 
L_{T(n+1)} K_G(R)$ is complete for the family of 
cyclic subgroups. 
In particular, the $T(n+1)$-local geometric fixed points vanish for 
non-cyclic subgroups. 
Suppose then that $G$ is cyclic and has order divisible by $p$; we must show that 
$L_{T(n+1)} \Phi^G K(\mathcal{D}^{\bor}) = 0$. 

In fact, since there is an inclusion $H \trianglelefteq G$ with $G/H \simeq
C_p$, 
we have the transfer map 
\[ (K(\mathcal{D}^{\bor})^{H})_{hC_p} \to K(\mathcal{D}^{\bor})^G.  \]
This map,
or equivalently $K( \mathcal{D}^{hH})_{hC_p} \to K(\mathcal{D}^{hG})$,
is $T(n+1)$-locally an equivalence thanks to \Cref{combinationAB} (applied to
the residual $C_p$-action on $\mathcal{D}^{hH}$). 
Since it factors through the map 
$(E \mathscr{P} \otimes K(\mathcal{D}^{\bor}))^{G} \to
K(\mathcal{D}^{\bor})^G$ for $\mathscr{P}$ the family of proper subgroups, it follows that this last map has $T(n+1)$-local
image (on $\pi_0$) containing the unit, whence 
$L_{T(n+1)} \Phi^G K(\mathcal{D}^{\bor}) = 0$
as desired. 
\end{proof}

Next, we prove \Cref{swanvanishintro}, which was
inspired by the generalized character theory of Hopkins, Kuhn, and Ravenel
\cite{HKR00} as well as the results of \cite{MNN19}. We will apply this below to
recover some cases of 
the chromatic bounds on $K$-theory spectra.

\begin{prop}\label{thm:bdd-complex}
	Fix a prime $p$ and a non-negative integer $n$. Let $R$ be an
	$\mathbb{E}_\infty$-ring spectrum and $G=C_p^{\times n}$. Suppose that the sum of the rationalized transfer maps
	\begin{equation}\label{eqn:tr}
	 \bigoplus_{H\subsetneq G} R^0(BH)\otimes \mathbb{Q} \to R^0(BG)\otimes \mathbb{Q}
	 \end{equation}
	  is a surjection (or equivalently has image containing the unit). Then
	  $L_{T(n+i)} R\simeq 0$ for all $i\geq 0$ (at the prime $p$). 
\end{prop}
\begin{proof}
A $T(n)$-local ring spectrum is
contractible if and only if its $K(n)$-localization is contractible,
cf.~\cite[Lem.~2.3]{LMMT20}. Therefore,
it suffices to prove that $L_{K(n+i)} R = 0$. 
To verify the desired vanishing, we can replace $R$ by the
$\mathbb{E}_\infty$-$R$-algebra $L_{K(n+i)}(E_{n+i}
\otimes R)$; by
naturality of the transfer map, the hypotheses of the result are preserved by
this replacement. 
Thus, we may
assume throughout that $R=L_{K(n+i)}R$ receives an $\mathbb{E}_\infty$-map from $E_{n+i}$. 
By the main result of 
\cite{MNN_nilpotence} then, it suffices to show $\pi_0 R\otimes \mathbb{Q}=0$.

When $n=0$, the left hand side of \eqref{eqn:tr} is 0, so $\pi_0 R\otimes \mathbb{Q}=0$. So it suffices to consider the case $n>0$.	Since each transfer map factors through a maximal proper subgroup, the surjectivity of \eqref{eqn:tr} is equivalent to the surjectivity of 
\begin{equation}\label{eqn:better-tr}
	\bigoplus_{C_p^{\times (n-1)}\cong H\subsetneq G} R^0(BH)\otimes \mathbb{Q}\to R^0(BG)\otimes\mathbb{Q}.
\end{equation}
The left hand side of this equation contains $\frac{p^n-1}{p-1}$-copies of $R^0(BC_p^{\times (n-1)})$. 
Since $R$ is a $K(n+i)$-local $E_{n+i}$-module, it follows that for any $k$, one
has that
$C^*(BC_p^{\times k}, R) = C^*(BC_p^{\times k}, E_{n+i}) \otimes_{E_{n+i}} R$
is a free $R$-module of rank 
$p^{(n+i)k}$ (cf.~\Cref{triv on morava}). 
In particular, the right-hand-side of \eqref{eqn:better-tr} is free 
over $\pi_0(R) \otimes \mathbb{Q}$
of rank $p^{(n+i)n}$, while each summand on the left-hand-side has rank 
$p^{(n+i)(n-1)}$. 
Using the surjectivity of \eqref{eqn:better-tr}, 
we find that if $\pi_0(R) \otimes \mathbb{Q} \neq 0$, then we would conclude the inequality of ranks,
\[ p^{(n+i)(n-1)} \frac{p^n-1}{p-1} \geq p^{(n+i)n}.  \]
However, we see easily that this inequality cannot hold if $i \geq 0$ and $n>0$. 
This contradiction proves the result. 
\end{proof}

\begin{thm}\label{thm:swan-vanish} 
Let $p$ be a prime, $n\ge 0$ and $R$ an $\mathbb{E}_\infty$-ring spectrum. 
Suppose that $R$-based Swan induction holds for the family of proper subgroups of $C_p^{\times n}$. 
Then $L_{T(i)} K(R) = 0$ for $i \geq n$ at the prime $p$.
\end{thm}

\begin{proof}
We write $G=C_p^{\times n}$ and consider the $\mathbb{E}_\infty$-algebra in
$G$-spectra $K_G(R)$. By assumption, $K_G(R)_{\mathbb{Q}}$ is nilpotent for the family of
proper subgroups, cf. \Cref{SwaninductionandFnilp}. 
There is a natural map of $\mathbb{E}_\infty$-algebras in $\GSpec$ of the form 
$K_G(R)_{\mathbb{Q}} \to \left( K_G(R)^{\bor} \right)_{\mathbb{Q}}$, so the
target is also nilpotent for the family of proper subgroups. 
But $K_G(R)^{\bor}$ is simply the Borel-equivariant $G$-spectrum associated to
the trivial $G$-action on $K(R)$, so the condition that 
$\left( K_G(R)^{\bor} \right)_{\mathbb{Q}}$ should be nilpotent for the family
of proper subgroups is exactly that the map \eqref{eqn:tr} (with $K(R)$ replacing $R$) should be a surjection. 
Thus, the result follows 
from \Cref{thm:bdd-complex}. 
\end{proof}

\section{Swan induction theorems; proof of \Cref{Ex:Swaninduction}}

In this section, we establish several examples of Swan induction theorems for structured
ring spectra, and prove \Cref{Ex:Swaninduction}. 
In particular, we show that one always has Swan induction for the family of
abelian subgroups for $MU$ (\Cref{cxorientswan}), 
for the cyclic groups for $\mathbb{Z}$ 
(\Cref{swanthm}, recovering results of \cite{Swa60})
or for $\mathbb{S}[1/|G|]$ (\Cref{p-1local}), 
for the rank $\leq 2$ abelian subgroups for $KU$ (\Cref{artinindKU}), and for the rank $\leq
n+1$ abelian subgroups for $E_n$ at $p =2 $ (\Cref{thm:2-primary}). 
\subsection{Geometric arguments}

Throughout, let $R$ be an $\mathbb{E}_\infty$-ring spectrum. 
We first observe the following basic features of the Swan induction property. 

\begin{remark} 
\label{general:swaninduction}
\begin{enumerate}
\item  
If $R$ is an $\mathbb{E}_\infty$-ring such that 
one has 
$R$-based Swan induction with respect to a family of subgroups $\sF$ of some
group $G$, and $R'$ is an $\mathbb{E}_\infty$-ring admitting a map from $R$
(even an $\mathbb{E}_1$-map suffices), 
then $R'$-based Swan induction for $\sF$ and $G$ holds as well (this was used
in the proof of \Cref{thm:bdd-complex}). 
\item In order to prove that $R$-based Swan induction holds with respect to a
family $\sF$ of subgroups  of $G$, it suffices to show that for every subgroup
$H \subseteq G$ which is not in $\sF$, then one has Swan induction with respect to
the family of proper subgroups of $H$.  This is an elementary observation about Green 
functors, cf. \cite[Prop.~6.40]{MNN17}.
\item Suppose $G \twoheadrightarrow G'$ is a surjection, and $R$-based Swan
induction holds for the family of proper subgroups of $G'$. Then $R$-based Swan
induction holds 
for the family of proper subgroups of $G$. 
\end{enumerate}
\end{remark}

In this subsection, we give 
geometric proofs of Swan induction in several cases. 
Our basic tool is the following. 

\begin{prop} \label{artinworkhorse}
Suppose $M$ is a $G$-space such that: 
\begin{enumerate}
\item\label{item:CW_structure} M admits  a finite $G$-CW structure.
	\item  $M$ has isotropy in the family $\sF$ of subgroups of $G$.
	\item\label{item:trivialization} 
	There is an equivalence \begin{equation} \label{trivaction}R \otimes M_+ \simeq \bigoplus_{k = 1}^n
	\Sigma^{2i_k} R  \in \fun(BG, \perf(R))\end{equation}
	for some integers $i_1, \dots, i_n$. Here we equip the $\Sigma^{2i_k} R$ 
	with the trivial $G$-action.\footnote{In fact, for the argument, it suffices
	that the class in $\mathrm{Rep}(G, R)$ of  $R \otimes M_+$ is a nonzero
	integer.
	This would be satisfied, for example, if 
there are odd suspensions of $R$ that appear in \eqref{trivaction}, as long as
the Euler characteristic is nonzero.} 
\end{enumerate}
Then $R$-based Swan induction holds for the family $\sF$.  
\end{prop} 
\begin{proof} 
We consider the object 
$X = R \otimes M_+  \in \fun(BG, \perf(R))$ and calculate its $K_0$-class $[X]$ in two different ways.
\begin{enumerate}
	\item By assumption, $M$ has a finite $G$-CW decomposition with
	equivariant cells of the form $G/H \times D^n$. 
	The $G$-cells necessarily satisfy $H \in \sF$ by hypothesis on the isotropy of
	$M$. It follows that there exist integers $n_H, H \in \sF$ such that
	\begin{equation} \label{firsteq} [X] = \sum_{H \in \sF} n_H[R \otimes G/H_+ ] = \sum_{H \in \sF}
	\mathrm{Ind}_H^G (n_H) \in \sw(G, R).  \end{equation}

	\item The assumption \Cref{item:trivialization} gives an {equivalence} in $\fun(BG,
	\perf(R))$ between $X$ and a direct sum of $n>0$ even shifts of the unit. It follows
	that 
	\begin{equation} \label{secondeq}
		[X] = n \in \sw(G, R).  \end{equation}
	Equating \eqref{firsteq} and \eqref{secondeq}, we obtain the result. \qedhere
\end{enumerate}
\end{proof} 

The first condition in \Cref{artinworkhorse} will be satisfied if, for example, $M$ is a compact smooth manifold
with $G$-action, by the equivariant triangulation theorem \cite{Ill78}.  
We can check the condition \Cref{item:trivialization} of 
 \Cref{artinworkhorse}  via the following result.

\begin{prop} 
\label{rem:Gactiontriv}
Suppose that $M$ is a $G$-space with the homotopy type of a finite $G$-CW complex.
Then $M$ satisfies condition \cref{item:trivialization} of \Cref{artinworkhorse} if and only if: 
\begin{enumerate}
	\item The $R^*$-cohomology $R^*(M_{hG})$ is a free module on generators in
	even degrees over $R^*(BG)$. 

	\item The natural map
	\[ R^*(\ast) \otimes_{R^*(BG)} R^*(M_{hG}) \to R^*(M)  \]
	is an isomorphism. 
	\end{enumerate}
	\end{prop}  

\begin{proof}

In fact, using $(1)$, we can produce a $G$-equivariant map from a sum
	of shifts of the
	unit into 
	$R \otimes \mathbb{D} M_+  \in \fun(BG, \perf(R))$, by choosing a basis of
	$R^*(M_{hG}) = \pi_{-*}\hom_{\fun(BG, \perf(R)}( R, R \otimes
	\mathbb{D}M_+)$; the induced map is an equivalence in $\fun(BG, \perf(R))$ by
	the second condition. 
\end{proof}
\begin{thm} 
\label{cxorientswan}
Suppose there exists an $\mathbb{E}_1$-map $MU \to R$. Then $R$-based Swan induction holds
for the family of abelian subgroups (for any finite group $G$).
\end{thm} 
\begin{proof} 
We fix an embedding $G \subseteq U(n)$ and consider the action on the flag variety $M= F =
U(n)/T$ for $T \subseteq U(n)$ a maximal torus. As a smooth $G$-manifold, $M$ admits a finite $G$-CW structure. The stabilizers of the $G$-action are abelian (as they are
contained in conjugates of $T$). By \cite[Prop.~7.49]{MNN17}, we obtain an
equivalence of the form \eqref{trivaction}. 
Alternatively, 
we can use \Cref{rem:Gactiontriv} and the projective or flag bundle formula to
see this. 
Therefore, we can apply
\Cref{artinworkhorse} to conclude. 
\end{proof} 

When $R$ is a discrete commutative ring, a classical theorem of Swan 
\cite{Swa60} states that one has Swan induction for the   family of \emph{cyclic}
subgroups. We give a geometric proof of Swan's theorem in the
spirit of some of our other results. 
\begin{thm}[Swan \cite{Swa60}] 
\label{swanthm}
Let $R$ be an $\mathbb{E}_\infty$-ring which admits an $\mathbb{E}_1$-map from
$H\mathbb{Z}$. 
Then $R$-based Swan induction holds
for the family of cyclic groups  (for any finite group $G$).
\end{thm} 
\begin{proof} 
Without loss of generality, we may take $R = H\mathbb{Z}$. 
By \Cref{cxorientswan} and downward induction based on 
\Cref{general:swaninduction}, we see that 
it suffices to consider $G = C_p ^{\times 2}$ for some prime $p$.
We consider 
the $p$-dimensional projective {Heisenberg representation} of $G$ on
$\mathbb{C}^p$, given by the matrices
\begin{equation} 
A = 
\begin{bmatrix}
1 \\
 & \zeta_p \\
 & & \zeta_p^2 \\
 & & & \ddots \\
 & & & & \zeta_p^{p-1}
\end{bmatrix}, \quad  \quad  B = 
\begin{bmatrix}
 & 1 \\
 & & 1 \\
 & & & \ddots \\
 1 & & 
\end{bmatrix}.
\end{equation} 
Here $A$ is a diagonal matrix whose eigenvalues are the powers of a primitive $p$th root $\zeta_p$ of unity, and $B$ is the permutation matrix for a cyclic permutation.
Since the matrices $A$ and $B$ commute up to scalars, they define a {projective} representation of $G$,
yielding
an embedding
$G \subseteq PGL_p(\mathbb{C})$.

The group $PGL_p(\mathbb{C})$ acts naturally on $\mathbb{CP}^{p-1}$, and
 the action of the subgroup  $G \subseteq
PGL_p(\mathbb{C})$ has no fixed points. It follows that the class
$[H\mathbb{Z} \otimes \mathbb{CP}^{p-1}_+]$ in 
$\sw(G, \mathbb{Z})$
is a sum of classes induced from proper subgroups. 
To calculate the class $[H \mathbb{Z} \otimes \mathbb{CP}^{p-1}_+]$ in another manner,
we can also consider the finite Postnikov filtration $\{\tau_{\leq 2i} 
( H\mathbb{Z} \otimes \mathbb{CP}^{p-1}_+)
\}$
whose successive subquotients are even suspensions $\Sigma^{2i} \mathbb{Z}$.
The $G$-action on each of the (shifted discrete) associated graded terms is trivial
because the $G$-action extend to an action of the connected group
$PGL_p(\mathbb{C})$. 
Therefore, we find that $[H\mathbb{Z} \otimes \mathbb{CP}^{p-1}_+] = p \in \sw(G, \mathbb{Z})$. 
It follows that we have 
integers $n_H$ for each $H \subsetneq G$ such that
\begin{equation} 
p = \sum_{H \subsetneq G} n_H\mathrm{Ind}_H^G(1) \in \sw(G, \mathbb{Z}). \qedhere
\end{equation} 
\end{proof}

\begin{thm} 
\label{p-1local}
Let $G$ be a finite group. If $R$ is any $\mathbb{E}_\infty$-ring with $|G|
\in \pi_0(R)^{\times}$, then $R$-based Swan induction holds
for the family of cyclic subgroups of $G$. 
\end{thm} 
\begin{proof} 
Without loss of generality, we may take $R = \mathbb{S}[1/|G|]$. 
We have $\fun(BG, \perf(R)) \simeq \perf(
R[G])$; 
equivalently, an $R[G]$-module is perfect if and only if its underlying
$R$-module is perfect, and similarly for every subgroup of $G$. 
This follows from the fact that taking homotopy $G$-fixed points commutes with
arbitrary colimits in the $\infty$-category of $R$-modules. 
Recall that if $A$ is a connective $\e{1}$-ring, then the natural map $K_0(A) \to
K_0( \pi_0 A)$ is an isomorphism. We thus find a chain of isomorphisms
\begin{align*}
	\sw(G,R)= K_0 ( \fun(BG, \perf(R))) & \simeq K_0( R[G])  \\
	& \simeq K_0 (\pi_0 R[G])
	\simeq K_0 ( \fun(BG, \perf( \mathbb{Z}[1/|G|] )))
	 .  
	\end{align*}
	Applying Swan's theorem (\Cref{swanthm} above), we conclude the result.  
\end{proof}

Next, we prove a Swan induction result for $KU$. 
We give a geometric argument here that only works at small primes; we will prove
the result in full generality later in \Cref{artinindKU}. 
\begin{thm} 
\label{artindKU0} \label{artinindKU0}
For any finite group $G$, 
 $KU$-based Swan induction holds for  the family of abelian subgroups of $G$
 whose $p$-part for $p \in \{2, 3, 5\}$ has
	rank $\leq 2$.
\end{thm} 
\begin{proof} 
In view of 
\Cref{cxorientswan}, 
we may assume that $G$ is abelian. 
We then reduce to proving 
that one has Swan induction for $G = C_p^{\times 3}$ and $p\in\{ 2,3,5\}$ for the family of proper subgroups
(cf.~\Cref{general:swaninduction}). 
Since $p \leq 5$,  we have an embedding $G \hookrightarrow \Gamma$ for
$\Gamma$ a
suitable simply connected compact Lie group whose image is not contained in any
maximal torus: by the results of \cite{Borel61}, it suffices
to choose $\Gamma$ such that $H^*(\Gamma; \mathbb{Z})$ has
$p$-torsion in its cohomology, e.g., we can take $\Gamma=E_8$. 
Let ${T} \subseteq \Gamma$ be a maximal torus and consider the $\Gamma$-action 
on the flag variety $\Gamma/{T}$, as well as its restricted $G$-action. 
We will show that the $G$-space $\Gamma/T$ satisfies the hypotheses of
\Cref{artinworkhorse}, for $\sF$ the family of proper subgroups.
First of all, $G$ acts without fixed points since it is not contained in any
maximal torus of
$\Gamma$.
By \cite[Cor.~8.17]{MNN17}, we have an equivalence
\[ \Gamma/T_+ \otimes KU \simeq \bigoplus_1^{|W|} KU \in \fun(B\Gamma ,
\perf(KU)),  \]
where $W$ is the Weyl group of $\Gamma$. 
Restricting to $G$, this proves hypothesis \Cref{item:trivialization} of \Cref{artinworkhorse}
and thus our result. 
\end{proof} 

Next, we include some results which are specific to the prime $2$, based on the
use of representation spheres; they have the advantage of applying at arbitrary
chromatic heights. 
We first need two lemmas that will enable us to recognize the triviality of
group actions (for which we fix an arbitrary prime $p$).

\begin{lem} 
\label{lemCptrivial} 
Let $E$ be an even $\mathbb{E}_1$-ring spectrum such that $\pi_*E$ is
$p$-torsion-free. Let $M \in
\fun(BC_p, \perf(E))$ be such that: 
\begin{enumerate}
\item The underlying $E$-module of $M$ is equivalent to a direct sum of copies
of $E$.
\item The $C_p$-action on $\pi_* M$ is trivial. 
\end{enumerate}
Then $M$ is equivalent to a direct sum of copies of the unit in $\fun(BC_p,
\perf(E))$.
\end{lem} 
\begin{proof} 
Let $\{x_i\}\subseteq \pi_0 M$ be a basis of the free $\pi_* E$-module $\pi_* M$.
For each $i$, we want to produce a $C_p$-equivariant map 
of $E$-modules
\begin{equation} \label{equivmapsigma}  E \to M  \end{equation} 
which carries $1 \mapsto x_i$ in homotopy. Taking the direct sum of these maps,
we will have the desired equivalence. 
Equivalently, to produce \eqref{equivmapsigma}, we need to show that the image
of $\pi_* (M^{hC_p} )\to \pi_* M$ contains each $x_i$. However, the $E_2$-term of the homotopy
fixed point spectral sequence for $\pi_* (M^{hC_p})$ is concentrated in even total degree by our assumptions, and thus collapses. This shows that there are no obstructions to producing the maps
\eqref{equivmapsigma} and thus to providing the equivalence of the lemma. 
\end{proof} 

\begin{lemma} 
Fix a group $G$. 
Let $E$ be an $\mathbb{E}_1$-ring spectrum such that $\pi_*( C^*(BG; E))$
is even and  $p$-torsion-free. 
Let 
$M \in \fun(B (C_p \times G), \perf(E))$ be such that: 
\begin{enumerate}
\item  
The underlying object of $M$ in $\fun(BG, \perf(E))$ is equivalent to a direct sum of
copies of the unit. 
\item The $C_p$-action on $\pi_*(M^{hG})$ is trivial. 
\end{enumerate}
Then $M$ is equivalent to a direct 
sum of copies of the unit in 
$\fun(B(C_p \times G), \perf(E))$: in particular, the $C_p \times G$-action is trivial. 
\label{trivactionproduct}
\end{lemma} 

\begin{proof} 
We use that $\fun(B(C_p \times G), \perf(E)) = \fun(BC_p,  \fun(BG, \perf(E)))$. 
The thick subcategory of $\fun(BG, \perf(E))$ generated by the unit is
equivalent to $\perf( C^*(BG; E))$, via the functor $(-)^{hG}$. 
Thus, the result follows from 
\Cref{lemCptrivial} applied to the object $M^{hG} \in \fun(BC_p,
\perf(C^*(BG; E)))$. 
\end{proof}

\begin{prop} \label{thm:induction}
Let $p$ be a prime, $n\ge 1$ and $G$ be an elementary abelian $p$-group of rank
$n+2$. Let $R$ be a $T(n)$-local, even 
$\mathbb{E}_\infty$-ring under $E_n$ such that $\pi_* R$ is torsion-free. Let $M \in \fun(BG, \perf(R))$. 
Suppose that for each $H \subsetneq G$, the object $\Res^G_H M \in \fun(BH,
\perf(R))$ is equivalent to a direct sum
of copies of the unit. Then $M$ is equivalent to a direct sum of copies of the
unit in $\fun(BG, \perf(R))$.

\end{prop} 
The above proposition establishes a very weak result towards the general expectation that 
the complexity of the representation theory over $\mathbb{E}_\infty$-rings should stabilize once
the rank of the group is a bit larger than the chromatic complexity of the coefficients.
A more subtle such result would be our \Cref{conj:Morava_Swan}.

\begin{proof}[Proof of \Cref{thm:induction}] 
Let $G'   \subsetneq  G$ be a maximal proper subgroup, and fix a complement
$C_p \simeq H \subseteq G$, so that $G = G' \times H$. 
By our assumptions and \Cref{trivactionproduct}, 
it suffices to show that the $H$-action on 
$\pi_*( M^{hG'})$
is trivial (cf.~\Cref{triv on morava}, which shows that $C^*(BG'; R)$ is even and torsion-free).

To see this, we claim that the map of $C^*(BG', R)$-modules
\begin{equation} \label{mapofRmodules}  M^{hG'} \to \prod_{G'' \simeq
C_p^{\times n} \subsetneq  G' } M^{hG''}  \end{equation}
is injective on homotopy. Since $M$ is $G'$-equivariantly isomorphic to a sum of
copies of the unit, it suffices to verify the injectivity of
$\pi_*(\eqref{mapofRmodules})$ with $M$ replaced by $R$; this case follows because
\begin{equation} \label{equation:restrict} C^*(BG', R) \to \prod_{G''\simeq
C_p^{\times n} \subsetneq G'} C^*(BG'', R)  \end{equation}
is injective on $\pi_*(-) \otimes \mathbb{Q}$ in light of
\cite[Th.~3.18 and Prop.~5.36]{MNN19} (this is essentially a consequence of the
character theory of
\cite{HKR00}) since both sides are torsion-free. 

For any $G''\subsetneq G'$, we have an induced $H$-action on 
the $C^*(BG'', R)$-module
$M^{hG''}$, and the map in \eqref{mapofRmodules} is $H$-equivariant. 
Since $M$ restricts to a direct sum of copies of the unit for every proper subgroup of $G$, the
$H$-action on $M^{hG''}$
is trivial;  indeed, this follows because the action of the
\emph{proper} subgroup generated by
$G''$ and $H$ on $M$ is trivial.
It follows from the injectivity on homotopy of 
\eqref{mapofRmodules}  and this observation that the $H$-action on 
$M^{hG'} $ is trivial on homotopy groups. This completes the proof. 
\end{proof}

We now start considering representation spheres.
	For a based space $B$, let $B\langle n\rangle$ be the $(n-1)$st connective cover of $B$, so the first potentially non-trivial homotopy group is in degree $n$. 
	Let $MO\langle n\rangle$ be the Thom spectrum associated to the
	$J$-homomorphism $BO\langle n\rangle \to BO\to BGL_1 \mathbb{S}$.
We define the function $\phi$ for all integers $n \geq 1$ via  $\phi(n) = 8a +2^b$ when $n=4a +b+1$ with $0\leq b\leq 3$.

\begin{lemma}\label{lem:orientable}
	Let $n\ge 1$, $G=C_2^{\times n}$ and consider the characters $\{\eta_i\}_{1\leq
	i\leq n}$ of $G$ obtained by pulling back the sign character along the $n$
	projection maps. Define $\alpha=\prod_{i=1}^{n} (1-\eta_i)\in RO(G)$. Then for any $MO\langle
	\phi(n) \rangle$-oriented $\mathbb{E}_1$-ring spectrum $R$, there is an equivalence $S^\alpha \otimes R\simeq R$ in $\fun(BG, \perf(R))$.
\end{lemma}
\begin{proof}
We claim that the map 
\begin{equation} \label{nthBC2map} BC_2^{\times n} \xrightarrow{\prod_{i=1}^n (1 -\eta_i)}
BO\end{equation}
lifts to $BO\left \langle \phi(n)\right\rangle$. 
This implies that for any  $MO\langle \phi(n)\rangle$-oriented $R$, there is 
an equivalence $S^\alpha \otimes R \simeq R $ in $\fun(BG, \perf(R))$ as
desired. 

By Bott periodicity, $\phi(n)$ is the 
degree of the $n$th nonzero homotopy group of $BO$, starting with $\phi(1) = 1$. 
We argue inductively on $n\geq 1$ to construct the lifting. If $\eta$ is the sign representation of
$C_2$, then the virtual representation sphere
$S^{1-\eta}$ is classified by a map \[BC_2\simeq BO(1)\to BO \]  
which lifts to $BO\left\langle \phi(1)\right \rangle = BO \left\langle 1\right \rangle$
because $BC_2$ is connected. This settles the base case $n=1$.

Suppose now we have a lifting of 
\eqref{nthBC2map} to $BO\left \langle \phi(n)\right\rangle$ for some $n\ge 1$.
The next classifying map is obtained as follows:
	\[
	BC_2^{\times n+1} \simeq BC_2 \times BC_2^{\times n}
		\xrightarrow{(1-\eta_{n+1})\otimes \prod_{i =1}^n(1-\eta_i)) }  BO(1) \wedge BO\langle \phi(n)\rangle 
		\to  BO\wedge BO 
		\xrightarrow {\otimes} BO .
		\]
	Since $BO(1)\wedge BO\langle \phi(n)\rangle$ is $\phi(n)$-connected, the
	composite of the last two maps lifts to $BO\langle \phi(n+1)\rangle$ as desired.
\end{proof}

\begin{thm} \label{thm:2-primary}
Let $G$ be an abelian 2-group. 
Let $R$ be an $\mathbb{E}_\infty$-ring spectrum. Suppose that for some $n \geq 1$, we
have either: 
\begin{enumerate}
  \item $R  =E_n$, a Lubin-Tate theory of height $n$ at the prime $2$.
  \item $R$ is $MO\langle \phi (n+2)\rangle$-orientable (as an $\mathbb{E}_1$-ring and with  $\phi$ as defined before \Cref{lem:orientable}).  \end{enumerate}
Then $R$-based Swan induction holds for the family of subgroups of $G$ of rank at
most $n+1$. 
\end{thm} 
\begin{proof} 
By pulling back along the map to a maximal elementary abelian quotient of $G$,
it suffices to treat the case where $G=C_2^{\times{n+2}}$ and prove that
$R$-based Swan
induction holds for the family of proper subgroups, cf. \Cref{general:swaninduction}.
Let $\epsilon_1, \dots, \epsilon_{n+2}$ be independent sign
characters $G \to \left\{\pm 1\right\}$, i.e., $\left\{\epsilon_i\right\}$ is a
basis for the $\mathbb{F}_2$-vector space $\hom(G, \left\{\pm 1\right\})$. 
Let $\eta_i \in RO(G)$ $(1\leq i\leq n+2)$ be the class of the associated one-dimensional
representation (as in \Cref{lem:orientable}). We consider the class
\[ \alpha = \prod_{i=1}^{n+2}( 1 - \eta_i)  \in RO(G), \]
and the associated representation $G$-sphere $S^{\alpha} \in \GSpec$. 
Note that for any proper subgroup $H \subsetneq G$, 
$\alpha$ restricts to a class in $RO(H)$ which is divisible by $2$ and
therefore comes from the complex representation ring. This follows
because $\alpha$ belongs to the $(n+2)$-th power of the augmentation ideal,
and for any $m$, the augmentation ideal in $RO(C_2^{\times m})/2 =
\mathbb{F}_2[C_2^{\times m}]$ has its
$(m+1)$th power
equal to zero.

Given a nontrivial character $\mu$ of $G$, considered as a 1-dimensional real
representation, the cofiber sequence $S(\mu)_+ \to
\mathbb{S} \to S^{\mu}$ shows that the class $[S^{\mu}] \in \rep(G, \mathbb{S})$ has the property that $[S^{\mu}]-1=[S(\mu)_+]$ is induced from a
proper subgroup, namely the kernel of $\mu$. 
Consequently, if $V$ is a sum of nontrivial characters in $RO(G)$, 
then
$[S^V] -1 \in \rep(G, \mathbb{S})$ is a sum of classes induced from proper subgroups in 
$\rep(G, \mathbb{S})$. 
Expanding out the product defining $\alpha$ into a sum of characters, we find
only a single trivial representation (since the $\epsilon_i$ are linearly independent).
It follows that
in 
$\rep(G, \mathbb{S})$, one has 
\begin{equation} [S^{\alpha}] = -[S^{\alpha -1}] = -1 + C,
\label{Ceq}\end{equation}  for $C$ a sum of classes induced
from proper subgroups. 

We also claim that \[ S^\alpha \otimes R \simeq  R \mbox{ in } \fun(BG, \perf(R)).  \]
Under the first hypothesis, this follows from \Cref{thm:induction} since the
restriction of $\alpha$ to a proper subgroup is a complex representation sphere,
and thus
trivializable.
Under the second hypothesis, this follows from \Cref{lem:orientable}.
Consequently, $[S^\alpha \otimes R] = 1 \in \rep(G, R)$. 
By \eqref{Ceq}, it follows that $1 = -1 + C$, so 
 $2 \in \rep(G, R)$ is a sum of classes induced from proper subgroups, as desired. 
\end{proof} 

Note that this argument cannot work at odd primes, since all representations of
$C_p$ are complex for $p > 2$.

\begin{proof}[Proof of \Cref{Ex:Swaninduction}, \Cref{Ex:Swaninduction_Morava}]
This follows from \Cref{p-1local} to handle the prime-to-$2$ case combined with
\Cref{thm:2-primary}, which handles the prime $2$.
\end{proof}

\subsection{Swan induction for $KU$}\label{section:Swan_for_KU}

In this subsection, we prove the Swan induction theorem for $KU$. 
Note that we have already given (geometric) proofs earlier for the $p$-part with
$
p\leq 5$, see \Cref{artinindKU0}. 

\begin{thm} 
\label{artinindKU}
Let $G$ be any finite group. 
Then 
$KU$-based Swan induction  holds for the family of abelian subgroups of
rank $\leq 2$. 
\end{thm}

To prove \Cref{artinindKU}, 
it suffices (cf.~\Cref{general:swaninduction}) to treat the case of $G
= C_p^{\times
3}$ for an arbitrary prime $p$. 
Our proof will rely essentially on {twisted} $K$-theory.
We will first need various preliminaries. 

\begin{cons}[Twists of $K$-theory]
There is a natural map of anima
\[ K(\mathbb{Z}, 3) \to BGL_1(KU),  \]
 where $BGL_1(KU)$ is the classifying space of trivial invertible $KU$-modules, 
cf.~\cite[Sec.~7]{ABG10} for an account,  
which induces the identity on $\pi_3$.
Consequently, for any finite group $G$, we have a natural map
\begin{equation} \label{twistK} H^3(G; \mathbb{Z}) \to \pic( \fun(BG, \perf(KU))),  \end{equation}
where the right-hand-side is the Picard group of the symmetric monoidal
$\infty$-category 
$\fun(BG, \perf(KU))$. 
Given a class $\tau \in H^3(G; \mathbb{Z})$, we let $KU_{\tau}$ 
be the associated object of $\fun(BG, \perf(KU))$. 
\end{cons}

We will be especially interested in the case $G = C_p^{\times 2}$. 
Choosing a nonzero class $\tau \in H^3(G; \mathbb{Z}) = \mathbb{F}_p$, 
we obtain an invertible object $KU_{\tau} \in \fun(B(C_p^{\times 2}),
\perf(KU))$. 
The induced map $B(C_p^{\times 2}) \to BGL_1(KU)$ is nontrivial, as it lifts
uniquely to
the $3$-connective cover $\tau_{\geq 3} BGL_1(KU)$, and $K(\mathbb{Z}, 3)$
splits off as a direct factor from here; this means that $KU_{\tau}$ is not
equivalent to the unit in $\fun(BG, \perf(KU))$. 

In the next lemma, to distinguish the factors, we write $C_p^a \subseteq
C_{p}^{\times 2}$ for the first factor and $C_p^b \subseteq C_p^{\times 2}$
for the second. 
Note that $KU^0(BC_p^b) $ is isomorphic to the completion of the
representation  ring $R(C_p^b)$ at the augmentation ideal by
the Atiyah--Segal completion theorem \cite{AS69, At61}. 
If $\zeta$ is a nontrivial character of $C_p^b$, then $R(C_p^b)$ is free on the
powers of $[\zeta]$.

\begin{lem}\label{lem:twisted-action} 
Let $\tau \in H^3(C_p^{\times 2}; \mathbb{Z})$ be a nontrivial element. 
\begin{enumerate}
\item  The underlying object $KU_{\tau}|_{C^b_p}$ in $\fun(BC^b_p, \perf(KU))$ is
isomorphic to the unit. 
\item The residual $C_p^a$-action on $(KU_{\tau})^{hC_p^b}\simeq
C^*({BC_p^b}, KU)$ has the property that the action
by a generator in $C_p^a$ acts by multiplication by 
$[\zeta]^i, $ for some $0 < i < p$. 
\end{enumerate}
\end{lem} 
\begin{proof} 
The first assertion follows because $\tau$ restricts to zero in $H^3(C_p;
\mathbb{Z}) = 0$. The second assertion follows because the action of a
generator is necessarily given by multiplication by an element of $KU^0(BC_p)$ whose $p$th power is the
identity.  Moreover, this generator is necessarily nontrivial or the entire
$C_p^{\times 2}$-action on $KU_{\tau}$ would be trivial by \Cref{trivactionproduct}. 
\end{proof} 

Our key tool is the following result. 
We consider the $C_p^{\times 2}$-action on $\mathbb{CP}^{p-1}$ arising from the
projective representation on $\mathbb{C}^p$ as in the proof of \Cref{swanthm}. 
We identify the $KU$-linearization of this action. 

\begin{prop} 
\label{projdecomp}
We have a decomposition in $\fun(B(C_p^{\times 2} ), \perf(KU))$
\begin{equation}  KU \otimes \mathbb{D}\mathbb{CP}^{p-1}_+  \simeq
\bigoplus_{\tau \in H^3(C_p^{\times 2}; \mathbb{Z})} KU_{\tau}.  \label{KUCPequiv} \end{equation}
\end{prop} 
\begin{proof} 
Again, we label the first and second factors of $C_p^{\times 2}$ by $C_p^a, C_p^b$. 
We first calculate 
$KU_{C_p^{b}}^*( \mathbb{CP}^{p-1}) $, i.e., 
the $C_p^b$-equivariant $KU$-theory of $\mathbb{CP}^{p-1}_+$. 
Fix  a nontrivial character $\zeta$ of $C_p^{b}$. 
The underlying $C_p^b$-space of $\mathbb{CP}^{p-1}_+$ is the projectivization of the 
representation $1 \oplus \zeta \oplus \dots \oplus \zeta^{\otimes (p-1)}$ of
$C_p^{b}$.

By the projective bundle theorem, it follows that 
there is an isomorphism of $R(C_p^b)$-algebras,
\[ KU_{C_p^{b}}^*( \mathbb{CP}^{p-1}) \simeq 
R(C_p^b)[x]/ \prod_{i=0}^{p-1} ( x - [\zeta^i])
,\]
cf.~\cite[Prop.~3.9]{Segal68}.
Here $x$ is the class of 
the tautological line bundle on $\mathbb{CP}^{p-1}$, which is (canonically)
$C_p^b$-equivariant. 
We have a residual $C_p^{a}$-action on this $R(C_p^{b})$-algebra. 
Using the definition of $x$ as the class of a tautological bundle, one checks
that 
a generator of $C_p^a$ carries $x$ to $x [\zeta^i]$ for an appropriate $i \neq
0$.\footnote{Explicitly, 
we consider the $C_p^b$-equivariant line bundle on $\mathbb{CP}^{p-1}$ given by
the set of pairs $(x, v)$ for $x \in \mathbb{CP}^{p-1}$ and $v \in x$; the
$C_p^b$-equivariant structure is by action on the pair. 
The claim follows by noting that $C_p^a, C_p^b$ act on $\mathbb{C}^p$, but their
actions fail to commute by a $p$th root of unity.} 

From this, it follows that $(KU \otimes
\mathbb{D}\mathbb{CP}^{p-1}_+)|_{C_p^{b}}$ is a direct sum of $p$ copies of the unit in
$\fun(BC_p^b, \perf(KU))$. 
Therefore, we have 
$(KU \otimes \mathbb{D}\mathbb{CP}^{p-1}_+)^{hC_p^b} \simeq \bigoplus_{i=0}^{p-1}
C^*(BC_p^b,
KU)$. By the comparison between equivariant and Borel-equivariant $K$-theory,
and the above calculation, we see that the residual 
$C_p^a$ acts on the $i$th factor by multiplication by $[\zeta^i] \in R(C_p^b)
\to KU^0(BC^b_p)$ (up to renumbering factors). 

Now we prove the desired equivalence. 
It suffices to compare the $C_p^b$-homotopy fixed points of both sides of
\eqref{KUCPequiv}, $C_p^a$-equivariantly as free modules over the even,
torsion-free $\mathbb{E}_\infty$-ring spectrum $C^*(BC_p^b; KU)$. We will do this by a
homotopy fixed-point spectral sequence argument. 
On $\pi_0$, we have seen from the previous paragraph and 
\Cref{lem:twisted-action}
that 
$\pi_0 \left( (KU \otimes \mathbb{D}\mathbb{CP}^{p-1}_+)^{hC_p^b} \right)$
and $\pi_0 ( \bigoplus_{\tau} KU_{\tau}^{hC_p^b} )$ 
are isomorphic as 
$\pi_0( C^*(BC_p^b, KU))$-modules equipped with a 
$C_p^a$-action. 
Using the homotopy fixed point spectral sequence (and observing that there is no
room for obstructions\footnote{Note here that 
all the summands $\pi_0 ( KU_{\tau}^{hC_p^b})$ for $\tau \neq 0$ have trivial
higher $C_p^a$-cohomology.}), we can produce 
$C_p^a$-equivariant maps $KU_{\tau}^{hC_p^b} \to (KU \otimes
\mathbb{D}\mathbb{CP}^{p-1})^{hC_p^b}$ for each $\tau$, and the direct sum of these is an
equivalence. 
\end{proof}

\begin{proof}[Proof of \Cref{artinindKU}] 
It suffices to show that $KU$-based Swan induction holds for the
family of proper subgroups of $C_p^{\times 3}$. 
We first observe that $\sw(-, KU) \otimes \mathbb{Q}$ is a Green functor
and thus $\sw(G, KU) \otimes \mathbb{Q}$ receives a map from the rationalized Burnside ring $A(G) \otimes \mathbb{Q}$. 
For any finite group $G$, we have complementary idempotents $e_G, \widetilde{e}_G$ in
$ A(G) \otimes \mathbb{Q} $ (which is isomorphic to a product of copies of
$\mathbb{Q}$ over conjugacy classes of subgroups $H \subseteq G$) such that: 
\begin{enumerate}
\item  $e_G$ is a $\mathbb{Q}$-linear combination of the classes of the $G$-sets $G/H, H\subsetneq 
G$. 
\item  
For each $H \subsetneq G$, the 
restriction of $e_G$ to $A(H) \otimes \mathbb{Q}$ is equal to $1$. 
Equivalently, for each $H \subsetneq G$, the homomorphism $A(G) \otimes \mathbb{Q} \to
\mathbb{Q}$ which sends a $G$-set $T$ to $|T^H|$ carries $e_G$ to $1$. 
\item $e_G + \widetilde{e}_G = 1$.
\end{enumerate}

Let $M$ be a rational  Green functor for the group $G$, so that 
we have a ring map $A(G) \otimes \mathbb{Q} \to M(G)$. Then $M$ is induced from the
family of proper subgroups (equivalently, $1 \in M(G)$ is a sum of classes
induced from proper subgroups) if and only if this map carries $e_G$ to $1$ (or, equivalently, sends
$\widetilde{e}_G$ to zero); indeed, this follows because multiplication by $e_G$ acts as the
identity on classes induced from a proper subgroup. 

Our strategy of proof is to verify this identity in $\mathrm{Rep}(C_p^{\times 3}, KU)
\otimes \mathbb{Q}$ directly using \eqref{KUCPequiv}, using a relation (proved in
the next paragraph) between the idempotent for $C_p^{\times 2}$ and the class of
$\mathbb{CP}^{p-1}$.

Consider $\sw(C_p^{\times 2}, R)$ for any $\e{\infty}$-ring $R$. 
In this case, we have another expression for the image of $e_{C_p^{\times 2}}$
under $A(C_p^{\times 2})\otimes\mathbb{Q} \to \sw(C_p^{\times 2}, R)\otimes \mathbb{Q}$
(for which we will simply write $e_{C_p^{\times 2}}$). In fact, we claim that 
\begin{equation} [R \otimes \mathbb{CP}^{p-1}_+]/p = e_{C_p^{\times 2}} \in \rep( C_p^{\times 2}, R) \otimes \mathbb{Q}, \label{expreCp}\end{equation} 
for the $C_p^{\times 2}$-action on $\mathbb{CP}^{p-1}$ arising from the
$p$-dimensional projective representation as above. 
In other words, 
$pe_{C_p^{\times 2} }$ is the class of $R \otimes \mathbb{CP}^{p-1}_+ \in \fun(BC_p^{\times 2},
\perf(R))$ in rationalized $K_0$. To see this, we first observe that any finite $G$-CW complex has a well-defined Euler characteristic taking values in $A(G)$ which can be calculated by taking the Euler characteristic of the cellular chains. Now $\mathbb{CP}^{p-1}$ is a finite
fixed-point-free 
$C_p^{\times 2}$-complex such that the fixed points
under any proper subgroup have Euler characteristic $p$ (since these fixed
points will either be $p$ distinct points or $\mathbb{CP}^{p-1}$).
This implies the
associated class in $A(C_p^{\times 2})$ is $pe_{C_p^{\times 2}}$ as desired, by
the above characterization of the idempotent $e_G$, whence the claim.

We specialize now to the case where $R = KU$. 
Let $\tau$ be a generator of $H^3(C_p^{\times 2}; \mathbb{Z}) = \mathbb{F}_p$ and let $x =
[KU_{\tau}]$.
Then, combining \eqref{expreCp} and  the decomposition of \Cref{projdecomp},
we conclude 
$$\frac{1 + x + \dots + x^{p-1}}{p} =   e_{C_p^{\times 2}} \in \sw(C_p^{\times
2}, KU) \otimes \mathbb{Q}.$$ 
Note that $x^p = 1$, so the left hand side is clearly idempotent. This also
determines the complementary idempotent;  we therefore have
\begin{equation} \frac{1}{p}\prod_{j=1}^{p-1} ( 1 - x^j) =
\widetilde{e}_{C_p^{\times 2}},
\label{eGt} \end{equation}
since  
$\frac{1}{p}\prod_{j=1}^{p-1} ( 1- x^j)$ is the complementary idempotent to
$\frac{1 + x + \dots + x^{p-1}}{p}$ in the group ring $\mathbb{Z}[1/p, x]/(x^p =
1) = \mathbb{Z}[1/p, \zeta_p] \times \mathbb{Z}[1/p]$.

Our goal is to show that $\widetilde{e}_{C_p^{\times 3}}= 0$ in $\sw(C_p^{\times 3}, KU) \otimes
\mathbb{Q}$, which is equivalent to the Swan induction claim. 
Given an elementary abelian $p$-group $G$ of rank $\mathrm{rk}(G)\geq 2$, 
one has $\widetilde{e}_G = \prod_{\phi \colon G \twoheadrightarrow G'} \phi^*
e_{\widetilde{G}'}$, where the product ranges over all surjections $G
\twoheadrightarrow G'$ with $\mathrm{rk}(G') = \mathrm{rk}(G) -1$.
This follows 
since the given product over all $\phi$ is an idempotent in $A(G) \otimes \mathbb{Q}$
with trivial restriction to proper subgroups and with image under the $G$-fixed
point map $A(G) \to \mathbb{Q}$ equal to $1$. 
Therefore, we can express 
$ \widetilde{e}_{C_p^{\times 3}}$ as the product 
\begin{equation} \label{eCp3expr}
\prod_{\phi \colon  C_p^{\times 3} \twoheadrightarrow
C_p^{\times 2}} \phi^* \widetilde{e}_{C_p^{\times 2}}, \end{equation} 
using the pullback in
the representation ring. 
We will now analyze this using group rings.

For any finite group $G$, we have a 
natural map
\( H^3(BG; \mathbb{Z}) \to \pic( \fun(BG, \perf(KU))) , \)
which defines a map of commutative rings
\begin{equation} \label{mapgpringko} \varphi_G  \colon \mathbb{Q}[  H^3(BG; \mathbb{Z})  ] \to
\mathbb{Q} \otimes_{\mathbb{Z}} \rep(G, KU), \end{equation}
which is compatible with pullback in $G$. 
By \eqref{eGt}, there exists a class in 
the group ring 
$\mathbb{Q}[H^3( B C_p^{\times 2}; \mathbb{Z})]$
whose image under $\varphi_{C_p^2}$ is precisely the idempotent
$\widetilde{e}_{C_p^{\times 2}}$. 
Using the expression \eqref{eCp3expr}, 
we see that there is a class $u \in \mathbb{Q}[H^3( B C_p^{\times 3}; \mathbb{Z})]$
whose image under $\varphi_{C_{p^3}}$ is $\widetilde{e}_{C_{p}^{\times 3}}$. 
Moreover, $u$ restricts to zero 
in $\mathbb{Q}[H^3(BH; \mathbb{Z})]$
for all proper subgroups $H \subsetneq C_{p^3}$. 
By the next two lemmas,  this is enough to force $u = 0$, which proves the theorem. 
\end{proof}

\begin{lemma} 
Let $X \simeq C_p^{\times 3}$ be a rank $3$ elementary abelian $p$-group, so $H^3(X; \mathbb{Z})$ is
also a rank $3$ elementary abelian $p$-group. 
As $Z \subseteq X$ ranges over the rank $2$ subgroups of $X$, the maps 
$H^3(X; \mathbb{Z}) \to H^3(Z; \mathbb{Z}) \simeq \mathbb{F}_p$ range over the
nonzero maps $H^3(X; \mathbb{Z}) \to \mathbb{F}_p$, up to scalars. 
\end{lemma} 
\begin{proof} 
The construction which sends $Z \subseteq X$ to the kernel of the surjection $H^3(X; \mathbb{Z})
\to H^3(Z; \mathbb{Z})$ establishes a map
\begin{equation} \label{combmap} \Psi \colon \left\{\text{$2$-dimensional subspaces } Z \subseteq X\right\}  \to 
\left\{\text{hyperplanes in } H^3(X; \mathbb{Z})\right\}.
\end{equation}
We need to show that 
\eqref{combmap} is a bijection. Note that both sides are finite sets of the
same cardinality, and that the map is $\mathrm{Aut}(X) =
GL_3(\mathbb{F}_p)$-equivariant (using the induced action on $H^3(X;
\mathbb{Z})$). 

Choose a decomposition $X = V \oplus W$ where $V$ has rank $2$ and $W$ has rank
$1$. 
By the universal coefficient theorem, we have a natural short exact sequence
\begin{equation}  0 \to H^3(V; \mathbb{Z}) \to H^3(X; \mathbb{Z}) \to \mathrm{Tor}_1( H^2(V;
\mathbb{Z}), H^2(W; \mathbb{Z})) \to 0,   \label{univcoeffseq} \end{equation}
where 
$\mathrm{Tor}_1( H^2(V;
\mathbb{Z}), H^2(W; \mathbb{Z}))$ has rank two. 
On the left-hand-side of \eqref{combmap}, we consider the collection
$C$ of
subspaces
$Z \subseteq X$ such that the composite $Z \subseteq X \twoheadrightarrow V$ is 
not surjective; equivalently, $Z = L \oplus W$ for some $L \subseteq V$ a
1-dimensional subspace. 
Note that $|C| = p + 1$. 
On the right-hand-side, consider the collection $D$ of hyperplanes in $H^3(X;
\mathbb{Z})$ which contain $H^3(V; \mathbb{Z})$; the exact sequence
\eqref{univcoeffseq} also easily shows $|D| = p+1$. 

We claim that $\Psi^{-1}(D) = C$. 
In fact, given a two-dimensional subspace $Z \subseteq X$ such that $H^3(V;
\mathbb{Z}) \to H^3(X; \mathbb{Z}) \to H^3(Z; \mathbb{Z})$ is zero, it follows
easily that the map $Z \to X \twoheadrightarrow V$ fails to be surjective, and
conversely. 
The group $\mathrm{Aut}(V) \subseteq \mathrm{Aut}(X)$ (via the diagonal
embedding, fixing $W$) preserves and acts transitively 
on $C$.  Moreover, $\mathrm{Aut}(V) \subseteq \mathrm{Aut}(X)$ preserves and acts transitively
on $D$, because we have an $\mathrm{Aut}(V)$-equivariant identification
$H^2(V; \mathbb{Z}) \simeq H^1(V; \mathbb{Q}/\mathbb{Z}) = \mathrm{Hom}(V,
\mathbb{F}_p)$, and $D$ is identified with the set of lines in $H^2(V;
\mathbb{Z})$. 
Therefore, for $c \in C$, we necessarily have that $\Psi^{-1}(\Psi(c))$ consists of a
single point since the fibers of $\Psi$ at points of $D$ must all have the same
cardinality. 
Since $\mathrm{Aut}(X)$ acts transitively on the set of 2-dimensional subspaces
of $X$, it follows easily that \eqref{combmap} is an isomorphism as desired. 
\end{proof} 

\begin{lemma} 
Let $A$ be a finite abelian group. Let $x \in \mathbb{Q}[A]$ be an element
such that for every map $A \to C$, for $C$ a cyclic group, the image of $x$
under $\mathbb{Q}[A] \to \mathbb{Q}[C]$ is zero. Then $x = 0$. 
\end{lemma} 
\begin{proof} 
We can extend scalars to $\mathbb{C}$. 
Then we have a natural isomorphism 
$\mathbb{C}[A] \simeq \prod_{A^{\vee}} \mathbb{C}$, for $A^{\vee}$ the group of
characters of $A$. 
Our assumption is that for any map $C' \to A^{\vee}$ with $C'$ cyclic, the
restriction 
$\prod_{A^{\vee}} \mathbb{C} \to \prod_{C'} \mathbb{C}$ annihilates $x$; this
clearly forces $x =0$. 
\end{proof} 

\subsection{Applications to chromatic complexity}
In this subsection, we record the 
applications of the above Swan induction theorems to chromatic bounds for the
$K$-theory of certain ring spectra. 
We recover another new proof of Mitchell's theorem and 
are able to treat some special cases of \Cref{TateMitchell}.

\begin{corollary}[Mitchell \cite{Mitchell90}] \label{mitchellthm}
For $i \geq 2$, we have $L_{T(i)}K(\mathbb{Z}) \simeq 0$ (for any prime $p$).
\end{corollary} 
\begin{proof} 
Combine \Cref{thm:swan-vanish} and 
\Cref{swanthm} 
(Swan induction for $H\mathbb{Z}$).  
\end{proof}

\begin{cor}
	Let $E_n$ be a height $n\ge 1$ Lubin-Tate theory at the prime $2$ and $G
	\subseteq \mathbb{G}_n$ a finite subgroup of the extended Morava stabilizer group. Then 
	$L_{T(n+m)}K(E_n^{hG})=0$ for all $m\geq 2$.
\end{cor}
\begin{proof}
	When $G$ is the trivial subgroup, this follows directly from
	\Cref{thm:2-primary} and \Cref{thm:swan-vanish}. The general case then follows
	from the Galois descent theorem \cite[Thm. 1.10]{CMNN}, which gives 
	\(
	  L_{T(n+m)} K(E_n^{hG})\simeq (L_{T(n+m)} K(E_n))^{hG} \simeq 0.
	\)
\end{proof}

We next recover the following result. 
At $p \geq 5$, the result is a consequence of the calculations of
Ausoni--Rognes \cite{AR02} and Ausoni \cite{Ausoni10}, which determine the mod $(p, v_1)$ homotopy groups of
$K(l)$ (resp.~$K(ku)$) at such primes (and in particular yield the stronger 
Lichtenbaum--Quillen style claim that $K(ku)/(p, v_1)$ agrees with its
$T(2)$-localization in high degrees). 
For all primes, this result has been recently proved by Angelini-Knoll--Salch
\cite{AKS20}  and Hahn--Raksit--Wilson \cite{HRW22}. 
The result is also a special case of \Cref{TateMitchell}. 

\begin{corollary} 
For $i \geq 3$, we have $L_{T(i)} K(KU) = L_{T(i)} K(KO) \simeq 0$ (at any prime
$p$).
\end{corollary} 
\begin{proof} 
By Galois descent \cite{CMNN}, it suffices to handle the case of $KU$. In this
case, \Cref{thm:swan-vanish} together with \Cref{artinindKU} imply the result.  \end{proof}

Motivated by the above results, we conjecture the following Swan induction
result for $E_n$; we have proved it at $p = 2$ in 
\Cref{thm:2-primary}, or for $n = 1$ as a consequence of  \Cref{artinindKU}. 

\begin{conj}\label{conj:Morava_Swan}
Let $p$ be a prime, $n\ge 1$, $E_n$ a Lubin-Tate theory of height $n$
at the prime $p$ and $G$  finite group. Then $E_n$-based Swan induction 
holds for the family of those abelian subgroups of $G$ for which:
	\begin{enumerate}
		\item The prime-to-$p$ part is cyclic.  
		\item The $p$-part has rank $\leq n+1$.
	\end{enumerate}
\end{conj}

\begin{remark}[A purely algebraic question] 
Finally, \Cref{thm:swan-vanish} can be used to prove that $L_{K(1)}
K(\mathbb{F}_p) = 0$, which is a consequence of the stronger result 
$K(\mathbb{F}_p; \mathbb{Z}_p) = H\mathbb{Z}_p$ proved by Quillen; indeed, one
sees that $H\mathbb{F}_p$-based Swan induction holds for the trivial family in
$C_p$ using the filtration of the regular representation $\mathbb{F}_p[C_p]$ by
trivial representations. 
One also knows that $L_{K(1)}K(\mathbb{Z}/p^n) =0$ for any $n \geq 1$,
cf.~\cite{LMMT20, BCM20, MTR20}
 for three proofs. Can this result also be proved using 
\Cref{thm:swan-vanish}, i.e., does $H\mathbb{Z}/p^n$-based Swan induction hold
for the trivial family in $C_p$? 

\end{remark} 

\appendix

\input{mackey_vs_orthogonal.tex}

\bibliographystyle{alpha}
\bibliography{DescentGroupAct}

\end{document}

%% file: mackey_vs_orthogonal.tex
\newcommand{\psg}{\mathcal{P}_{\Sigma}}

\newcommand{\CMon}{\mathrm{CMon}}
\newcommand{\OG}{\mathcal{O}(G)}
\newcommand{\OGP}{\mathcal{O}(G)_{\mathcal P}}
\newcommand{\Fin}{\mathrm{Fin}}
\newcommand{\Span}{\mathrm{Span}}
\newcommand{\Sp}{\mathrm{Sp}}
\newcommand{\Mack}{\mathrm{Mack}}
\newcommand{\op}{\mathrm{op}}
\newcommand{\colim}{\varinjlim}
\newcommand{\id}{\mathrm{id}}
\newcommand{\map}{\mathrm{Hom}}
\newcommand{\fin}{\mathrm{Fin}}
\newcommand{\Psigma}{\mathcal{P}_\Sigma}
\newcommand{\M}{\mathcal{M}}
\newcommand{\C}{\mathcal{C}}
\newcommand{\Pcal}{\mathcal{P}}
\newcommand{\Scal}{\mathcal{S}}

\section{Mackey functors and orthogonal $G$-spectra}

This appendix provides a fairly self-contained proof of the fact that, for a finite group $G$,
the symmetric monoidal $\infty$-categories afforded by orthogonal $G$-spectra and by spectral Mackey
functors are equivalent.
This result is due originally to Guillou and May \cite{GM11} (ignoring the monoidal structure),
and was revisited by Barwick and Barwick-Glasman-Shah \cite{Bar17, BGS20} in the context of more general 
parametrized homotopy theory, see specifically \cite[Thm.~A.4]{Na16}. Compared to their work, our approached is streamlined by ignoring all
models (as used by \cite{GM11}), and by not addressing any universal properties of Mackey functors
(as in \cite{Bar17, BGS20}). 

The motivation for giving our proof of their result is the immediate need of the present paper:
We use categorical methods to construct Mackey functors, and then apply descent results proven for the homotopy theory of orthogonal $G$-spectra to them.
Our work also yields a new proof of the equivariant Barratt-Priddy-Quillen theorem (which however uses the non-equivariant one).\\

Throughout, let $G$ denote a finite group. We refer the reader to
\cite[Sec.~5]{MNN17} for a quick account of the symmetric monoidal $\infty$-category $\Sp_G$ extracted from the model category of orthogonal $G$-spectra.
We denote by $\OG$ the orbit category of $G$, by $\Scal$ the $\infty$-category
of anima, by $\Scal_G:=\Fun(\OG^{\op},\Scal)$ the presentable, cartesian closed $\infty$-category of
$G$-anima (see \cite[Lem.~2.1]{BH17}), by $\Scal_{G,\bullet}\simeq \Scal_{G,*/}$
the presentable, closed symmetric monoidal $\infty$-category of based $G$-anima, and by $\Sigma_G^\infty\colon\Scal_{G,\bullet}\longrightarrow\Sp_G$ the suspension spectrum functor. We consider $\Scal_G$ with its cartesian monoidal structure.
In the appendix, we will write the units of $G$-spectra and spectral Mackey
functors by $1$ rather than $\mathbb{S}$. 

To set the notation for Mackey functors, we denote by $\Fin_G$ the category of finite $G$-sets, by
$\Span(\Fin_G)$ the $(2,1)$-category of spans on $\Fin_G$ (cf. \cite[App. C]{BH17}) and set
\[ \Mack_G:=\Mack_G(\Sp)=\Fun^\times(\Span(\Fin_G)^{\op},\Sp), \]
the category of finite product-preserving presheaves on $\Span(\fin_G)$ with values in the $\infty$-category $\Sp$ of spectra.
Note that $\Span(\Fin_G)=\mathrm{Burn}^{\mathrm{eff}}_G$ as recalled in \Cref{def:burnside},
but we stick to the former notation in the Appendix, in order to be compatible with our main references.
We recall that $\Mack_G\simeq\Psigma(\Span(\Fin_G))\otimes\Sp$, cf. \Cref{rem:p_sigma}. 
Below, we will recall the suspension functor in the Mackey context, to be denoted
\[ \Sigma_\M^\infty\colon  \Scal_{G,\bullet}\longrightarrow\Mack_G.\]
The cartesian product on $\Fin_G$ induces a symmetric monoidal structure
on $\Span(\Fin_G)$, we endow $\Mack_G$ with the symmetric monoidal structure
given by Day-convolution and denote it by $\otimes$. This is the unique
symmetric monoidal structure which is bicocontinuous and is such that $\Sigma^\infty_\M$ is symmetric monoidal. 

Our main result is the following.

\begin{theorem}\label{thm:orth_mackey}
There is a unique symmetric monoidal left-adjoint $L\colon \Sp_G\to\Mack_G$ such
that $L\circ\Sigma_G^\infty\simeq\Sigma_\M^\infty$, and $L$ is an equivalence.
\end{theorem}

The rest of this section will provide a proof of this result.

The construction of $L$ rests on the following result, and we thank
Markus Hausmann for providing a key reference in its proof. Compare also
\cite[App.~C]{GM20} for a treatment. 

\begin{theorem}\label{thm:folklore} The suspension $\Sigma_G^\infty\colon \Scal_{G,\bullet}\longrightarrow\Sp_G$ is the
initial example of a presentably symmetric monoidal functor\footnote{In other words, a map in $\CAlg(\Pr^L)$.} which inverts the functor $S^V\otimes -$
for all finite-dimensional, orthogonal representations $V$ of $G$.
\end{theorem}

\begin{proof} 
We have the symmetric monoidal suspension functor 
$\Sigma^\infty_G: \mathcal{S}_{G, \bullet} \longrightarrow \Sp_G$. 
By 
construction of $\Sp_G$, the representation spheres map to invertible objects in
$\Sp_G$. 
Now the initial presentably 
symmetric monoidal $\infty$-category $\mathcal{C}$
equipped with a cocontinuous, symmetric monoidal functor from $\mathcal{S}_{G,
\bullet} $ inverting the representation spheres 
is discussed (in a more general context) in \cite[Sec.~2.1]{robalo}, cf.~also
\cite[Lem.~4.1]{BH17}; note that the representation spheres are symmetric
objects by \cite[Lem.~C.5]{GM20} and that the required cyclic invariance condition is easily checked.
Equivalently, we can also perform this construction at the level of small finitely
cocomplete symmetric monoidal
$\infty$-categories by restricting to the compact objects.  By
\cite[Prop.~2.19, Cor.~2.22]{robalo}, we find that 
this formal inversion (in $\prl$)
is given by the colimit of smashing with $S^V$ on 
$\mathcal{S}_{G, \bullet}$, as $V$ ranges over $G$-representations.  By
construction, we obtain a canonical, cocontinuous symmetric monoidal functor
$\mathcal{C} \to \GSpec$.  It follows from the above that the mapping
anima between the finite
$G$-sets $T, T'$ are computed in the same way (namely, as 
$\varinjlim_V \hom_{\mathcal{S}_{G, \bullet}}(S^V \wedge T_+,  S^V \wedge
T'_+)$), so $\mathcal{C} \to \GSpec$ is fully
faithful on compact generators, whence the result. 
\end{proof}

To construct $\Sigma_\M^\infty$, recall from \cite[\S 9.1 before Lem.~9.4]{BH17} the canonical cartesian monoidal functor $\iota\colon\Fin_{G,+}\to\Span(\Fin_G)$ and the symmetric monoidal equivalence $\Psigma(\Fin_{G,+})\simeq
\Scal_{G,\bullet}$ (\cite[Lem.~2.1]{BH17}). This induces $\Sigma_\M^\infty$, to be defined 
as the composition \[  \Sigma_\M^\infty:=\left( \Scal_{G,\bullet}\simeq \Psigma(\Fin_{G,+})\xrightarrow{\Psigma(\iota)}\Psigma(\Span(\Fin_G))=\Fun^\times(\Span(\Fin_G)^{op},\Scal)\rightarrow\Mack_G\right),\]
where the final map is the stabilization. By construction, $\Sigma_{\mathcal{M}}^\infty$ is a map in $\CAlg(\Pr^L)$.

\Cref{thm:folklore} tells us that to construct the functor $L$ in \Cref{thm:orth_mackey}, we need to see that $\Sigma_\M^\infty$ inverts all
representation spheres. We will do this by constructing from scratch on $\Mack_G$ what 
will a posteriori turn out to be geometric fixed point functors, and by establishing some of their
basic properties. Denote by $\Pcal$ the family of proper subgroups of
$G$ and recall the cofiber sequence in $\Scal_{G,\bullet}$,
defining $\widetilde{E\Pcal}$:
\[ \colim_{G/H\in\OG_\Pcal} G/H_+\simeq E\Pcal_+\longrightarrow
*_+=S^0\longrightarrow \widetilde{E\Pcal},\]
(cf. \cite[Appendix A.1]{MNN19}). The least formal part of our argument is the following.

\begin{lemma}\label{lem:fixed_of_tilde_EP}
We have an equivalence $\left(\Sigma_\M^\infty(\widetilde{E\mathcal{P}})
\right)(G/G)\simeq \mathbb{S}$ in $\Sp$ (since the source is an
$\mathbb{E}_\infty$-ring, the equivalence is uniquely specified).
\end{lemma}

To see this, we will need the following result on manipulating colimits. For an $\infty$-category $\C$, we
denote by $\C^\vartriangleright$ the result of freely adjoining a final object to $\C$,
and by $\C^\simeq$ the maximal underlying subgroupoid of $\C$. The construction
$\C\mapsto\C^\simeq$ is right adjoint to the inclusion $\Scal\simeq\mathcal{G}rp_\infty\subseteq
\mathcal{C}at_\infty$ of $\infty$-groupoids into all $\infty$-categories.

\begin{prop}\label{prop:compute_colim} Let $\C$ be an $\infty$-category and $F\colon \C^\vartriangleright\to \Scal$
the functor defined by $F(c)=\left( (\C^\vartriangleright)_{/c}
\right)^{\simeq}$. Then the canonical map of anima $\colim_\C F \to\colim_{\C^\vartriangleright}F$ is equivalent to the inclusion
$\C^\simeq\subseteq\left( \C^\vartriangleright\right)^\simeq$.
\end{prop}

\begin{proof} The closely related functor $F'\colon \C^\vartriangleright\to
\mathcal{C}at_\infty$ defined by $F'(c):=\left(\C^\vartriangleright\right)_{/c}$
classifies the cocartesian codomain fibration $cd\colon \Fun(\Delta^1,\C^\vartriangleright)\longrightarrow \C^\vartriangleright$ given by evaluation on 1 \cite[Cor.~2.4.7.12]{Lur09}.
It follows that $F=(-)^\simeq\circ F'$ classifies the left fibration $cd'\colon\Fun(\Delta^1,\C^\vartriangleright)^{\mathrm{left}}\longrightarrow \C^\vartriangleright$ obtain by passing 
from $\Fun(\Delta^1,\C^\vartriangleright)$ to the sub-simplicial set 
$\Fun(\Delta^1,\C^\vartriangleright)^{\mathrm{left}}\subseteq \Fun(\Delta^1,\C^\vartriangleright)$
consisting of all simplices all of whose edges are $cd$-cocartesian.
Informally then, the objects of $\Fun(\Delta^1,\C^\vartriangleright)^{\mathrm{left}}$ are the morphisms in $\C^\vartriangleright$, and the morphisms are the commuting squares in which the map between sources is an equivalence.

We now observe that evaluation {\em at zero}, $ez\colon\Fun(\Delta^1,\C)^{\mathrm{left}}
\to \C^\simeq$, is a Cartesian fibration which has all fibers contractible (because each of them has an initial object). In particular, $ez$ is a weak equivalence, and an
inverse equivalence is provided by sending objects to identity morphisms. We have thus seen that $\colim_{\C} F\simeq\C^\simeq$. 

Furthermore, the canonical map $\colim_{\C} F\to \colim_{\C^\vartriangleright} F$ is equivalent to the obvious map
$\Fun(\Delta^1,\C)^{\mathrm{left}}\to \Fun(\Delta^1,\C^\vartriangleright)^{\mathrm{left}}$, the target of 
which is equivalent to the fiber over the cone point, namely $(\C^\vartriangleright)^\simeq$. One checks that this identifies the canonical map with the inclusion $\C^\simeq\subseteq\left( \C^\vartriangleright\right)^\simeq$, as claimed.
\end{proof}

\begin{proof}[Proof of \Cref{lem:fixed_of_tilde_EP}]
Applying \Cref{prop:compute_colim} with $\C=\OGP$ the category of
orbits with proper isotropy (hence $\C^\vartriangleright=\OG$), we obtain a cofiber sequence
in anima
\[ \colim_{G/H\in\OGP}\left( \OG/{(G/H)}\right)^{\simeq}\simeq \OG_\Pcal^\simeq\hookrightarrow
\OG^\simeq\longrightarrow * \sqcup +\]
where the final map sends all orbits with proper isotropy to $+$, and sends $G/G$ to $*$.
We can consider this as a cofiber sequence in pointed anima $\Scal_\bullet$
of the form
\[ \colim_{G/H\in\OGP}\left( \OG/{(G/H)}\right)^{\simeq}_+\simeq \OG_{\Pcal,+}^\simeq\longrightarrow
\OG^\simeq_+\longrightarrow S^0=*_+\]
where the final map sends all orbits with proper isotropy to the base-point $+$,
and sends $G/G$ to $*$.
Applying the free commutative monoid functor $\mathbb{P}\colon \Scal_\bullet\to\CMon(\Scal)$
yields a cofiber sequence in $\CMon(\Scal$):
\begin{equation}\label{eq:cofiber}
 \colim_{G/H\in\OGP}\left( \Fin_G/{(G/H)}\right)^{\simeq}\longrightarrow
\Fin_G^\simeq\longrightarrow \Fin^\simeq
\end{equation}
in which the final map is identified with taking $G$-fixed points.
To see this, observe that $\OG_{/(G/H)}\simeq\mathcal{O}(H)$ and that 
$\mathbb{P}(\mathcal{O}(H)_+^\simeq)\simeq\Fin_H^\simeq$, as can be checked most
easily using the general formula $\mathbb{P}(Z)=\bigvee_{n\ge 0}\left( Z^{\times
n}\times_{\Sigma_n} E\Sigma_{n_+}\right)$.

We denote by $(-)^+$ the group completion on $\CMon(\Scal)$,
and observe that
\[ \Omega^\infty\left(\Sigma^\infty_\M(G/H_+)(G/G)\right) =
\hom_{\Span(\Fin_G)}(G/G,G/H)^+\simeq
\left( \Fin_G/(G/H)\right)^{\simeq,+}.\]
We thus see that the delooping of the group completion of the cofiber sequence \eqref{eq:cofiber} is a cofiber sequence 
in $\Sp$ of the form
\[ \left(\Sigma_\M^\infty(E\mathcal{P})\right)(G/G)\longrightarrow
\Sigma_\M^\infty(S^0)(G/G)\longrightarrow
\Sigma_\M^\infty(\widetilde{E\mathcal{P}})(G/G)\simeq \mathbb{S},\]
 using the Barratt--Priddy--Quillen theorem that $\Omega^\infty(\mathbb{S})\simeq\Fin^{\simeq,+}$. 
\end{proof}

Next, we will need to discuss restriction for Mackey functors. 

\begin{construction}[Restriction for Mackey functors]
Let $H \subseteq G$ be a subgroup. 
\begin{enumerate}
\item We have a symmetric monoidal and coproduct preserving functor
\[ \mathrm{Res}^G_H\colon \Span(\Fin_G) \longrightarrow \Span(\Fin_H).  \]
This sends a $G$-set $U$ to the 
underlying $H$-set of $U$, and behaves accordingly on correspondences (cf. \cite[App.~C.3]{BH17}).
\item
We also have a functor
\[ G \times_H (-)\colon \Span(\Fin_H) \longrightarrow \Span(\Fin_G),  \]
which takes a $H$-set $T$ to the $G$-set $G \times_H T$, and behaves
analogously on correspondences. 
Note that the construction $ T \mapsto G \times_H T$ 
on finite $H$-sets preserves fiber products.  
\end{enumerate}
\end{construction}

\begin{prop}
\label{biadjointburn}
Both the functors $(\mathrm{Res}^G_H(-), G \times_H(-) ): \Span(\Fin_G)
\rightleftarrows \Span(\Fin_H)$ and the functors $(G \times_H (-) ,\mathrm{Res}^G_H(-) ): \Span(\Fin_H)
\rightleftarrows \Span(\Fin_G)$ are biadjoint.
\end{prop} 
\begin{proof} 
If $S$ is a finite $H$-set and $T$ is a finite $G$-set, then we have an
equivalence of categories
\[ \left(\mathrm{Fin}_G\right)_{/ (G \times_H S ) \times T} \stackrel{\simeq}{\longrightarrow}
\left(\mathrm{Fin}_H \right)_{/S \times \mathrm{Res}^G_H(T)},
\]
given by pulling back along the $H$-map $S \times T \to (G \times_H S  )
\times T$. 
Using that $\map_{\Span(\Fin_G)}(X,Y)\simeq\left( \Fin_G\right)_{/X\times Y}^\simeq$, the
result follows. See also  \cite[App. C.3]{BH17} for a more general treatment of Span$(-)$
as an $(\infty,2)$-functor.
\end{proof} 

We now define restriction for Mackey functors, essentially by left Kan extension. Namely, 
we define the symmetric monoidal, cocontinuous functor $\mathrm{Res}^G_{H,\M}
\colon \Mack_G\to \Mack_H$ to be $\mathrm{Res}^G_{H,\M}:=\Pcal_\Sigma(\mathrm{Res}^G_H)\otimes\id_{\Sp_G}$.

As an example, note that for $F\in\Mack_G$ and subgroups $H'\subseteq H\subseteq G$ we have
\[ \Res_{H,\M}^G(F)(H/H')\simeq F(G/H').\]
To see this, since both sides are colimit preserving functors of $F$, it suffices to 
check the case when $F$ is the suspension of an orbit, and then the claim is immediate from the second adjunction in
\Cref{biadjointburn}. 

\begin{prop}\label{prop:isotropy_sep} Assume $F\in\Mack_G$ is such 
that for all proper subgroups $H\subseteq G$ we have 
$\Res_{H,\M}^G(F)\simeq *$. Then the canonical map
\[ F(G/G)\longrightarrow\left(\Sigma_\M^\infty(\widetilde{E\Pcal})\otimes F)\right)(G/G)\]
is an equivalence.
\end{prop}
\begin{proof}
First observe that for all subgroups $H\subseteq G$, the suspension
$\Sigma_\M^\infty(G/H_+)\in\Mack_G$ is self-dual. 
It then follows that for all proper subgroups $H\subseteq G$,
the spectrum \[(F\otimes\Sigma_\M^\infty(G/H_+))(G/G)\simeq F(G/H)\simeq
\Res^G_{H,\M}(F)(H/H)\simeq *\]
is contractible, and hence that $\left( F\otimes\Sigma_\M^\infty(E\Pcal_+)\right)(G/G)\simeq *$.
The result follows. 
\end{proof}

We next introduce geometric fixed points in the Mackey context.
The fixed point functor $(-)^G:\Fin_G\longrightarrow\Fin$ commutes 
with pullbacks and hence induces a functor on span categories.
This functor preserves finite coproducts and the cartesian product,
hence the functor
\[ \Phi^G_\M:=\Psigma(\Span((-)^G))\otimes\id_\Sp:
\Mack_G\longrightarrow\Psigma(\Span(\Fin))\otimes\Sp\simeq \Sp\]
commutes with all colimits and is symmetric
monoidal. By construction, it takes the expected values on orbits, namely
$\Phi^G_\M(\Sigma_\M^\infty(G/H_+))$ is contractible for a proper subgroup
$H\subseteq G$, and equivalent to $\mathbb{S}$ for $H=G$.
In fact, more generally, for each $X\in \Scal_{G,\bullet}$ we have
\[ \Phi^G_\M(\Sigma_\M^\infty(X))\simeq \Sigma^\infty(X(G/G)).\]

For a subgroup $H\subseteq G$, we denote 
$\Phi^{G,H}_\M:=\Phi^H_\M\circ\Res^G_{H,\M}$.

We will need to know that our geometric fixed points are given by the familiar contruction:

\begin{prop}\label{prop:geo_fixed_formula} There is an equivalence\footnote{The
equivalence is constructed in the proof, we will only need an abstract
equivalence.} 
of functors $\Phi^G_\M(-) \simeq
\left(\Sigma_\M^\infty(\widetilde{E\Pcal})\otimes (-)\right)(G/G)$.
\end{prop}

\begin{proof} 
First, we have a natural transformation 
for $F \in \mathrm{Mack}_G$ given by 
$F(G/G) \to \Phi^G_{\mathcal{M}}(F)$, since
$F(G/G)$ is corepresented by the unit and 
$\Phi^G_{\mathcal{M}}(-)$ is symmetric monoidal by construction.
Since the target is unaffected by replacing $F$ by $\Sigma^\infty_{\mathcal{M}}(
\widetilde{E \mathcal{P}} \otimes F)$, we obtain a map 
$\left( \Sigma^\infty_{\mathcal{M}} (\widetilde{E\mathcal{P}} \otimes (-))
\right)(G/G) \to \Phi^G_{\mathcal{M}}(-)$. 
This map is an equivalence  on all orbits, because for a subgroup $H\subseteq G$ we can compute that 
\[\left(\Sigma_\M^\infty(\widetilde{E\Pcal})\otimes\Sigma_\M^\infty(G/H_+)\right)(G/G)\simeq
\left(\Sigma_\M^\infty\left(\widetilde{E\Pcal}\otimes (G/H_+) \right)\right)(G/G)\]
is contractible if $H$ is proper, and is $\mathbb{S}$ if $H=G$ by \Cref{lem:fixed_of_tilde_EP}.
Therefore, the result follows since both functors preserve colimits. 
\end{proof}

This allows to easily establish the basic properties of geometric fixed points
in the Mackey context:

\begin{prop}\label{prop:conservative} The family $\{  \Phi_\M^{G,H}\}_{H\subseteq G}$ of symmetric
monoidal left adjoints is jointly conservative.
\end{prop}

\begin{proof}
Assume $\Phi^{G,H}_\M(F)\simeq *$ for all $H\subseteq G$, and we need to see that $F\simeq *$.

This is clear for trivial $G$, and we argue by induction on the 
group order in general. We can thus assume that $\Res^G_{H,\M}(F)
\simeq *$ for all proper subgroups $H\subseteq G$. In particular then, for all proper subgroups $H\subseteq G$ we know that
\[ F(G/H)=\Res^G_{H,\M}(F)(H/H)=*\]
is contractible, and need to see that $F(G/G)$ is as well.
But combining \Cref{prop:isotropy_sep} and \Cref{prop:geo_fixed_formula}, we see that $F(G/G)\simeq\Phi^G_\M(F)=\Phi^{G,G}_\M(F)$, and this is contractible by assumption.
\end{proof}

This finally lets us check that suspension for Mackey functors inverts all
representation spheres.

\begin{prop}\label{prop:invert_sphere}
For every representation $V$ of $G$, $\Sigma_\M^\infty(S^V)\in
\Mack_G$ is invertible.
\end{prop}

\begin{proof}
We first note that $\Sigma_\M^\infty(S^V)\in\Mack_G$ is at least dualizable. Since $\Mack_G$ is stable, the dualizable objects are stable under finite colimits, and $S^V$ is a finite colimits of orbits. It thus suffies to remark that the orbits are dualizable (in fact, self-dual) already in $\Span(\Fin_G)$.
Once we know $\Sigma_\M^\infty(S^V)\in\Mack_G$ is dualizable, it will be invertible if and only if it becomes so 
after applying any family of jointly conservative symmetric monoidal functors. By \Cref{prop:conservative} it will thus suffice to see that for every subgroup $H\subseteq G$, the spectrum $\Phi^{G,H}_\M(\Sigma_\M^\infty(S^V))$ is invertible, but this follows from a direct computation:
\[\Phi^{G,H}_\M(\Sigma_\M^\infty(S^V))=\Phi^H_\M(\Res^G_{H,\M}(\Sigma_\M^\infty(S^V)))\simeq\Sigma^\infty((S^V)^H)\simeq
S^{\dim(V^H)}.\]
\end{proof}

We can now complete the proof of our main result.

\begin{proof}[Proof of \Cref{thm:orth_mackey}]
\Cref{thm:folklore} and \Cref{prop:invert_sphere} show that there is a unique symmetric monoidal left
adjoint $L:\Sp_G\longrightarrow\Mack_G$ such that $L\circ\Sigma_G^\infty\simeq\Sigma_\M^\infty$.
It remains to see that $L$ is an equivalence. Denote by $R$ the right adjoint of $L$. Since both $\Sp_G$
and $\Mack_G$ are generated under colimits by dualizable objects (namely the suspensions of orbits), it
follows from \cite[Thm. 1.3]{balmer-grothendieck} that $R$ admits itself a right adjoint, hence preserves colimits, and that the
adjunction $(L,R)$ satisfies a projection formula. Furthermore, $R$ is conservative because the image of its left adjoint $L$ contains a set of generators. We can thus apply \cite[Prop. 5.29]{MNN17} to conclude that the
adjunction $(L,R)$ induces an adjoint equivalence
\[ \mathrm{Mod}_{\Sp_G}(R(\mathrm{1}_{\Mack_G}))\simeq\Mack_G,\]
and it remains to see that the counit of the adjunction
\begin{equation}  \mathrm{1}_{\Sp_G}\longrightarrow
R(L(\mathrm{1}_{\Sp_G}))\simeq R(\mathrm{1}_{\Mack_G})
\label{counitadj} \end{equation}
is an equivalence. 

Now we use induction on the group order. Given a proper subgroup $H \subsetneq
G$, we have a commutative diagram in
$\mathrm{CAlg}(\Pr^L)$,
\[ \xymatrix{
\Sp_G \ar[d]  \ar[r] &  \Mack_G \ar[d] \\
\Sp_H \ar[r] &  \Mack_H,
}\]
by the universal property of $\Sp_G$. 
The inductive hypothesis gives that the bottom horizontal arrow is an
equivalence. 
This implies that if $X \in \Sp_G$, then 
$\hom_{\Sp_G}( G/H_+, X) = \hom_{\Mack_G}(G/H_+, X)$ since both sides are
calculated as maps out of the unit in $\Sp_H$ (resp.~$\Mack_H$). 
In particular, this implies that  
\eqref{counitadj}
restricts to an equivalence after restriction to proper subgroups; therefore, it suffices to see that
$\Phi^G$ (i.e., geometric 
fixed points for orthogonal spectra) turns this map into an equivalence. Since
$\Phi^G(\mathrm{1}_{\Sp_G})=\mathbb{S}$ and we are looking at a map of
commutative algebras, it suffices in fact to see that there is an equivalence of
spectra $\Phi^G(R(\mathrm{1}_{\Mack_G}))\simeq \mathbb{S}$.
This follows from the following computation:
\[ \Phi^G(R(\mathrm{1}_{\Mack_G}))\simeq\left(
\Sigma_G^\infty(\widetilde{E\Pcal}) \otimes
R(\mathrm{1}_{\Mack_G}))\right)^G\simeq\left( R\left[
L(\Sigma_G^\infty(\widetilde{E\Pcal}))\otimes \mathrm{1}_{\Mack_G})\right] \right)^G\simeq\]
\[\simeq L(\Sigma_G^\infty(\widetilde{E \Pcal}))(G/G) \simeq
(\Sigma_\M^\infty(\widetilde{E\Pcal}))(G/G)\simeq \mathbb{S}.\]
This computation used in turn: The definition of $\Phi^G$, the projection formula for $(L,R)$, 
the fact that $\left( R(-)\right)^G\simeq (-)(G/G)$ (by adjointness of $L$ and $R$), the fact that $L\circ\Sigma_G^\infty\simeq
\Sigma_\M^\infty$, and finally \Cref{lem:fixed_of_tilde_EP}.
\end{proof}

As promised earlier, our account yields the following proof of the equivariant
Barratt--Priddy--Quillen theorem (originally due to \cite{GM11}), which by-passes any loop-space theory (but uses the non-equivariant version).

\begin{cor}
For a finite group $G$, there is an equivalence in $\CMon(\Scal)$
\[ \colim_V \Omega^V S^V \simeq \left( \Fin_G\right)^{\simeq,+},\]
where the colimit is taken along any cofinal system of representations of $G$,
and $(-)^+$ denotes group completion.
\end{cor}

For the proof, one simply computes the endomorphism anima of the unit of both
$\Sp_G$ and $\Mack_G$ from the definition, and compares the result.

%% file: DescentGroupAct_revised2.bbl
\newcommand{\etalchar}[1]{$^{#1}$}
\begin{thebibliography}{AKAC{\etalchar{+}}22}

\bibitem[ABG10]{ABG10}
Matthew Ando, Andrew~J. Blumberg, and David Gepner.
\newblock Twists of {$K$}-theory and {TMF}.
\newblock In {\em Superstrings, geometry, topology, and {$C^\ast$}-algebras},
  volume~81 of {\em Proc. Sympos. Pure Math.}, pages 27--63. Amer. Math. Soc.,
  Providence, RI, 2010.

\bibitem[AKAC{\etalchar{+}}22]{AKACHR22}
Gabriel Angelini-Knoll, Christian Ausoni, Dominic~Leon Culver, Eva H{\"o}ning,
  and John Rognes.
\newblock Algebraic {K}-theory of elliptic cohomology.
\newblock {\em arXiv preprint arXiv:2204.05890}, 2022.

\bibitem[AKS20]{AKS20}
Gabriel Angelini-Knoll and Andrew Salch.
\newblock Commuting unbounded homotopy limits with {M}orava {K}-theory.
\newblock {\em arXiv preprint arXiv:2003.03510}, 2020.

\bibitem[AR02]{AR02}
Christian Ausoni and John Rognes.
\newblock Algebraic {$K$}-theory of topological {$K$}-theory.
\newblock {\em Acta Math.}, 188(1):1--39, 2002.

\bibitem[AR08]{AR}
Christian Ausoni and John Rognes.
\newblock The chromatic red-shift in algebraic {K}-theory.
\newblock In {\em Guido's book of conjectures}, volume~54 of {\em Enseign.
  {M}ath.}, pages 9--11. 2008.

\bibitem[AS69]{AS69}
M.~F. Atiyah and G.~B. Segal.
\newblock Equivariant {$K$}-theory and completion.
\newblock {\em J. Differential Geometry}, 3:1--18, 1969.

\bibitem[Ati61]{At61}
M.~F. Atiyah.
\newblock Characters and cohomology of finite groups.
\newblock {\em Inst. Hautes \'{E}tudes Sci. Publ. Math.}, (9):23--64, 1961.

\bibitem[Aus10]{Ausoni10}
Christian Ausoni.
\newblock On the algebraic {$K$}-theory of the complex {$K$}-theory spectrum.
\newblock {\em Invent. Math.}, 180(3):611--668, 2010.

\bibitem[Ban17]{Banerjee17}
Romie Banerjee.
\newblock Galois descent for real spectra.
\newblock {\em J. Homotopy Relat. Struct.}, 12(2):273--297, 2017.

\bibitem[Bar16]{Bar16}
Clark Barwick.
\newblock On the algebraic {$K$}-theory of higher categories.
\newblock {\em J. Topol.}, 9(1):245--347, 2016.

\bibitem[Bar17]{Bar17}
Clark Barwick.
\newblock Spectral {M}ackey functors and equivariant algebraic {$K$}-theory
  ({I}).
\newblock {\em Adv. Math.}, 304:646--727, 2017.

\bibitem[BCM20]{BCM20}
Bhargav Bhatt, Dustin Clausen, and Akhil Mathew.
\newblock Remarks on {$K (1)$}-local {$K$}-theory.
\newblock {\em Selecta Math. (N.S.)}, 26(3):Paper No. 39, 16, 2020.

\bibitem[BCSY22]{BCSY}
Tobias Barthel, Shachar Carmeli, Tomer~M. Schlank, and Lior Yanovski.
\newblock The chromatic {F}ourier transform.
\newblock {\em arXiv preprint arXiv:2210.12822}, 2022.

\bibitem[BDS16]{balmer-grothendieck}
Paul Balmer, Ivo Dell'Ambrogio, and Beren Sanders.
\newblock Grothendieck-{N}eeman duality and the {W}irthm\"uller isomorphism.
\newblock {\em Compos. Math.}, 152(8):1740--1776, 2016.

\bibitem[BGS20]{BGS20}
Clark Barwick, Saul Glasman, and Jay Shah.
\newblock Spectral {M}ackey functors and equivariant algebraic {$K$}-theory,
  {II}.
\newblock {\em Tunis. J. Math.}, 2(1):97--146, 2020.

\bibitem[BGT13]{BGT13}
Andrew~J. Blumberg, David Gepner, and Gon{\c{c}}alo Tabuada.
\newblock A universal characterization of higher algebraic {$K$}-theory.
\newblock {\em Geom. Topol.}, 17(2):733--838, 2013.

\bibitem[BH21]{BH17}
Tom Bachmann and Marc Hoyois.
\newblock Norms in motivic homotopy theory.
\newblock {\em Ast\'{e}risque}, (425):207, 2021.

\bibitem[Bor61]{Borel61}
Armand Borel.
\newblock Sous-groupes commutatifs et torsion des groupes de {L}ie compacts
  connexes.
\newblock {\em Tohoku Math. J. (2)}, 13:216--240, 1961.

\bibitem[Bou01]{Bou01}
A.~K. Bousfield.
\newblock On the telescopic homotopy theory of spaces.
\newblock {\em Trans. Amer. Math. Soc.}, 353(6):2391--2426, 2001.

\bibitem[BSY22]{BSY22}
Robert Burklund, Tomer~M. Schlank, and Allen Yuan.
\newblock The chromatic {N}ullstellensatz.
\newblock {\em arXiv preprint arXiv:2207.09929}, 2022.

\bibitem[CM17]{CM17}
Dustin Clausen and Akhil Mathew.
\newblock A short proof of telescopic {T}ate vanishing.
\newblock {\em Proc. Amer. Math. Soc.}, 145(12):5413--5417, 2017.

\bibitem[CM21]{CM19}
Dustin Clausen and Akhil Mathew.
\newblock Hyperdescent and \'{e}tale {$K$}-theory.
\newblock {\em Invent. Math.}, 225(3):981--1076, 2021.

\bibitem[CMNN20]{CMNN}
Dustin Clausen, Akhil Mathew, Niko Naumann, and Justin Noel.
\newblock Descent in algebraic {$K$}-theory and a conjecture of
  {A}usoni--{R}ognes.
\newblock {\em J. Eur. Math. Soc. (JEMS)}, 22(4):1149--1200, 2020.

\bibitem[Cra10]{Cra10}
James Cranch.
\newblock Algebraic theories and (infinity,1)-categories.
\newblock {\em arXiv preprint arXiv:1011.3243}, 2010.

\bibitem[DGM13]{DGM13}
Bj\o rn~Ian Dundas, Thomas~G. Goodwillie, and Randy McCarthy.
\newblock {\em The local structure of algebraic {K}-theory}, volume~18 of {\em
  Algebra and Applications}.
\newblock Springer-Verlag London, Ltd., London, 2013.

\bibitem[DH04]{DH04}
Ethan~S. Devinatz and Michael~J. Hopkins.
\newblock Homotopy fixed point spectra for closed subgroups of the {M}orava
  stabilizer groups.
\newblock {\em Topology}, 43(1):1--47, 2004.

\bibitem[GGN15]{GGN15}
David Gepner, Moritz Groth, and Thomas Nikolaus.
\newblock Universality of multiplicative infinite loop space machines.
\newblock {\em Algebr. Geom. Topol.}, 15(6):3107--3153, 2015.

\bibitem[Gla16]{Gla16}
Saul Glasman.
\newblock Day convolution for {$\infty$}-categories.
\newblock {\em Math. Res. Lett.}, 23(5):1369--1385, 2016.

\bibitem[GM11]{GM11}
Bertrand Guillou and J.~P. May.
\newblock Models of {$G$}-spectra as presheaves of spectra.
\newblock {\em arXiv preprint arXiv:1110.3571}, 2011.

\bibitem[GM20]{GM20}
David Gepner and Lennart Meier.
\newblock On equivariant topological modular forms.
\newblock {\em arXiv preprint arXiv:2004.10254}, 2020.

\bibitem[GS96]{GS96}
J.~P.~C. Greenlees and Hal Sadofsky.
\newblock The {T}ate spectrum of {$v_n$}-periodic complex oriented theories.
\newblock {\em Math. Z.}, 222(3):391--405, 1996.

\bibitem[Hah16]{Hahn16}
Jeremy Hahn.
\newblock On the {B}ousfield classes of {$H_\infty$}-ring spectra.
\newblock {\em arXiv preprint arXiv:1612.04386}, 2016.

\bibitem[HKR00]{HKR00}
Michael~J. Hopkins, Nicholas~J. Kuhn, and Douglas~C. Ravenel.
\newblock Generalized group characters and complex oriented cohomology
  theories.
\newblock {\em J. Amer. Math. Soc.}, 13(3):553--594, 2000.

\bibitem[HN19]{HN19}
Lars Hesselholt and Thomas Nikolaus.
\newblock Topological cyclic homology.
\newblock In Haynes Miller, editor, {\em Handbook of homotopy theory}. CRC
  Press/Chapman and Hall, 2019.

\bibitem[HRW22]{HRW22}
Jeremy Hahn, Arpon Raksit, and Dylan Wilson.
\newblock A motivic filtration on the topological cyclic homology of
  commutative ring spectra.
\newblock {\em arXiv preprint arXiv:2206.11208}, 2022.

\bibitem[HS96]{HS96}
Mark Hovey and Hal Sadofsky.
\newblock Tate cohomology lowers chromatic {B}ousfield classes.
\newblock {\em Proc. Amer. Math. Soc.}, 124(11):3579--3585, 1996.

\bibitem[HSS17]{HSS17}
Marc Hoyois, Sarah Scherotzke, and Nicol\`o Sibilla.
\newblock Higher traces, noncommutative motives, and the categorified {C}hern
  character.
\newblock {\em Adv. Math.}, 309:97--154, 2017.

\bibitem[HW22]{HW20}
Jeremy Hahn and Dylan Wilson.
\newblock Redshift and multiplication for truncated {B}rown--{P}eterson
  spectra.
\newblock {\em Ann. of Math. (2)}, 196(3):1277--1351, 2022.

\bibitem[Ill78]{Ill78}
S\"{o}ren Illman.
\newblock Smooth equivariant triangulations of {$G$}-manifolds for {$G$} a
  finite group.
\newblock {\em Math. Ann.}, 233(3):199--220, 1978.

\bibitem[KP78]{KP78}
Daniel~S. Kahn and Stewart~B. Priddy.
\newblock The transfer and stable homotopy theory.
\newblock {\em Math. Proc. Cambridge Philos. Soc.}, 83(1):103--111, 1978.

\bibitem[Kuh89]{Kuh89}
Nicholas~J. Kuhn.
\newblock Morava {$K$}-theories and infinite loop spaces.
\newblock In {\em Algebraic topology ({A}rcata, {CA}, 1986)}, volume 1370 of
  {\em Lecture Notes in Math.}, pages 243--257. Springer, Berlin, 1989.

\bibitem[Kuh04]{Kuh04}
Nicholas~J. Kuhn.
\newblock Tate cohomology and periodic localization of polynomial functors.
\newblock {\em Invent. Math.}, 157(2):345--370, 2004.

\bibitem[LMMT20]{LMMT20}
Markus Land, Akhil Mathew, Lennart Meier, and Georg Tamme.
\newblock Purity in chromatically localized algebraic {$K$}-theory.
\newblock {\em arXiv preprint arXiv:2001.10425}, 2020.

\bibitem[LRRV19]{LRRV19}
Wolfgang L\"{u}ck, Holger Reich, John Rognes, and Marco Varisco.
\newblock Assembly maps for topological cyclic homology of group algebras.
\newblock {\em J. Reine Angew. Math.}, 755:247--277, 2019.

\bibitem[LT19]{LT19}
Markus Land and Georg Tamme.
\newblock On the {$K$}-theory of pullbacks.
\newblock {\em Ann. of Math. (2)}, 190(3):877--930, 2019.

\bibitem[Lur09]{Lur09}
Jacob Lurie.
\newblock {\em Higher topos theory}, volume 170 of {\em Annals of Mathematics
  Studies}.
\newblock Princeton University Press, Princeton, NJ, 2009.

\bibitem[Lur17]{Lur17}
Jacob Lurie.
\newblock {\em Higher algebra}.
\newblock 2017.
\newblock Available at \url{http://www.math.ias.edu/~lurie/papers/HA.pdf}.

\bibitem[Mah81]{Ma81}
Mark Mahowald.
\newblock {$b{\rm o}$}-resolutions.
\newblock {\em Pacific J. Math.}, 92(2):365--383, 1981.

\bibitem[Mal17]{Mal17}
Cary Malkiewich.
\newblock Coassembly and the {$K$}-theory of finite groups.
\newblock {\em Adv. Math.}, 307:100--146, 2017.

\bibitem[Mat16]{MGal}
Akhil Mathew.
\newblock The {G}alois group of a stable homotopy theory.
\newblock {\em Adv. Math.}, 291:403--541, 2016.

\bibitem[Mat21]{MTR20}
Akhil Mathew.
\newblock On {$K(1)$}-local {TR}.
\newblock {\em Compos. Math.}, 157(5):1079--1119, 2021.

\bibitem[Mer17]{Me17}
Mona Merling.
\newblock Equivariant algebraic {K}-theory of {$G$}-rings.
\newblock {\em Math. Z.}, 285(3-4):1205--1248, 2017.

\bibitem[Mil81]{Mi81}
Haynes~R. Miller.
\newblock On relations between {A}dams spectral sequences, with an application
  to the stable homotopy of a {M}oore space.
\newblock {\em J. Pure Appl. Algebra}, 20(3):287--312, 1981.

\bibitem[Mil92]{Miller92}
Haynes Miller.
\newblock Finite localizations.
\newblock {\em Bol. Soc. Mat. Mexicana (2)}, 37(1-2):383--389, 1992.
\newblock Papers in honor of Jos{\'e} Adem (Spanish).

\bibitem[Mit90]{Mitchell90}
Stephen~A. Mitchell.
\newblock The {M}orava {$K$}-theory of algebraic {$K$}-theory spectra.
\newblock {\em $K$-Theory}, 3(6):607--626, 1990.

\bibitem[MM19]{MM19}
Cary Malkiewich and Mona Merling.
\newblock Equivariant {$A$}-theory.
\newblock {\em Doc. Math.}, 24:815--855, 2019.

\bibitem[MNN15]{MNN_nilpotence}
Akhil Mathew, Niko Naumann, and Justin Noel.
\newblock On a nilpotence conjecture of {J}. {P}. {M}ay.
\newblock {\em J. Topol.}, 8(4):917--932, 2015.

\bibitem[MNN17]{MNN17}
Akhil Mathew, Niko Naumann, and Justin Noel.
\newblock Nilpotence and descent in equivariant stable homotopy theory.
\newblock {\em Adv. Math.}, 305:994--1084, 2017.

\bibitem[MNN19]{MNN19}
Akhil Mathew, Niko Naumann, and Justin Noel.
\newblock Derived induction and restriction theory.
\newblock {\em Geom. Topol.}, 23(2):541--636, 2019.

\bibitem[Nar16]{Na16}
Denis Nardin.
\newblock Parametrized higher category theory and higher algebra: {E}xpos{\'e}
  {IV} -- stability with respect to an orbital $\infty$-category.
\newblock {\em arXiv preprint arXiv:1608.07704}, 2016.

\bibitem[NS18]{NS18}
Thomas Nikolaus and Peter Scholze.
\newblock On topological cyclic homology.
\newblock {\em Acta Math.}, 221(2):203--409, 2018.

\bibitem[R{\O}06]{RO06}
Andreas Rosenschon and P.~A. {\O}stv{\ae}r.
\newblock Descent for {$K$}-theories.
\newblock {\em J. Pure Appl. Algebra}, 206(1-2):141--152, 2006.

\bibitem[Rob15]{robalo}
Marco Robalo.
\newblock {$K$}-theory and the bridge from motives to noncommutative motives.
\newblock {\em Adv. Math.}, 269:399--550, 2015.

\bibitem[Rog08]{Rognes08}
John Rognes.
\newblock Galois extensions of structured ring spectra. {S}tably dualizable
  groups.
\newblock {\em Mem. Amer. Math. Soc.}, 192(898):viii+137, 2008.

\bibitem[RV18]{RV18}
Holger Reich and Marco Varisco.
\newblock Algebraic {$K$}-theory, assembly maps, controlled algebra, and trace
  methods.
\newblock In {\em Space---time---matter}, pages 1--50. De Gruyter, Berlin,
  2018.

\bibitem[Seg68]{Segal68}
Graeme Segal.
\newblock Equivariant {$K$}-theory.
\newblock {\em Inst. Hautes \'{E}tudes Sci. Publ. Math.}, (34):129--151, 1968.

\bibitem[Sus84]{Suslin84}
Andrei~A. Suslin.
\newblock On the {$K$}-theory of local fields.
\newblock In {\em Proceedings of the {L}uminy conference on algebraic
  {$K$}-theory ({L}uminy, 1983)}, volume~34, pages 301--318, 1984.

\bibitem[Swa60]{Swa60}
Richard~G. Swan.
\newblock Induced representations and projective modules.
\newblock {\em Ann. of Math. (2)}, 71:552--578, 1960.

\bibitem[Swa70]{Swa70}
Richard~G. Swan.
\newblock {\em {$K$}-theory of finite groups and orders}.
\newblock Lecture Notes in Mathematics, Vol. 149. Springer-Verlag, Berlin-New
  York, 1970.

\bibitem[Tho85]{Tho85}
R.~W. Thomason.
\newblock Algebraic {$K$}-theory and \'etale cohomology.
\newblock {\em Ann. Sci. \'Ecole Norm. Sup. (4)}, 18(3):437--552, 1985.

\bibitem[Tre15]{Treumann15}
David Treumann.
\newblock Representations of finite groups on modules over {K}-theory (with an
  appendix by {A}khil {M}athew).
\newblock {\em arXiv preprint arXiv:1503.02477}, 2015.

\bibitem[TT90]{TT90}
R.~W. Thomason and Thomas Trobaugh.
\newblock Higher algebraic {$K$}-theory of schemes and of derived categories.
\newblock In {\em The {G}rothendieck {F}estschrift, {V}ol.\ {III}}, volume~88
  of {\em Progr. Math.}, pages 247--435. Birkh\"auser Boston, Boston, MA, 1990.

\bibitem[Yua21]{Yuan21}
Allen Yuan.
\newblock Examples of chromatic redshift in algebraic {$K$}-theory.
\newblock {\em arXiv preprint arXiv:2111.10837}, 2021.

\end{thebibliography}
